\documentclass[11pt,a4paper,oneside]{amsart}
\usepackage{amsfonts, amsmath, amssymb, amsthm}
\usepackage[colorlinks=true,citecolor=blue]{hyperref}
\usepackage[margin=1in]{geometry}
\usepackage[utopia]{mathdesign}
\usepackage[mathscr]{euscript}

\usepackage{graphicx}
\usepackage{parskip}
\usepackage{enumerate}
\usepackage{verbatim}
\usepackage{tikz}
\usepackage{algorithmic}
\usetikzlibrary{decorations,decorations.pathreplacing}

\usepackage{lineno}
\usepackage{booktabs}
\newtheorem{theorem}{Theorem}[section]
\newtheorem*{theorem*}{Theorem}
\newtheorem{lemma}[theorem]{Lemma}
\newtheorem{prop}[theorem]{Proposition}

\newtheorem{corollary}[theorem]{Corollary}

\theoremstyle{definition}
\newtheorem{define}[theorem]{Definition}
\newtheorem{remark}[theorem]{Remark}
\newtheorem{example}[theorem]{Example}

\newtheorem{algo}[theorem]{Algorithm}

\begingroup
    \makeatletter
    \@for\theoremstyle:=definition,remark,plain\do{%
        \expandafter\g@addto@macro\csname th@\theoremstyle\endcsname{%
            \addtolength\thm@preskip\parskip
            }%
        }
\endgroup

\usepackage{mathtools}
\newcommand{\defeq}{\vcentcolon=}

\newcommand{\df}[1]{{{\color{blue!50!black}\em #1}}}
\newcommand{\mb}[1]{\mathbb{#1}}

\numberwithin{equation}{section}

\renewcommand{\epsilon}{\varepsilon}
\renewcommand{\phi}{\varphi}

\newcommand{\upstart}{ .. controls +(0.5,0) and +(-0.2,-0.2) .. ++(1,0.5) }
\newcommand{\upend}{ .. controls +(0.2,0.2) and +(-0.5,0) .. ++(1,0.5)}
\newcommand{\downstart}{ .. controls +(0.5,0) and +(-0.2,0.2) .. ++(1,-0.5)}
\newcommand{\downend}{ .. controls +(0.2,-0.2) and +(-0.5,0) .. ++(1,-0.5)}
\newcommand{\downup}{ .. controls +(0.3,-0.3) and +(-0.3,-0.3) .. ++(1,0)}
\newcommand{\updown}{ .. controls +(0.3,0.3) and +(-0.3,0.3) .. ++(1,0)}

\newcommand{\ds}{\downstart}
\newcommand{\de}{\downend}
\newcommand{\us}{\upstart}
\newcommand{\ue}{\upend}
\newcommand{\ud}{\updown}
\newcommand{\du}{\downup}
\newcommand{\g}[1]{ -- ++(#1,0)} 
\newcommand{\ug}[1]{ -- ++(#1,#1)} 
\newcommand{\dg}[1]{ -- ++(#1,-#1)}

\newcommand{\ft}{\footnotesize}
\newcommand{\lu}{--++(1,1)}
\newcommand{\ls}{--++(1,0)}
\newcommand{\ld}{--++(1,-1)}

\newcommand{\tabox}{%
{\begin{tikzpicture}[x={(2.2pt,0)},y={(0,2.2pt)},baseline= -1.2pt]
\draw[line width=0.6pt] (0,0) -- (1,0) -- (1,1) -- (0,1) -- cycle;
\end{tikzpicture}}}

\title{Regular systems of paths and families of convex sets in convex position}

\author{Michael G. Dobbins}
\address{M. G. Dobbins \\ GAIA \\ Postech \\ Pohang \\ South Korea}
\email{dobbins@postech.ac.kr}

\author{Andreas F. Holmsen}
\address{A. F. Holmsen \\ Department of Mathematical Sciences \\ KAIST \\
  Daejeon \\ South Korea} 
\email{andreash@kaist.edu}

\author{Alfredo Hubard}
\address{A. Hubard \\ Laboratoire de l'Institut Gaspard Monge\\ Universit\'e Paris-Est Marne-la-Vall\'ee \\Paris \\ France}
\email{hubard@di.ens.fr}

\begin{document}

\begin{abstract} 
In this paper we show that every sufficiently large family of convex
bodies in the  plane has a large subfamily in convex position
provided that the number of common tangents of each pair of bodies is
bounded and every subfamily of size five is in convex position. (If
each pair of bodies have at most two common tangents it is enough to
assume that every triple is in convex position, and likewise, if
each pair of bodies have at most four common tangents it is enough to
assume that every quadruple is in convex position.) This confirms
a conjecture of Pach and T{\'o}th, and generalizes a theorem of
Bisztriczky and Fejes T\'{o}th. Our results on families of convex
bodies are consequences of more general Ramsey-type results about the
crossing patterns of systems of graphs of continuous functions $f :
[0,1]\to \mb{R}$. On our way towards proving the Pach-T\'{o}th conjecture we obtain a 
combinatorial characterization of such systems of graphs in which all
subsystems of equal size induce equivalent crossing patterns. These
highly organized structures are what we call {\em regular systems of
  paths} and they are natural generalizations of the notions of cups
and caps from the famous theorem of Erd\H{o}s and Szekeres. The characterization
of regular systems is combinatorial and introduces some auxiliary 
structures which may be of independent interest.
\end{abstract} 

\maketitle

\section{Introduction} Loosely speaking, Ramsey theory can be
summarized by the statement 
 ``any sufficiently large structure must contain
a highly organized substructure''. Typically, solving a Ramsey-type
problem consists of two parts. Firstly, the existence of a  ``highly
organized 
substructure'' is established, secondly a quantitative estimate of
``sufficiently large'' is given. This paper focuses on the first of
the two problems. Our results rely on a crude application of Ramsey's
theorem  and therefore we expect our bounds to be quantitatively
meaningless. On the other hand deciphering what ``highly organized''
means in the context of families of convex bodies turned out to be the
heart of the problem.

\subsection{The Erd\H{o}s--Szekeres theorem} 
Quantitative aspects of Ramsey theory were first considered in a
foundational paper by Erd\H{o}s and Szekeres. They were led to Ramsey's
theorem on their way to showing the following beautiful result
\cite{erd-sze1, erd-sze2}.  

\begin{theorem*}[Erd\H{o}s--Szekeres, 1935] For 
every positive integer $n$ there exists a minimal positive integer
$f(n)$ such that the following holds: Any set of at least $f(n)$
points in the plane such that no three are collinear, contains $n$
points which are in convex position. 
\end{theorem*}

We say that a set, or family of sets, is in \df{convex position} if no
element, or member, is contained in the convex hull of the
others. Determining the precise growth of $f(n)$ is one of the
longest-standing open problems in combinatorial geometry, and as such
it has generated a considerable amount of research. 
For every integer $n\geq 3$, Erd\H{o}s and Szekeres
\cite{erd-sze2} have constructed a set of $2^{n-2}$ points with no
three collinear points, which does not contain the vertices of any
convex $n$-gon. This example shows that the Erd\H{o}s--Szekeres
function $f(n)$ is strictly greater than $2^{n-2}$. (See \cite{morris}
for a simple description of this construction.) For $n \leq 6$ it is
known that $f(n) = 2^{n-2}+1$, and it is conjectured that this
equality holds for all $n$ \cite{erd-sze1, 
  erd-sze2, szekeres17}. For $n > 6$ the best known upper bound, due
to T{\'o}th and Valtr \cite{totval}, is $f(n) \leq \binom{2n-5}{n-2}+1
\sim \frac{4^n}{\sqrt{n}}$. Asymptotically this is the same as
the bound given by Erd\H{o}s and Szekeres in their seminal paper. For
more more information about the Erd\H{o}s--Szekeres problem, its
generalizations and related results, the reader should consult the
surveys \cite{BaraKaro,morris}.

\subsection{Families of convex bodies} Bisztriczky and Fejes
T\'{o}th \cite{biszFT1} extended the Erd\H{o}s--Szekeres theorem to
families of convex bodies in the plane. (Here a convex body means a compact
convex set.)

\begin{theorem*}[Bisztriczky--Fejes T\'{o}th, 1989]
For every positive integer $n$ there exists 
of a minimal positive integer $h_0(n)$ such that the following holds:
Any family of at least $h_0(n)$ pairwise disjoint convex bodies
in the plane such that any three are in convex position, contains
$n$ members which are in convex position. 
\end{theorem*}

The case when the convex bodies are points shows that $f(n) \leq
h_0(n)$, but bounding $h_0(n)$ is considerably more
difficult. Nevertheless, Bisztriczky and Fejes T\'{o}th conjectured
that the two functions are equal, that is, $f(n) = h_0(n)$, which is
known to hold for all $n\leq 6$ (see \cite{DHH}). The original upper
bound on $h_0(n)$ was reduced to $16^n/n$  by Pach and T\'{o}th in
\cite{PachTothBodies}, and in a subsequent paper \cite{PachToth1} the
disjointness condition was relaxed. They showed that it is sufficient
to assume that each pair of bodies are {\em non-crossing}, which means
that each pair of bodies have precisely two {\em common supporting
  tangents}.\footnote{There are several other definitions of 
  non-crossing convex bodies, but our definition has some technical
  advantages, and more importantly, does not cause any loss of generality.} 
Let $h_1(n)$ denote the corresponding Erd\H{o}s--Szekeres
function for non-crossing bodies, which obviously satisfies $h_0(n) \leq
h_1(n)$. The upper bound of $h_1(n)$ was reduced in \cite{FPSS,
  hubsuk}, while the most substantial improvement was given by the
present authors in \cite{DHH} where it was shown that $h_1(n)\leq
\binom{2n-5}{n-2}+1$. Notice that this is the same bound as T\'{o}th
and Valtr's bound for the original Erd\H{o}s--Szekeres function.  

\subsection{The Pach--T\'{o}th conjecture} 
Pach and T\'{o}th  conjectured in \cite{PachToth1} that the non-crossing
condition could be further relaxed.
However, they constructed an
infinite family of segments in which every three members are in convex
position, but no four are, indicating that additional assumptions
are needed. (See \cite{PachToth1} for details.) Here we confirm the 
conjecture of Pach and T\'{o}th.\footnote{Their original
  formulation of generalizing the non-crossing condition is by
  bounding the number of common boundary points among any two
  bodies, but it is easily seen that this is implied by bounding the
  number of common supporting tangents.}

\begin{theorem} \label{main Erd Sze} For all integers $n > k \geq 1$,
  there exists a minimal positive integer $h_k(n)$ such that the
  following holds: Any family of at least $h_k(n)$ convex bodies in
  the plane such that any two have at most $2k$ common
  supporting tangents and any $m_k$ are in convex position,
  contains $n$ members which are in convex position, where $m_1=3$,
  $m_2 = 4$, and $m_k = 5$ for all $k\geq 3$.  
\end{theorem}

Surprisingly, $m_k$ does not grow as $k$ tends to infinity. The family
of segments given by Pach and T\'{o}th shows that the bound for $m_2$
cannot be reduced, and it will follow as a simple consequence of our
analysis that for $k\geq 3$ the bound for $m_k$ cannot be reduced. In
particular, there exists arbitrarily large families of convex bodies
such that
\begin{itemize}
\item any two members have precisely {\em six} common tangents,
\item any four members are in convex position, and 
\item no five members are in convex position. 
\end{itemize}

The construction will be given in section \ref{remark4}.

\subsection{Systems of paths}
 A \df{path} is the graph of a continuous function $f: [0,1] \to
 \mb{R}$, drawn on the vertical strip $[0,1]\times
 \mb{R}$.\footnote{What we call a path is often referred to as an
   $x$-monotone curve, and our terminology is used mainly for the sake
   of brevity.}  

\begin{define}
A \df{system of paths} is a finite collection of at least three paths
which satisfy the following conditions: 

\vbox{
\begin{itemize}
\item The paths have distinct endpoints.
\item The intersection of any pair of paths is finite. 
\item No two paths are tangent; paths cross at every 
  point where they intersect.  
\item The intersection of any three paths is empty.
\end{itemize}}
\end{define}

For brevity, we refer to a systems of paths simply as {\em
  systems}. 
We say a system is \df{$k$-crossing} when each pair of paths cross 
{\em at least 1 time and at most $k$ times}.
The \df{size} of a system $S$ is the number of
paths in the system and is denoted by $|S|$. By taking a subset of at
least three paths of a system we obtain a \df{subsystem}. A path
belongs to the \df{upper envelope} of a system if at some point in
$[0,1]\times \mb{R}$ it appears {\em above} every other path of the
system. The lower envelope is defined similarly. The system is called
\df{upper convex} if all its paths appear on the upper envelope, and
\df{lower convex} if all its paths appear on the lower envelope.  
  
\begin{remark}
  A $1$-crossing system can be viewed as a {\em simple} pseudoline
  arrangement with a distinguished horizontal direction, and in fact
  every simple pseudoline arrangement can be represented in this way
  by a ``wiring diagram'' (see section 5.1 of \cite{goody} or section
  6.3 of \cite{OMS}). 
\end{remark}

Our proof of the conjecture of Pach and T\'{o}th applies a
combinatorial analysis of $k$-crossing systems to a correspondence
between upper convex systems and subfamilies in convex position. The
main result of this paper provides an answer to the following
Ramsey-type question: {\em Does every sufficiently large $k$-crossing
  system contain a large subsystem which is upper convex (or lower
  convex)?}    

\begin{theorem} \label{general cupscaps}
For any integers $k\geq 1$ and $n\geq 3$, there exists a minimal positive
integer $C_k(n)$ such that the following holds. Every $k$-crossing
system $S$ of size at least $C_k(n)$ contains a subsystem of size $n$
which is upper convex or lower convex.  
\end{theorem} 

The case $k=1$ is a dual version of the well-known {\em cups-caps}
theorem of Erd\H{o}s and Szekeres \cite{erd-sze1}. In this case the
precise value of $C_1(n)$ is known and equals $\binom{2n-4}{n-2} +
1$. It is not hard to see that $C_k(3)=3$ for all $k$, and clearly
$C_k(n) \leq C_{k+1}(n)$, but as far as we know equality may hold for
all $k$. Our bound on $C_k(n)$ is in terms of certain Ramsey numbers
and we suspect  it to be very far from the truth. 

Unfortunately, Theorem \ref{general cupscaps} does not seem to be
applicable towards the conjecture of Pach and T\'{o}th. For this we need a
``one-sided'' version guaranteeing the existence of large {\em upper
  convex} subsystem, which requires the hypothesis to be strengthened
accordingly.  

\begin{theorem} \label{general ES}
For any integers $k\geq 1$ and $n\geq 3$, there exists a minimal positive
integer $U_k(n)$ such that the following holds.  
For any $k$-crossing system $S$ of size at least $U_k(n)$:
\begin{enumerate}
\item If $k\leq 2$ and every subsystem of size $3$ is upper convex,
  then $S$ contains an upper convex subsystem of size $n$. 
\item If $k\leq 4$ and every subsystem of size $4$ is upper convex,
  then $S$ 
  contains an upper convex subsystem of size $n$.
\item If $k\geq 5$ and every subsystem of size $5$ is upper convex,
  then $S$ 
  contains an upper convex subsystem of size $n$.
\end{enumerate}
\end{theorem}

Our bound on $U_k(n)$ is the same as on $C_k(n)$, and again we
suspect it to be far from the truth. 

We are in position to prove the Pach--T\'{o}th conjecture. 

\begin{proof}[Proof of theorem \ref{main Erd Sze}] 
For a convex body $K$, its support function $g_K\colon \mb{S}^1 \to
\mb{R}^1$ is defined as \[g_K(\theta) := \max_{p\in K}\langle \theta,
p\rangle.\] (Here $\langle \cdot , \cdot\rangle$ denotes the usual
Euclidean inner product.) Using this map we associate the body $K$
with a dual support path $K^*$ given by \[K^* = \{ (t, g_K(2\pi
t) ) \: : \: 0\leq t \leq 1\}.\] In this way a family $F$ of convex
bodies can be associated with a dual system of paths $F^*$, where the
common supporting tangents of pairs of bodies of $F$ are in bijective
correspondence with the intersection points between pairs of dual
curves of $F^*$. In general, $F^*$ may not be a
system, but the members of $F$ may be perturbed in such a way that
there are no tangential intersections and no triple intersections
among of the paths in $F^*$. By standard compactness arguments, this
can be done without increasing the number of common supporting
tangents and without changing which subfamilies are in convex
position. The paths of $F^*$ may be assumed to have distinct
endpoints as well. Since any two members of $F$ are in convex position, they must have at least two common supporting tangents. Consequently, since 
each pair of members of $F$ have at most $2k$ common supporting
tangents, then each pair of paths in $F^*$ cross at least twice and
at most $2k$ times. The key observation is that a subfamily of $F$ is
in convex position if and only if the corresponding subsystem of $F^*$
is upper convex. Consequently, $h_k(n) \leq U_{2k}(n)$. 
\end{proof}

\subsection{Outline of the paper}

The proofs of Theorems \ref{general cupscaps} and \ref{general ES} are
purely combinatorial and essentially boil down to a simple parity
argument: If path $i$ starts below path $j$, then every time path $i$
appears above path $j$, they must have crossed an odd number of
times. This is of course just the final punchline, and
a large part of this paper is devoted to developing the appropriate
combinatorial machinery to make this formal. 

In section \ref{sec:CombProp} we translate the problem into
combinatorial terms. The most crucial notion is that of the {\em local
  sequences} of a system. This is an encoding of the ``crossing
pattern'' of the system, and it records, for each path, the order in
which it meets the other paths. This encoding allows us to formally
define what we mean by a ``highly organized substructure'' in a system
of paths, which we refer to as {\em regular systems}. It is a simple
consequence of Ramsey's theorem that every sufficiently large system
of paths contains a large regular subsystem. The main technical
results of this paper give detailed descriptions of the envelopes of a
regular system. This is the content of Theorems \ref{UL envelope} and
\ref{envelope theorem}, and it is easily seen that these results imply
Theorems \ref{general cupscaps} and \ref{general ES}.  

The difficulty that arises when trying to analyze regular systems is witnessed by the fact 
that the number of distinct regular $k$-crossing systems grows very
rapidly with $k$. The proofs of Theorems \ref{UL envelope} and
\ref{envelope theorem} span over sections \ref{sec:SeqTab} -- \ref{evlps}, 
and actually contain a complete characterization of
all regular systems. Such a characterization is obtained by
reformulating the notion of a regular system in terms of certain
properties concerning {\em sequences} on an {\em ordered
  alphabet}. The discussion is organized by dividing these properties
into combinatorial and geometric ones.   

In section \ref{sec:SeqTab} we treat the combinatorial
properties. Here we introduce the notion of a {\em tableau} which is a
sequence of sequences on an ordered alphabet. The connection to our 
problem is that the local sequences of a system gives rise to a
tableau (but we do not require the converse to be true). We further
define {\em regular tableaux} which are a combinatorial abstraction of
regular systems. The main result of this section is the bijective
correspondence established in Corollary \ref{tab-corr1} which
characterizes all regular tableaux. Our results in this section are
completely elementary and somewhat technical. The notion of regular
tableaux seems natural, but we are unaware of connections with
previously studied structures.

In section \ref{sec:GeomTab} we treat the geometric properties. That
is, we give conditions for when a tableau corresponds to the local
sequences of a system of paths. These results follow standard
arguments, most of which were introduced in the study of arrangements
of lines and pseudolines.  

In section \ref{sec:Kara} we are ready to characterize the regular
systems. This essentially amounts to describing the intersection of
the combinatorial properties of section \ref{sec:SeqTab} and the
geometric properties of section \ref{sec:GeomTab}. As a consequence we establish the basic structure of the local sequences of any regular system.

In section \ref{evlps} we apply the characterization of regular systems with the  aforementioned parity argument to
describe the upper and lower envelopes. 

In section \ref{remark4} we conclude with some final remarks and open
problems. 

All the necessary notions are elementary and will be formally introduced along
the way. As usual, the set of natural numbers is denoted by $\mb{N}$,
and the finite set $\{1,2,\dots,n\}$ is denoted by $[n]$.

\section{Combinatorial properties of  systems of paths} \label{sec:CombProp} 


\subsection{Local sequences} Let $S$ be a system and let
$A \subset \mb{N}$ with $|A|=|S|$. 
Label the paths of
$S$ by the elements of $A$ according to the order of their left
endpoints from bottom to top. We say that $S$ is \df{labeled} by
$A$. Throughout the rest of this paper we will always assume
that a system is labeled by an increasing sequence of positive integers.  

\begin{define} \label{locals} Let $S$ be a system
  labeled by $A \subset \mb{N}$. The \df{local sequence} of path $i$
  is the sequence 
  on $A \setminus \{i\}$ which records the order in which path $i$
  intersects the other paths of the system as it is traversed from
  left to right. 
\end{define}

\begin{example} The figure below shows a system of paths labeled by
  $[4]$.

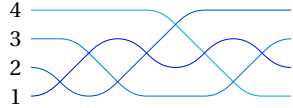
\begin{center}
  \begin{tikzpicture}
    \begin{scope}[scale = .38]      
      \draw[blue!60!black!30!cyan] (0,3) \ls\ls\ls\ls\ds\ld\ld\de\ls;
      \draw[blue!70!black!50!cyan] (0,2) \ls\ds\ld\de\ls\ls\us\lu\ue;
      \draw[blue!80!black!70!cyan] (0,1) \ds\de\us\lu\lu\ue\ls\ls\ls;
      \draw[blue!90!black!90!cyan] (0,0) \us\lu\ue\ds\de\us\ue\ds\de;
      \node[left] at (0,0) {\ft $1$};
      \node[left] at (0,1) {\ft $2$};
      \node[left] at (0,2) {\ft $3$};
      \node[left] at (0,3) {\ft $4$};
    \end{scope}
  \end{tikzpicture}
\end{center}

This gives us the corresponding local sequences.

\begin{center}
  \begin{tikzpicture}
    \begin{scope}
      \node[left] at (0,2) {\ft Path 1 :};
      \node[left] at (0.5,2) {\ft 2};
      \node[left] at (1,2) {\ft 3};
      \node[left] at (1.5,2) {\ft 2};
      \node[left] at (2,2) {\ft 4};
      \node[left] at (2.5,2) {\ft 3};
     \end{scope}

    \begin{scope}[yshift = .6cm]
      \node[left] at (0,2) {\ft Path 2 :};
      \node[left] at (0.5,2) {\ft 1};
      \node[left] at (1,2) {\ft 3};
      \node[left] at (1.5,2) {\ft 1};
      \node[left] at (2,2) {\ft 4};
     \end{scope}

    \begin{scope}[yshift = 1.2cm]   
      \node[left] at (0,2) {\ft Path 3 :};
      \node[left] at (0.5,2) {\ft 1};
      \node[left] at (1,2) {\ft 2};
      \node[left] at (1.5,2) {\ft 4};
      \node[left] at (2,2) {\ft 1};
     \end{scope}

     \begin{scope}[yshift = 1.8cm]   
      \node[left] at (0,2) {\ft Path 4 :};
      \node[left] at (0.5,2) {\ft 2};
      \node[left] at (1,2) {\ft 1};
      \node[left] at (1.5,2) {\ft 3};
     \end{scope}

  \end{tikzpicture}
\end{center}

\end{example}

For a finite set $A\subset \mb{N}$ and $i\in A$, let $A_i^- = \{j\in
A \: :\: j < i\}$ and $A_i^+= \{j\in A\: : \: j > i\}$. The following lemma implies that the upper and lower envelopes of a system can be
determined from the local sequences of its 
paths. 

\begin{lemma}\label{local-envelope}
  Let $S$ be a system labeled by $A$. Path $i$
  appears on the upper envelope of $S$ if and only if there is an
  initial string of its local sequence which contains every element of
  $A_i^+$ an odd number of times and every element of $A_i^-$ an even
  number of times. (The same holds for the lower 
  envelope by reversing the roles of $A_i^+$ and $A_i^-$.) 
  \end{lemma}

\begin{proof}
  Suppose $i < j$. This means that path $i$ starts below
  path $j$. Therefore, every time path $i$ appears above path
  $j$, the two paths should have crossed an odd number of
  times. Similarly, every time path $j$ appears above path $i$, they
  should have crossed an even number of times.
\end{proof}

\subsection{Signatures of systems of size $3$}
In our proofs of Theorems \ref{general cupscaps} and \ref{general
  ES} the systems on three paths play an important role. 
In this case the crossings of the system are linearly ordered. (For
systems of size greater than 3 we generally only have a partial
ordering of the crossings.) Here we
introduce a combinatorial signature which records this ordering.

\begin{define} \label{signatures}
Let $S$ be a system labeled by $\{i_1,i_2,i_3\}\subset
\mb{N}$ where 
$i_1 < i_2 < i_3$. 
The \df{signature} of $S$ is the word $\sigma = \sigma(S)$
on the alphabet $\{x,y,z\}$ which records the linear ordering of the
crossings of $S$ by the rules 
\[\{i_1,i_2\}\text{-crossing} \mapsto x \;\; , \;\; \{i_1,i_3\}\text{-crossing}
\mapsto y \;\; , \;\; \{i_2,i_3\}\text{-crossing} \mapsto z \]
\end{define}

\begin{remark}
  We consider the alphabet $\{x,y,z\}$ to be ordered $x \prec y \prec
  z$. This will be crucial later on. Notice that this just corresponds
  to the lexicographical ordering of the pairs $(i_1,i_2)$,
  $(i_1,i_3)$, $(i_2,i_3)$, but introducing letters $x$, $y$, $z$
  simplifies the notation.  
\end{remark}

\begin{remark} \label{sign=local}
  The signature of a system of size 3 is a compact way of encoding the
  local sequences of the system. Let $S$ be a system labeled by $[3]$
  with signature $\sigma$. Let $\sigma_{_{\{x,y\}}}$ denote the word
  obtained by deleting the letter $z$ from $\sigma$. If we replace
  each $x$ by $2$ and each $y$ by $3$ in $\sigma_{_{\{x,y\}}}$, then
  we obtain the local 
  sequence of path 1. This follows from the definition of the
  signature. Similarly, let $\sigma_{_{\{x,z\}}}$ and
  $\sigma_{_{\{y,z\}}}$ be the words obtained by deleting the letters
  $y$ and $x$ from $\sigma$, respectively. Replacing each $x$ by $1$
  and each $z$ by $3$ in $\sigma_{_{\{x,z\}}}$ gives us the local
  sequence of path 2, and replacing each $y$ by $1$ and each $z$ by
  $2$ in $\sigma_{_{\{y,z\}}}$ gives us the local sequence of path $3$.  
\end{remark}

\begin{example}
  The system depicted in the figure below
  has signature $\sigma =  xy^3zx^2z^2$.

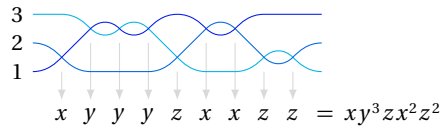
\begin{center}
\begin{tikzpicture}
\begin{scope}[scale = .38]      

\node at (1,-1.5) {\ft $x$};
\node at (2,-1.55) {\ft $y$};
\node at (3,-1.55) {\ft $y$};
\node at (4,-1.55) {\ft $y$};
\node at (5,-1.5) {\ft $z$};
\node at (6,-1.5) {\ft $x$};
\node at (7,-1.5) {\ft $x$};
\node at (8,-1.5) {\ft $z$};
\node at (9,-1.5) {\ft $z$};
\node[right] at (9.5,-1.5) {\ft $=\: \: xy^3zx^2z^2$};

\draw[gray!30!white, -latex] (1,0) --(1,-1);
\draw[gray!30!white, -latex] (2,1) --(2,-1);
\draw[gray!30!white, -latex] (3,1) --(3,-1);
\draw[gray!30!white, -latex] (4,1) --(4,-1);
\draw[gray!30!white, -latex] (5,0) --(5,-1);
\draw[gray!30!white, -latex] (6,1) --(6,-1);
\draw[gray!30!white, -latex] (7,1) --(7,-1);
\draw[gray!30!white, -latex] (8,0) --(8,-1);
\draw[gray!30!white, -latex] (9,0) --(9,-1);

\draw[blue!70!black!30!cyan] (0,2) \ls\ds\du\ud\ld\de\ls\us\ud\de;
\draw[blue!80!black!60!cyan] (0,1) \ds\de\ls\ls\us\lu\ud\ld\du\ue;
\draw[blue!90!black!90!cyan] (0,0) \us\lu\ud\du\ue\ds\du\ue\ls\ls;
\node[left] at (0,0) {\ft $1$};
\node[left] at (0,1) {\ft $2$};
\node[left] at (0,2) {\ft $3$};
\end{scope}
\end{tikzpicture}
\end{center}

The words $\sigma_{_{\{x,y\}}}$, $\sigma_{_{\{x,z\}}}$, and
$\sigma_{_{\{y,z\}}}$ give us the following 
local sequences.

\begin{center}
  \begin{tikzpicture}
    \begin{scope}
      \node[left] at (-2,1.95) {\ft $\sigma_{_{\{x,y\}}} = xy^3x^2$};
      \draw[-latex] (-1.9,2) -- (-1.3,2);      
      \node[left] at (0,2) {\ft Path 1 :};
      \node[left] at (0.5,2) {\ft 2};
      \node[left] at (1,2) {\ft 3};
      \node[left] at (1.5,2) {\ft 3};
      \node[left] at (2,2) {\ft 3};
      \node[left] at (2.5,2) {\ft 2};
      \node[left] at (3,2) {\ft 2};
     \end{scope}

    \begin{scope}[yshift = .6cm]
      \node[left] at (-2,1.95) {\ft $\sigma_{_{\{x,z\}}} = xzx^2z^2$};
      \draw[-latex] (-1.9,2) -- (-1.3,2);      
      \node[left] at (0,2) {\ft Path 2 :};
      \node[left] at (0.5,2) {\ft 1};
      \node[left] at (1,2) {\ft 3};
      \node[left] at (1.5,2) {\ft 1};
      \node[left] at (2,2) {\ft 1};
      \node[left] at (2.5,2) {\ft 3};
      \node[left] at (3,2) {\ft 3};
     \end{scope}

    \begin{scope}[yshift = 1.2cm]
      \node[left] at (-2,1.95) {\ft $\sigma_{_{\{y,z\}}} = y^3z^3$};
      \draw[-latex] (-1.9,2) -- (-1.3,2);      
      \node[left] at (0,2) {\ft Path 3 :};
      \node[left] at (0.5,2) {\ft 1};
      \node[left] at (1,2) {\ft 1};
      \node[left] at (1.5,2) {\ft 1};
      \node[left] at (2,2) {\ft 2};
      \node[left] at (2.5,2) {\ft 2};
      \node[left] at (3,2) {\ft 2};
     \end{scope}
  \end{tikzpicture}
\end{center}
\end{example}

\begin{prop} \label{parity}
  Let $S$ be a system of size 3 with signature $\sigma$. Suppose
  $u,v,w \in \{x,y,z\}$ where $u \neq v$ and 
  $w\neq v$. The
  following hold. 
  \begin{enumerate}
  \item For $\sigma =u^pv \cdots $ we have $u \in \{x,z\}$ and $v=y
    \iff p$ is odd.
  \item  For $\sigma  = \cdots u v^p w\cdots$ we have $u = w \iff p$ is even
  \end{enumerate}
\end{prop}

\begin{proof}
  The first claim of {\em (1)} is obvious: The first crossing
  which occurs must involve path $i_2$, since a crossing always
  involves the path which is currently in the middle, and $i_2$ starts
  in the middle. Path $i_2$ is in the
  middle after the first $p$ crossings if and only $p$ is even, hence,
  the $(p+1)$'st crossing is an $\{i_1,i_3\}$-crossing if and only if
  $p$ is odd. For part {\em (2)}, suppose the $u$ corresponds to an
  $\{i,j\}$-crossing, where $j$ is the path which is in the middle
  {\em after} this crossing occurs. This means that the next $p$ crossings
  which correspond to $v^p$ involve path $j$. The $w$ corresponds to an
  $\{i,k\}$-crossing and suppose $k$ is the path which is in the middle
  {\em before} this crossing occurs. Clearly, $k=j$
  if and only if $p$ is even.    
\end{proof}

\begin{remark}
  It is not hard to verify Proposition \ref{parity} characterizes the
  set of signatures. That is, any word on the alphabet $\{x,y,z\}$
  which satisfies the conditions of Proposition \ref{parity} corresponds 
  to the signature of a system of size 3. We leave the
  proof to the reader.
\end{remark}

\subsection{Regular systems} 

\begin{define} \label{SR}
A system $S$ is \df{regular} if it has size at least 4
and the signatures $\sigma(T) = \sigma(T') \neq \emptyset$ for all
subsystems $T$ and 
$T'$ of size 
3. The unique signature of the subsystems of size 3 is called the
\df{signature} of the regular system.
\end{define}

\begin{remark}
  Note that in a regular system each pair of paths must cross the
  same number of times and that any subsystem of size at least
  4 is also regular.
\end{remark}

\begin{example}\label{1reg}
  There are precisely two distinct signatures of 1-crossing systems of
  size 3, and for any $n\geq 4$ there exists regular systems of size
  $n$ with these signatures. The system below on the left is regular
  with signature $xyz$, while the system below on the right is regular
  with signature $zyx$.   

\medskip

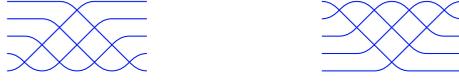
\begin{center}
\begin{tikzpicture}[scale=.23]
\begin{scope}
\draw[blue!90!black!90!cyan](0,5)\ls\ls\ls\ds\ld\ld\ld\de;
\draw[blue!90!black!90!cyan](0,4)\ls\ls\ds\ld\ld\de\us\ue;
\draw[blue!90!black!90!cyan](0,3)\ls\ds\ld\de\us\lu\ue\ls;
\draw[blue!90!black!90!cyan](0,2)\ds\de\us\lu\lu\ue\ls\ls;
\draw[blue!90!black!90!cyan](0,1)\us\lu\lu\lu\ue\ls\ls\ls;
\end{scope}
\begin{scope}[xshift=18cm, yscale = -1, yshift = -6cm]
\draw[blue!90!black!90!cyan](0,5)\ls\ls\ls\ds\ld\ld\ld\de;
\draw[blue!90!black!90!cyan](0,4)\ls\ls\ds\ld\ld\de\us\ue;
\draw[blue!90!black!90!cyan](0,3)\ls\ds\ld\de\us\lu\ue\ls;
\draw[blue!90!black!90!cyan](0,2)\ds\de\us\lu\lu\ue\ls\ls;
\draw[blue!90!black!90!cyan](0,1)\us\lu\lu\lu\ue\ls\ls\ls;
\end{scope}
\end{tikzpicture}
\end{center}

One can think of these systems as the dual to what Erd\H{o}s and
Szekeres call cups and caps \cite{erd-sze1}, and their result states
that any 1-crossing system of size $\binom{2n-4}{n-2}+1$ contains a
regular subsystem of size $n$.\end{example}
 
\begin{example}\label{stronglyreg}
The figure below shows a system $S$ of size 4 together with its 4
subsystems of size 3. It is easily seen that each subsystem of size 3
has signature $\sigma = xy^2x^3z^2xyz^2y^2z$.  Therefore $S$ is a
regular system. 

\medskip

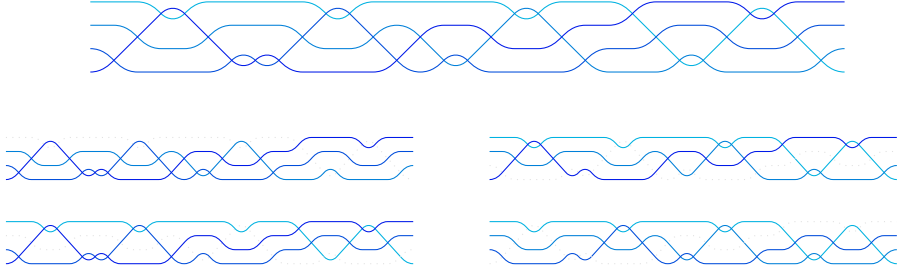
\begin{center}
\begin{tikzpicture}[scale=.31]
\begin{scope}[xshift=.5cm]
\draw[blue!60!black!30!cyan](0,4)\ls\ls\ds\du\ue\ls\ls\ls\ls\ds\du\ue\ls\ls\ls
\ls\ls\ds\du\ue\ls\ls\ds\ld\ld\du\lu\lu\ud\ld\ld\de;
\draw[blue!70!black!50!cyan](0,3)\ls\ds\de\ls\us\ue\ls\ls\ds\de\ls\us\ue\ds
\ld\du\lu\lu\ud\ld\ld\de\ls\ls\us\ud\de\ls\ls\ls\us\ue;
\draw[blue!80!black!70!cyan](0,2)\ds\de\ls\ls\ls\us\ud\du\lu\lu\ud\ld\ld\de
\us\ud\de\ls\ls\ls\us\ue\ls\us\ue\ls\ds\de\ls\us\ue\ls; 
\draw[blue!90!black!90!cyan](0,1)\us\lu\lu\ud\ld\ld\du\ud\de\ls\ls\ls\us\lu
\ue\ls\ds\de\ls\us\ue\ls\us\ue\ls\ls\ls\ds\du\ue\ls\ls;
\end{scope}

\begin{scope}[xshift=-2cm]
\begin{scope}[xscale=.9, scale=.6, yshift=-7cm, xshift=-2cm]
\draw (0,4)[white!88!black, dotted]\ls\ls\ds\du\ue\ls\ls\ls\ls\ds\du\ue\ls\ls\ls
\ls\ls\ds\du\ue\ls\ls\ds\ld\ld\du\lu\lu\ud\ld\ld\de;
\draw[blue!70!black!50!cyan] (0,3)\ls\ds\de\ls\us\ue\ls\ls\ds\de\ls\us\ue\ds
\ld\du\lu\lu\ud\ld\ld\de\ls\ls\us\ud\de\ls\ls\ls\us\ue;
\draw[blue!80!black!70!cyan](0,2)\ds\de\ls\ls\ls\us\ud\du\lu\lu\ud\ld\ld\de
\us\ud\de\ls\ls\ls\us\ue\ls\us\ue\ls\ds\de\ls\us\ue\ls; 
\draw[blue!90!black!90!cyan](0,1)\us\lu\lu\ud\ld\ld\du\ud\de\ls\ls\ls\us\lu
\ue\ls\ds\de\ls\us\ue\ls\us\ue\ls\ls\ls\ds\du\ue\ls\ls;
\end{scope}

\begin{scope}[xscale=.9, scale=.6, yshift=-13cm, xshift=-2cm]
\draw[blue!60!black!30!cyan](0,4)\ls\ls\ds\du\ue\ls\ls\ls\ls\ds\du\ue\ls\ls\ls
\ls\ls\ds\du\ue\ls\ls\ds\ld\ld\du\lu\lu\ud\ld\ld\de;
\draw(0,3)[white!88!black, dotted]\ls\ds\de\ls\us\ue\ls\ls\ds\de\ls\us\ue\ds
\ld\du\lu\lu\ud\ld\ld\de\ls\ls\us\ud\de\ls\ls\ls\us\ue;
\draw[blue!80!black!70!cyan](0,2)\ds\de\ls\ls\ls\us\ud\du\lu\lu\ud\ld\ld\de
\us\ud\de\ls\ls\ls\us\ue\ls\us\ue\ls\ds\de\ls\us\ue\ls; 
\draw[blue!90!black!90!cyan](0,1)\us\lu\lu\ud\ld\ld\du\ud\de\ls\ls\ls\us\lu
\ue\ls\ds\de\ls\us\ue\ls\us\ue\ls\ls\ls\ds\du\ue\ls\ls;
\end{scope}

\begin{scope}[xscale=.9, scale=.6, xshift=36cm, yshift=-7cm]
\draw[blue!60!black!30!cyan](0,4)\ls\ls\ds\du\ue\ls\ls\ls\ls\ds\du\ue\ls\ls\ls
\ls\ls\ds\du\ue\ls\ls\ds\ld\ld\du\lu\lu\ud\ld\ld\de;
\draw[blue!70!black!50!cyan](0,3)\ls\ds\de\ls\us\ue\ls\ls\ds\de\ls\us\ue\ds
\ld\du\lu\lu\ud\ld\ld\de\ls\ls\us\ud\de\ls\ls\ls\us\ue;
\draw(0,2)[white!88!black, dotted]\ds\de\ls\ls\ls\us\ud\du\lu\lu\ud\ld\ld\de
\us\ud\de\ls\ls\ls\us\ue\ls\us\ue\ls\ds\de\ls\us\ue\ls; 
\draw[blue!90!black!90!cyan](0,1)\us\lu\lu\ud\ld\ld\du\ud\de\ls\ls\ls\us\lu
\ue\ls\ds\de\ls\us\ue\ls\us\ue\ls\ls\ls\ds\du\ue\ls\ls;
\end{scope}

\begin{scope}[xscale=.9, scale=.6, xshift=36cm, yshift=-13cm]
\draw[blue!60!black!30!cyan](0,4)\ls\ls\ds\du\ue\ls\ls\ls\ls\ds\du\ue\ls\ls\ls
\ls\ls\ds\du\ue\ls\ls\ds\ld\ld\du\lu\lu\ud\ld\ld\de;
\draw[blue!70!black!50!cyan](0,3)\ls\ds\de\ls\us\ue\ls\ls\ds\de\ls\us\ue\ds
\ld\du\lu\lu\ud\ld\ld\de\ls\ls\us\ud\de\ls\ls\ls\us\ue;
\draw[blue!80!black!70!cyan](0,2)\ds\de\ls\ls\ls\us\ud\du\lu\lu\ud\ld\ld\de
\us\ud\de\ls\ls\ls\us\ue\ls\us\ue\ls\ds\de\ls\us\ue\ls; 
\draw(0,1)[white!88!black, dotted]\us\lu\lu\ud\ld\ld\du\ud\de\ls\ls\ls\us\lu
\ue\ls\ds\de\ls\us\ue\ls\us\ue\ls\ls\ls\ds\du\ue\ls\ls;
\end{scope}

\end{scope}
\end{tikzpicture}
\end{center}
\end{example}

\begin{example}
The figure below shows a system of size 3 with signature $\sigma =
xy^2x^4yz^4y^2z$. 

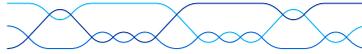
\medskip

\begin{center}
\begin{tikzpicture}[scale=.3]
\begin{scope}[xscale=.93]
\draw[blue!90!black!30!cyan](0,3)\ls\ds\du\ue\ls\ls\ls\ds\ld\du\ud\du
\lu\ud\ld\du\ue;
\draw[blue!75!black!60!cyan](0,2)\ds\de\ls\us\ud\du\ud\de\us\ud\du\ud
\de\ls\us\ud\de;  
\draw[blue!50!black!80!cyan](0,1)\us\lu\ud\ld\du\ud\du\lu\ue\ls\ls\ls
\ds\du\ue\ls\ls;
\end{scope}
\end{tikzpicture}
\end{center}

As we will see, there exists no regular system with signature
$\sigma$. This particular example will be revisited in Example \ref{special}.  
\end{example}

\subsection{The existence of large regular
  systems} \label{existence} For any fixed integer $k\geq 1$, it follows
from Ramsey's theorem that every sufficiently large $k$-crossing
system contains a large subsystem which is regular. The argument goes 
as follows. Let $R_3(n;M)$ denote the {\em symmetric Ramsey number} for
$M$-partitions of the edge set of the complete 3-uniform
hypergraph. In other words, for $N\geq R_3(n;M)$, every partition of
the triples of $[N]$ into at most $M$ classes, there exists
a subset $A\subset [N]$, with $|A|\geq n$, such that every
triple of $A$ belongs to the same class. For every $k \geq 1$, there is a
finite number of signatures of $k$-crossing systems of size 3. Let
$M_k$ denote this number. Let $S$ be a 
$k$-crossing system of size $N\geq R_3(n;M_k)$. 
If we partition the subsystems of size 3 according to their signature,
Ramsey's theorem implies that $S$ contains a regular system
of size $n$. 

\subsection{Envelopes of regular systems}
In a regular system, a path will typically cross all paths above or
below it going directly to an envelope, and if this happens for one
path, the same happens for all paths. This behavior can be seen in Example \ref{stronglyreg} and results in either very few or all paths appearing on the upper or lower envelopes. This is made precise in Theorems  \ref{UL envelope} and \ref{envelope theorem}, below, and since every sufficiently large $k$-crossing system contains a large subsystem which is regular (by the argument in section \ref{existence}), these statements easily imply Theorems \ref{general cupscaps} and \ref{general ES}.

\begin{theorem} \label{UL envelope}
Every regular system is upper convex or lower convex. \end{theorem}

\begin{proof}[Proof of Theorem \ref{general cupscaps}]
  By the argument in section \ref{existence},
  if $S$ is a $k$-crossing system of size at least $R_3(n;M_k)$,
  then $S$ contains a subsystem $S'$ of size $n$ which is
  regular. By Theorem \ref{UL envelope}, $S'$ is
  upper convex or lower convex. Therefore $C_k(n)\leq R_3(n;M_k)$.
\end{proof}

\bigskip

\begin{theorem}\label{envelope theorem}
Let $S$ be a regular system. 
\begin{enumerate}
\item If $S$ is $2$-crossing, then $S$ is upper convex or
  only $2$ paths appear on the upper envelope of $S$.  

\item If $S$ is $4$-crossing, then $S$ is upper convex or at
most $3$ paths appear on the upper envelope of $S$. 

\item If $S$ is $k$-crossing for $k>4$, then $S$ is upper
  convex or at most $4$ distinct paths appear on the upper
envelope of $S$.
\end{enumerate}
\end{theorem}

\begin{proof}[Proof of Theorem \ref{general ES}] 
  If $S$ is a $k$-crossing system of size at least $R_3(n;M_k)$, then $S$ 
  contains a regular subsystem $S'$ of size $n$. 
   For {\em (1)} of Theorem \ref{general ES} suppose $S$ is
  $2$-crossing. By {\em (1)} of Theorem \ref{envelope
    theorem}, $S'$ is upper convex or only 2 paths 
  appear on the upper envelope of $S'$. The latter case is
  impossible: If only 2 paths appear on the upper envelope, then $S'$
  contains a subsystem of size 3 which is not upper convex (take 
  the two paths from the upper envelope together with any other path
  of $S'$). This contradicts the hypothesis. Similarly, {\em
    (2)} and {\em (3)} of Theorem \ref{general ES} are implied by
   {\em (2)} and {\em (3)} of Theorem \ref{envelope theorem},
  respectively. Therefore $U_k(n) \leq R_3(n;M_k)$. 
\end{proof}

\section{Sequences and Tableaux} \label{sec:SeqTab}

\subsection{Regular Sequences} \label{2regs}

For a totally ordered alphabet $A$, a word is a finite sequences of letters in $A$, and a language is a set of words.  We will generally refer to words on $A \subset \mb{N}$ as sequences. 
Let $A^\star$ denote the language of all words on $A$, and for a language $L$, let $L^\star$ denote the language of all words formed by concatenating words in $L$.
For words $\omega$ and $\psi$ in $A^\star$, let $\omega \cdot \psi$ denote the concatenation of $\omega$ and $\psi$, 
let $|\omega|$ denote the length of $\omega$, and
let $[\omega]$ denote the  set of distinct letters appearing in $\omega$. 
For $a \in A$, let $a^i$ be the word $a\cdots a$ with $|a^i| = i$.
For a subset $X \subset A$, let $\omega_{_X}$ denote the subword of $\omega$ consisting of letters in $X$. We call this the \df{restriction} of $\omega$ to $X$. 
Define a map $N \colon A^\star \to \mb{N}^\star$ as follows. For $\omega \in A^\star$ with $[\omega] = \{a_1, \cdots, a_k\}$ where $a_i < a_{i+1}$, let $N(\omega)$ denote the sequence in $\mb{N}^\star$ obtained by the map $a_i \mapsto i$. For instance, $N(4, 2, 6, 6, 7) =  (2, 1, 3, 3, 4)$. Two words $\omega$ and $\omega'$ are \df{order equivalent} if $N(\omega) = N(\omega')$, in which case we write ${\omega \sim \omega'}$. Note that order equivalence is an equivalence relation. 
The map $N$ and the notion of order equivalence naturally extend to words which are not defined on the same alphabets. For instance, if $A$ is the Latin alphabet then $N({\tt h\; e\; l\; l\; o}) = (2,1,3,3,4)$, so
$(4, 2, 6, 6, 7) \sim ``{\tt h\; e\; l\; l\; o}"$. 
Let $\omega(a,i)$ denote the prefix of $\omega$ ending with the $i$'th occurrence of the letter $a$. In particular, 
$|\omega(a,i)|$ is the position of the $i$'th occurrence of $a$.
For instance, if $\omega = ``{\tt S\; z\; e\; k\; e\; r\; e\; s}"$ then $\omega({\tt e},2) = ``{\tt S\; z\; e\; k\; e}"$ and $|\omega({\tt e},2)|=5$.  

\begin{define}
Let $A\subset \mb{N}$ and $\omega \in A^\star$. We say $\omega$ is a \df{regular} sequence on $A$ if the following hold.
\begin{itemize}
\item $A$ consists of at least $3$ elements.
\item $\omega_{_X} \sim \omega_{_Y}$ for all subsets $X, Y \subset A$ of size 2. 
\end{itemize}
\end{define}

\begin{example}
The sequence $\omega = (1, 2, 3, 3, 3, 2, 2, 1, 1)$ is regular on $[3]$. The restrictions are:  
\[\begin{array}{lcclccl}  
\omega_{_{\{1,2\}}} = (1, 2, 2, 2, 1, 1), & & &  
\omega_{_{\{1,3\}}} = (1, 3, 3, 3, 1, 1), & & & 
\omega_{_{\{2,3\}}} = (2, 3, 3, 3, 2, 2). 
\end{array}\] 
\end{example}

\begin{example}
The sequence $\omega = (1,2,3,1,3,2)$ is not regular since the
restrictions $\omega_{_{\{1,2\}}} = (1,2,1,2)$ and
$\omega_{_{\{2,3\}}} = (2,3,3,2)$ are not order equivalent. 
\end{example}

A regular sequence on $A$ is uniquely determined by $A$ and its restriction to any two elements. 

\begin{lemma}\label{2seq-unique}
If $\psi$ and $\omega$ are regular sequences on $[n]$ and $\psi_{_{\{1,2\}}} = \omega_{_{\{1,2\}}}$,
then $\psi = \omega$.
\end{lemma}

\begin{proof}
Suppose the lemma fails.  
Let the first entry where $\psi$ and $\omega$ differ be the $p$'th occurrence of $i$ in $\psi$ and the $q$'th occurrence of $j$ in $\omega$.  
This implies that $\psi_{_{\{i,j\}}}$ and $\omega_{_{\{i,j\}}}$ coincide for the first $p+q-2$ entries, but at the $p+q-1$'th entry they differ,
which contradicts the fact that they are both order equivalent to $\omega_{_{\{1,2\}}}$.
\end{proof}


We define a language characterizing the restrictions of regular sequences to subsets of size 2. 
Let $\{a, b\}$ be an ordered alphabet with $a\prec b$. A \df{balanced block} on the alphabet $\{a,b\}$ is a word of the form \[ a^rb^r = \underset{^{\times r}}{a \cdots a}\underset{^{\times r}}{b\cdots b} \hspace{.6cm} \text{or} \hspace{.6cm} b^ra^r = \underset{^{\times r}}{b\cdots b}\underset{^{\times r}}{a\cdots a}
\hspace{.6cm} \text{for some $r\in \mb{N}$.}\] Let $B_{_{(ab)}}$
denote the language of balanced blocks on $\{a,b\}$. 
A word $\omega \in B\rlap{$^\star$}_{_{(ab)}}$ is called \df{balanced}; 
that is, words $\omega = \omega_1 \cdots \omega_k$ 
where $\omega_i = a^{r_i}b^{r_i}$ or $b^{r_i}a^{r_i}$.  
We define the \df{block sizes} of $\omega$ to be the sequence $(r_1, r_2, \dots, r_k)$, which we denote by $\langle \omega \rangle$. 
For instance, $\langle (a^2b^2)(ba)(ab) \rangle = (2,1,1)$. 
Here is a simple way to check if a word is balanced.  Starting from the first letter parsing one letter at a time, count the number of occurrences of $a$ and the number of occurrences of $b$.  Whenever $a$ is succeeded by $b$, the number of $a$'s counted so far must be at least the number of $b$'s, and analogously when $b$ is succeeded by $a$.  In the end, we must have counted the same number of each letter. 
That is, we have the following equivalent definition. 

\begin{lemma}\label{bal-def}
A word $\omega \in \{a,b\}^\star$ is balanced if and only if the following hold.
\begin{itemize}
\item $|\omega_{_{\{a\}}}| = |\omega_{_{\{b\}}}|$.
\item If $|\omega(a,i)| +1 = |\omega(b,j+1)|$ then $i \geq j$.
\item If $|\omega(b,i)| +1 = |\omega(a,j+1)|$ then $i \geq j$.
\end{itemize}
\end{lemma}

\begin{proof}
Let $\omega$ satisfy the conditions of the lemma.
We will show that $\omega$ is balanced by expressing $\omega$ as a concatenation of balanced blocks.  
We may assume by symmetry that $\omega$ begins with $a^jb\cdots$. 
If there are no more occurrences of $a$, then $\omega = a^j b^j$ is balanced. 
Otherwise $\omega$ begins with $a^jb^ia\cdots$,
and since $|\omega(b,i)|+1 = |\omega(a,j+1)|$ we have $i \geq j$, 
so the first $2j$ letters of $\omega$ are the balanced block $\omega_1 = a^j b^j$.

Let $\omega = \omega_1 \cdot \omega_2$.  
Now we have $|\omega_{2_{\{a\}}}| = |\omega_{_{\{a\}}}|-j$ and $|\omega_2(a,i-j)| = |\omega(a,i)| -2j$, and likewise for occurrences of the letter $b$ in $\omega_2$, so $\omega_2$ satisfies the conditions of the lemma.  
Therefore, by induction on the length of $\omega$, we can express $\omega$ as a concatenation of balanced blocks.

The other direction follows immediately from the definition of balanced blocks and the fact that the conditions of the lemma are preserved by concatenation.
\end{proof}

As a consequence we have the following.

\begin{corollary}\label{bal-concat}
If $\omega = \omega_1 \cdot \omega_2$ is balanced, then the following are equivalent.
\begin{itemize}
\item $\omega_1$ is balanced.
\item $\omega_2$ is balanced.
\item $|\omega_{1_{\{a\}}}| = |\omega_{1_{\{b\}}}|$.
\item $|\omega_{2_{\{a\}}}| = |\omega_{2_{\{b\}}}|$.
\end{itemize}
\end{corollary}

\begin{prop}\label{2reg}
For any $n\geq 3$ and $\omega \in \{a,b\}^\star$,  
there is a regular sequence $\nu$ on $[n]$ with $\nu_{_{\{1,2\}}} \sim \omega$ if and only if $\omega$ is balanced.
\end{prop}

\begin{proof}
We define a function, 
$\phi_{_n} \colon B\rlap{$^\star$}_{_{(ab)}} \to [n]^\star$,
which produces a regular sequence on $[n]$ such that $\phi_{_n}(\omega)_{_{\{1,2\}}} \sim \omega$. 
First, we define $\phi_{_n}$ on balanced blocks as  
\[\begin{array}{rcl}
  \phi_{_n}(a^rb^r) & \defeq & 
  \underset{^{\times r}}{1, \cdots, 1},
  \underset{^{\times r}}{2, \cdots, 2}, \cdots, 
  \underset{^{\times r}}{n, \cdots, n}  \\
 \phi_{_n}(b^ra^r) & \defeq & 
  \underset{^{\times r}}{n, \cdots, n}, \cdots,
  \underset{^{\times r}}{2, \cdots, 2}, 
  \underset{^{\times r}}{1, \cdots, 1}. 
\end{array}
\] 
Extend $\phi_{_n}$ to arbitrary words $\omega \in B\rlap{$^\star$}_{_{(ab)}}$ by
concatenation in the following way: For $\omega = \omega_1
\cdots \omega_k$, where $\omega_i \in B_{_{(ab)}}$, define
$\phi_{_n}(\omega) \defeq  \phi_{_n}(\omega_1) \cdots
\phi_{_n}(\omega_k)$. Note that the function $\phi_{_n}$ is well defined. 
Let $\nu = \phi_{_n}(\omega)$. 
Clearly $\nu_{_X} \sim \omega$ for any $X\subset [n]$ with $|X| = 2$, 
so the sequence $\nu = \phi_{_n}(\omega)$ is regular.  

For the other direction, 
let $\nu$ be a regular sequence on $[n]$, and suppose $\nu_{_{\{1,2\}}}$ is not balanced. 
Then, by Lemma~\ref{bal-def} and by symmetry, 
may assume there is $i < j-1$ such that $|\nu(1,i)|+1 = |\nu(2,j)|$.
Since $\nu_{_{\{1,2\}}} \sim \nu_{_{\{1,3\}}}$, 
the $j$'th occurrence of 3 must be between the $i$'th and $i{+}1$'th occurrence of 1,
${|\nu(1,i)|<|\nu(3,j)|<|\nu(1,i+1)|}$, 
and since $\nu_{_{\{1,2\}}} \sim \nu_{_{\{2,3\}}}$, 
the $j$'th occurrence of 3 must also be between the $i$'th and $i{+}1$'th occurrence of 2,
$|\nu(2,i)|<|\nu(3,j)|<|\nu(2,i+1)|$.
But this is impossible, since  $|\nu(2,i+1)| \leq |\nu(2,j)|-1 = |\nu(1,i)|$.
Thus, $\nu_{_{\{1,2\}}}$ is balanced.
\end{proof}

\begin{corollary} \label{2reg-cor}
Let $n\geq 3$ be fixed. 
The set of regular sequences on $[n]$ is in bijective correspondence with the language $B\rlap{$^\star$}_{_{(ab)}}$.
\end{corollary}

\begin{proof}
This follows from Lemma~\ref{2seq-unique} and Proposition~\ref{2reg}. 
\end{proof}

\begin{example}
  Consider $\omega = (ab)(a^2b^2)(ba) \in B\rlap{$^\star$}_{_{(ab)}}$. For $n=3$ we get 
  \[\nu = \phi_{_3}(\omega) = (1, 2, 3, 1, 1, 2, 2, 3, 3, 3, 2, 1).\]
  The restrictions are:
\[\begin{array}{lcclccl} 
\nu_{_{\{1,2\}}} = (1, 2, 1, 1, 2, 2, 2, 1), & & &  
\nu_{_{\{1,3\}}} = (1, 3, 1, 1, 3, 3, 3, 1), & & &  
\nu_{_{\{2,3\}}} = (2, 3, 2, 2, 3, 3, 3, 2).
\end{array}\]
\end{example}

\bigskip

\subsection{Regular tableaux} \label{3regs}

Let $A = \{a_1, \dots , a_n\} \subset \mb{N}$ with
$a_i<a_{i+1}$. 
A sequence $T = (\omega_1, \dots, \omega_n)$ where
$\omega_i \in (A\setminus\{a_i\})^\star$ is called a \df{tableau} on $A$. 
We call $\omega_i$, the $i$'th {row} of $T$, and denote it by $T(i) \defeq \omega_i$. 
Let $A^\tabox$ be the set of all tableaux on $A$. 
For a subset $X = \{a_{i_1}, \dots , a_{i_k}\}\subset A$
with $i_j<i_{j+1}$, 
we define a tableau $T_{_X} \defeq (T(i_1)_{_X}, \dots,T(i_k)_{_X}) \in X^\tabox$. 
We call this the \df{restriction} of $T$ to $X$.  
Define $N \colon A^\tabox \to \mb{N}^\tabox $ by $a_{i} \mapsto i$. 
The tableaux $T$ and $T'$ are \df{order equivalent} when $N(T) = N(T')$,
in which case we write $T\sim T'$. Given two tableaux $T_1$ and $T_2$ on $A$, we define
their concatenation, $T_1\cdot 
T_2$, by letting $(T_1\cdot T_2) (i) = T_1(i) \cdot T_2(i)$ for every
$i\in A$.

When considering specific examples of tableaux it is convenient to use a
graphical representation. We represent a tableau as
left-justified rows of boxes 
in increasing order from bottom to top 
containing the letters of each row in their given order. 
For instance, 
let $T = (\omega_1, \omega_2, \omega_3)$  
be a tableau on $\{4,7,9\}$ with rows $\omega_1 = (7,9,7,9,9,7)$,
$\omega_2 = (9,9,4,4)$,
and $\omega_3 = (4,7,7,4,7)$.
Its graphical representation is given below along with $N(T)$.

\begin{center}
  \begin{tikzpicture}
    \begin{scope}[scale = .4]
      \draw (0,0) --++(6,0) --++ (0,1) --++(-2,0) --++(0,1) --++(1,0)
      --++(0,1) --++(-5,0) --cycle;
      \draw (1,0) --++(0,3) (2,0) --++(0,3) (3,0) --++(0,3) (4,0)
      --++(0,1) (4,2) --++(0,1) (5,0) --++(0,1) (0,1) --++(4,0) (0,2)
      --++(4,0); 
      \node at (.5,.5) {\ft $7$};
      \node at (1.5,.5) {\ft $9$};
      \node at (2.5,.5) {\ft $7$};
      \node at (3.5,.5) {\ft $9$};
      \node at (4.5,.5) {\ft $9$};
      \node at (5.5,.5) {\ft $7$};
      \node (a) at (.5,1.5) {\ft $9$};
      \node at (1.5,1.5) {\ft $9$};
      \node at (2.5,1.5) {\ft $4$};
      \node at (3.5,1.5) {\ft $4$};
      \node at (.5,2.5) {\ft $4$};
      \node at (1.5,2.5) {\ft $7$};
      \node at (2.5,2.5) {\ft $7$};
      \node at (3.5,2.5) {\ft $4$};
      \node at (4.5,2.5) {\ft $7$};
    \end{scope}
   \node[left,xshift=-.2cm] at (a) {\ft $T=$};
    \begin{scope}[scale = .4, xshift = 12cm]
      \draw (0,0) --++(6,0) --++ (0,1) --++(-2,0) --++(0,1) --++(1,0)
      --++(0,1) --++(-5,0) --cycle;
      \draw (1,0) --++(0,3) (2,0) --++(0,3) (3,0) --++(0,3) (4,0)
      --++(0,1) (4,2) --++(0,1) (5,0) --++(0,1) (0,1) --++(4,0) (0,2)
      --++(4,0); 
      \node at (.5,.5) {\ft $2$};
      \node at (1.5,.5) {\ft $3$};
      \node at (2.5,.5) {\ft $2$};
      \node at (3.5,.5) {\ft $3$};
      \node at (4.5,.5) {\ft $3$};
      \node at (5.5,.5) {\ft $2$};
      \node (a) at (.5,1.5) {\ft $3$};
      \node at (1.5,1.5) {\ft $3$};
      \node at (2.5,1.5) {\ft $1$};
      \node at (3.5,1.5) {\ft $1$};
      \node at (.5,2.5) {\ft $1$};
      \node at (1.5,2.5) {\ft $2$};
      \node at (2.5,2.5) {\ft $2$};
      \node at (3.5,2.5) {\ft $1$};
      \node at (4.5,2.5) {\ft $2$};
    \end{scope}
   \node[left,xshift=-.2cm] at (a) {\ft $N(T)=$};
  \end{tikzpicture}
\end{center}

\begin{define}
  Let $T$ be a tableau on $A\subset \mb{N}$. We say $T$ is
  \df{regular} if the following hold.
  \begin{itemize}
  \item $A$ consists of at least $4$ elements.  
  \item $T_{_X} \sim T_{_Y}$  for all subsets $X, Y \subset A$ of size 3.
  \end{itemize} 
\end{define}

\begin{example} 
Below is a regular tableau on $[4]$.

\begin{center}
  \begin{tikzpicture}
    \begin{scope}[scale = .4]
      \draw (0,0) --++ (6,0) --++ (0,1) --++ (-1,0) --++ (0,1) --++
      (-1,0) --++ (0,1) --++ (-1,0) --++ (0,1) --++ (-3,0) --cycle;
      \draw (1,0) --++ (0,4) (2,0) --++ (0,4) (3,0)--++(0,3)
      (4,0)--++(0,2) (5,0)--++(0,1) (0,1)--++(5,0) (0,2) coordinate (a) --++(4,0)
      (0,3)--++(3,0);
      \node at (.5,.5) {\ft $4$};
      \node at (1.5,.5) {\ft $4$};
      \node at (2.5,.5) {\ft $3$};
      \node at (3.5,.5) {\ft $3$};
      \node at (4.5,.5) {\ft $2$};
      \node at (5.5,.5) {\ft $2$};
      \node at (.5,1.5) {\ft $4$};
      \node at (1.5,1.5) {\ft $4$};
      \node at (2.5,1.5) {\ft $3$};
      \node at (3.5,1.5) {\ft $3$};
      \node at (4.5,1.5) {\ft $1$};
      \node at (.5,2.5) {\ft $4$};
      \node at (1.5,2.5) {\ft $4$};
      \node at (2.5,2.5) {\ft $1$};
      \node at (3.5,2.5) {\ft $2$};
      \node at (.5,3.5) {\ft $1$};
      \node at (1.5,3.5) {\ft $2$};
      \node at (2.5,3.5) {\ft $3$};
    \end{scope}
   \node[left] at (a) {\ft $T=$};
  \end{tikzpicture}
\end{center}

Restricting to the 3-subsets of $[4]$ gives us
the following tableaux.

\vspace{.2cm}

\begin{center}
  \begin{tikzpicture}
    \begin{scope}[scale = .4]
      \draw (0,0) --++ (4,0) --++ (0,1) --++(-1,0) --++(0,1)
      --++(-1,0) --++(0,1) --++(-2,0) --cycle;
      \draw (1,0) --++ (0,3) (2,0) --++ (0,2) (3,0) --++ (0,1) (0,1)
      --++ (3,0) (0,2) --++ (2,0); 
      \node at (.5,.5) {\ft $3$};
      \node at (1.5,.5) {\ft $3$};
      \node at (2.5,.5) {\ft $2$};
      \node at (3.5,.5) {\ft $2$};
      \node (a) at (.5,1.5) {\ft $3$};
      \node at (1.5,1.5) {\ft $3$};
      \node at (2.5,1.5) {\ft $1$};
      \node at (.5,2.5) {\ft $1$};
      \node at (1.5,2.5) {\ft $2$};
    \end{scope}
   \node[left,xshift=-.2cm] at (a) {\ft $T_{_{\{1,2,3\}}}=$};
    \begin{scope}[scale = .4, xshift = 8cm]
      \draw (0,0) --++ (4,0) --++ (0,1) --++(-1,0) --++(0,1)
      --++(-1,0) --++(0,1) --++(-2,0) --cycle;
      \draw (1,0) --++ (0,3) (2,0) --++ (0,2) (3,0) --++ (0,1) (0,1)
      --++ (3,0) (0,2) --++ (2,0); 
      \node at (.5,.5) {\ft $4$};
      \node at (1.5,.5) {\ft $4$};
      \node at (2.5,.5) {\ft $2$};
      \node at (3.5,.5) {\ft $2$};
      \node (a) at (.5,1.5) {\ft $4$};
      \node at (1.5,1.5) {\ft $4$};
      \node at (2.5,1.5) {\ft $1$};
      \node at (.5,2.5) {\ft $1$};
      \node at (1.5,2.5) {\ft $2$};
    \end{scope}
   \node[left,xshift=-.2cm] at (a) {\ft $T_{_{\{1,2,4\}}}=$};
    \begin{scope}[scale = .4, xshift = 16cm]
      \draw (0,0) --++ (4,0) --++ (0,1) --++(-1,0) --++(0,1)
      --++(-1,0) --++(0,1) --++(-2,0) --cycle;
      \draw (1,0) --++ (0,3) (2,0) --++ (0,2) (3,0) --++ (0,1) (0,1)
      --++ (3,0) (0,2) --++ (2,0); 
      \node at (.5,.5) {\ft $4$};
      \node at (1.5,.5) {\ft $4$};
      \node at (2.5,.5) {\ft $3$};
      \node at (3.5,.5) {\ft $3$};
      \node (a) at (.5,1.5) {\ft $4$};
      \node at (1.5,1.5) {\ft $4$};
      \node at (2.5,1.5) {\ft $1$};
      \node at (.5,2.5) {\ft $1$};
      \node at (1.5,2.5) {\ft $3$};
    \end{scope}
   \node[left,xshift=-.2cm] at (a) {\ft $T_{_{\{1,3,4\}}}=$};
    \begin{scope}[scale = .4, xshift = 24cm]
      \draw (0,0) --++ (4,0) --++ (0,1) --++(-1,0) --++(0,1)
      --++(-1,0) --++(0,1) --++(-2,0) --cycle;
      \draw (1,0) --++ (0,3) (2,0) --++ (0,2) (3,0) --++ (0,1) (0,1)
      --++ (3,0) (0,2) --++ (2,0); 
      \node at (.5,.5) {\ft $4$};
      \node at (1.5,.5) {\ft $4$};
      \node at (2.5,.5) {\ft $3$};
      \node at (3.5,.5) {\ft $3$};
      \node (a) at (.5,1.5) {\ft $4$};
      \node at (1.5,1.5) {\ft $4$};
      \node at (2.5,1.5) {\ft $2$};
      \node at (.5,2.5) {\ft $2$};
      \node at (1.5,2.5) {\ft $3$};
    \end{scope}
   \node[left,xshift=-.2cm] at (a) {\ft $T_{_{\{2,3,4\}}}=$};
  \end{tikzpicture}
\end{center}

\end{example}

A regular tableau on $A$ is uniquely determined by $A$ and its restriction to any three elements.

\begin{lemma}\label{3tab-unique}
If $T$ and $U$ are regular tableau on $[n]$ such that $T_{_{\{1,2,3\}}} = U_{_{\{1,2,3\}}}$, 
then $T = U$. 
\end{lemma}

\begin{proof}
Suppose the lemma fails.  Then there is some $i\in [n]$ such that $T(i) \neq U(i)$.
Consider the first entry where they differ; say the $m$'th entry of $T(i)$ is $j$ and the $m$'th entry of $U(i)$ is $k$. 
As in the proof of Lemma~\ref{2seq-unique}, $T(i)_{_{\{j,k\}}} \neq U(i)_{_{\{j,k\}}}$, so $T_{_{\{i,j,k\}}} \neq U_{_{\{i,j,k\}}}$.  
But this is impossible, since these are both the unique tableau on $\{i,j,k\}$ which is order equivalent to $ T_{_{\{1,2,3\}}}$.
\end{proof}

A sequence 
$(s_1, s_2, \dots, s_n) \in \mb{N}^\star$ is a \df{refinement} of $(t_1, t_2, \dots,t_m)$ 
when there exist integers $i_j$ such that 
\[ (t_1,\;t_2,\;\dots)\; =\; (s_1+\dots+s_{i_1},\ s_{i_1+1}+\dots+s_{i_2},\ \dots). \]
The \df{exponent sequence} of a letter $a$ in a word $\omega$ is the sequence of sizes of consecutive occurrences of $a$ in $\omega$, and is denoted by $\exp_a(\omega)$.  
For instance, $\exp_a(a^2bac^3a^5) = (2,1,5)$.

\begin{prop}\label{3tab-bal-ref}
For any $n\geq 4$ and $U \in [3]^\tabox$, there is a regular tableau $T$ on $[n]$ with $T_{_{\{1,2,3\}}} = U$
if and only if the following hold. 
\begin{enumerate}
\item Rows $U(1)$ and $U(3)$ are balanced.
\item The block sizes of $U(1)$ are a refinement of $\exp_3(U(2))$.
\item The block sizes of $U(3)$ are a refinement of $\exp_1(U(2))$.
\end{enumerate}
\end{prop}

The proof will provide a correspondence analogous to that of Corollary~\ref{2reg-cor}. Let $\{a,b,c,d\}$ be an ordered alphabet with $a\prec b \prec c \prec d$, and let $B_{_{(ab,cd)}} = B_{_{(ab)}} \cup B_{_{(cd)}}$. 
Roughly speaking, the first and the last rows of a regular tableau are regular sequences corresponding to words $\gamma \in B\rlap{$^\star$}_{_{(cd)}}$ and  $\alpha \in B\rlap{$^\star$}_{_{(ab)}}$, respectively. The intermediate rows are an interpolation between the first and the last rows described by interlacing the balanced blocks of $\alpha$ and $\gamma$.

\begin{corollary}\label{tab-corr1}
Let $n\geq 4$ be fixed. The set of regular tableaux on $[n]$ is in bijective correspondence with 
the language $B\rlap{$^\star$}_{_{(ab,cd)}}$. 
\end{corollary}

\begin{proof}[Proof of Proposition \ref{3tab-bal-ref}]
We define a function $\phi_{_n} \colon B\rlap{$^\star$}_{_{(ab,cd)}} \to [n]^\tabox$ such that the range of $\phi_{_n}$ consists of all regular tableaux on $[n]$.  
The rows the tableau are given by $\phi_{_{i,n}}\colon B\rlap{$^\star$}_{_{(ab,cd)}} \to [n]^\star$; that is,  
\[\phi_{_n}(\omega) = (\phi_{_{1,n}}(\omega),\phi_{_{2,n}}(\omega),\dots,\phi_{_{n,n}}(\omega)).\]
We first define $\phi_{_{i,n}}$ on balanced blocks as
\[\begin{array}{rclll}
  \phi_{_{i,n}}(a^rb^r) & \defeq & 
  (\underset{^{\times r}}{1, \cdots, 1}, \cdots, 
  \underset{^{\times r}}{i-1, \cdots, i-1})  &\text{for $i>1$,} &
  \phi_{_{1,n}}(a^rb^r) = \emptyset\\
  \phi_{_{i,n}}(b^ra^r) & \defeq & 
  (\underset{^{\times r}}{i-1, \cdots, i-1}, \cdots,
  \underset{^{\times r}}{1, \cdots, 1})  &\text{for $i>1$,} &
  \phi_{_{1,n}}(b^ra^r) = \emptyset \\ 
  \phi_{_{i,n}}(c^rd^r) & \defeq & 
  (\underset{^{\times r}}{i+1, \cdots, i+1}, \cdots, 
  \underset{^{\times r}}{n, \cdots, n})  &\text{for $i<n$,} &
  \phi_{_{n,n}}(c^rd^r) = \emptyset \\ 
  \phi_{_{i,n}}(d^rc^r) & \defeq & 
  (\underset{^{\times r}}{n, \cdots, n}, \cdots,
  \underset{^{\times r}}{i+1, \cdots, i+1})  &\text{for $i<n$,} &
  \phi_{_{n,n}}(d^rc^r) = \emptyset. 
\end{array}\]
Extend $\phi_{_{i,n}}$ to arbitrary words 
$\omega \in B\rlap{$^\star$}_{_{(ab,cd)}}$ by concatenation in the following way: 
For $\omega = \omega_1 \cdots \omega_k$ where 
$\omega_j \in B_{_{(ab,cd)}}$, let 
$\phi_{_{i,n}}(\omega)  \defeq \phi_{_{i,n}}(\omega_1)  \cdots \phi_{_{i,n}}(\omega_k)$. Since the
functions $\phi_{_{i,n}}$ are injective, the function $\phi_{_n}$ is also injective. 
Let $T^\omega = \phi_{_n}(\omega)$.
To see that the tableau is regular, let $X = \{i,j,k\}$ with 
$1\leq i < j < k \leq n$.  
Now we have $T^\omega(i)_{_X} \sim \omega_{_{\{c,d\}}}$, 
$T^\omega(j)_{_X} \sim \omega_{_{\{a,c\}}}$, and 
$T^\omega(k)_{_X} \sim \omega_{_{\{a,b\}}}$.

Let $U \in [3]^\tabox$ satisfy conditions {\em (1), (2)} and {\em (3)}. 
We will define $\omega$ such that $T\rlap{$^\omega$}_{_{\{1,2,3\}}}=U$
by interlacing the balanced blocks of words in $B\rlap{$^\star$}_{_{(cd)}}$ and $B\rlap{$^\star$}_{_{(ab)}}$ corresponding to $U(1)$ and $U(3)$ according to the refinements of the exponent sequences of $U(2)$. Let 
\[U(2) = 
(\underset{^{\times p_1}}{1, \cdots, 1},
\underset{^{\times q_1}}{3, \cdots, 3},
\underset{^{\times p_2}}{1, \cdots, 1},
\underset{^{\times q_2}}{3, \cdots, 3}, \cdots)\] 
where $p_1 \geq 0$ and any subsequent defined terms $q_i$ or $p_i$ are positive. 
Let $\alpha \in \{a,b\}^\star$ be order equivalent to $U(3)$, and $\gamma \in \{c,d\}^\star$ be order equivalent to $U(1)$. 
Let $\alpha = \alpha_1 \alpha_2 \cdots$ and $\gamma = \gamma_1 \gamma_2 \cdots$ such that 
$|\alpha_i| = 2p_i$ and $|\gamma_i| = 2q_i$, and let $\omega = \alpha_1 \gamma_1 \alpha_2 \gamma_2 \cdots$.

We first show $\omega \in B\rlap{$^\star$}_{_{(ab,cd)}}$.  
The block sizes of $\gamma \sim U(1)$ are a common refinement of $(q_1,q_2,\dots) = \exp_3(S_2)$.  This implies that the words $\gamma_i$ are balanced, and likewise the $\alpha_i$ are balanced, so $\omega \in B\rlap{$^\star$}_{_{(ab,cd)}}$. 
We now show $T\rlap{$^\omega$}_{_{\{1,2,3\}}}=U$. 
By construction, $\omega_{_{\{c,d\}}} = \gamma \sim U(1)$, so $T^\omega(1)_{_{\{2,3\}}} = U(1)$, 
and similarly $T^\omega(3)_{_{\{1,2\}}} = U(3)$. 
Lastly, $\omega_{_{\{a,c\}}} \sim U(2)$, since $|\alpha_{i_{\{a\}}}| = p_i$ and $|\gamma_{i_{\{c\}}}| = q_i$, 
so $T^\omega(2)_{_{\{1,3\}}} = U(2)$.
Thus, we have a regular tableau $T^\omega$ such that $T\rlap{$^\omega$}_{_{\{1,2,3\}}}=U$.

For the other direction let $T$ be a regular tableau on $[n]$, and let $U = T_{_{\{1,2,3\}}}$.

We first show that $U(1)$ and $U(3)$ are balanced.
We claim that $T(1)$ and $T(n)$ are regular sequences on $[n]_1^+$ and $[n]_n^-$, respectively.
For any $X = \{1,i, j\}$ and $Y =\{1,i',j'\}$, the regularity of $T$ implies that
$T(1)_{_X} \sim T(1)_{_Y}$. Therefore $T(1)$ is regular on $[n]_1^+$. The same 
reasoning, restricting to sets $X = \{i,j,n\}$ and $Y = \{i',j',n\}$, shows $T(n)$ is regular. 
Now, by Proposition~\ref{2reg}, $U(1) = T(1)_{_{\{2,3\}}}$ is balanced, since $T(1)$ is regular, and likewise $U(3) = T(n)_{_{\{1,2\}}}$ is balanced, since $T(n)$ is regular.

Now we show that the refinements hold. 
By definition $n \geq 4$. 
Since $T_{_{\{1,2,3\}}} \sim T_{_{\{1,2,4\}}}$, 
we can let 
\[ (t_1,t_2,\dots,t_m) = \exp_3(T(2)_{_{\{1,3\}}}) = \exp_4(T(2)_{_{\{1,4\}}}), \]
We decompose $T(2)_{_{\{3,4\}}}$ into subsequences appearing between consecutive occurrences of 1 in $T(2)_{_{\{1,3,4\}}}$; that is, 
let $T(2)_{_{\{3,4\}}} = \omega_1 \cdot \omega_2 \cdots \omega_m$ such that
\[ T(2)_{_{\{1,3,4\}}} = (\dots,1,\omega_i,1,\dots,1,\omega_{i+1},1,\dots). \]
Since $T_{_{\{2,3,4\}}} \sim U$, $T(2)_{_{\{3,4\}}} \sim U(1)$ is balanced, 
and since $|\omega_{i_{\{3\}}}| = |\omega_{i_{\{4\}}}| =t_i$ for all $i$, 
each subsequence $\omega_i$ is balanced by Corollary~\ref{bal-concat}. 
Let $(s_{(i,1)},\dots,s_{(i,k_i)})$ be the block sizes of $\omega_i$. 
The block sizes of $U(1)$ are $(s_{(1,1)}.\dots,s_{(1,k_1)},s_{(2,1)},\dots,s_{(m,k_m)})$, and $s_{(i,1)}+\dots+s_{(i,k_i)} = t_i$. 
Therefore, the block sizes of $U(1)$ are a refinement of $\exp_3(U(2))$, and likewise for $U(2)$ and $\exp_1(U(2))$.
\end{proof}

\begin{example}
  Consider the word $(ab)(d^2c^2)(ba) \in B\rlap{$^\star$}_{_{(ab,cd)}}$. For $n=4$,  we get the following tableau.

\vspace{.2cm}

\begin{center}
  \begin{tikzpicture}
    \begin{scope}[scale = .4]
      \draw (0,0) --++ (6,0) --++ (0,4) --++ (-6,0) --cycle;
      \draw (1,0) --++ (0,4) (2,0) --++ (0,4) (3,0) --++ (0,4) (4,0)
      --++ (0,4) (5,0) --++ (0,4); 
      \draw (0,1) --++ (6,0) (0,2) --++ (6,0) (0,3) --++ (6,0);
      \node at (.5,.5) {\ft $4$};
      \node at (1.5,.5) {\ft $4$};
      \node at (2.5,.5) {\ft $3$};
      \node at (3.5,.5) {\ft $3$};
      \node at (4.5,.5) {\ft $2$};
      \node at (5.5,.5) {\ft $2$};
      \node at (.5,1.5) {\ft $1$};
      \node at (1.5,1.5) {\ft $4$};
      \node at (2.5,1.5) {\ft $4$};
      \node at (3.5,1.5) {\ft $3$};
      \node at (4.5,1.5) {\ft $3$};
      \node at (5.5,1.5) {\ft $1$};
      \node at (.5,2.5) {\ft $1$};
      \node at (1.5,2.5) {\ft $2$};
      \node at (2.5,2.5) {\ft $4$};
      \node at (3.5,2.5) {\ft $4$};
      \node at (4.5,2.5) {\ft $2$};
      \node at (5.5,2.5) {\ft $1$};
      \node at (.5,3.5) {\ft $1$};
      \node at (1.5,3.5) {\ft $2$};
      \node at (2.5,3.5) {\ft $3$};
      \node at (3.5,3.5) {\ft $3$};
      \node at (4.5,3.5) {\ft $2$};
      \node at (5.5,3.5) {\ft $1$};
    \end{scope}
  \end{tikzpicture}
\end{center}

Restricting to the 3-subsets of $[4]$ gives
us the following tableaux.

\vspace{.2cm}

\begin{center}
  \begin{tikzpicture}
    \begin{scope}[scale = .4]
      \draw (0,0) --++ (4,0) --++ (0,3) --++ (-4,0) --cycle;
      \draw (1,0) --++ (0,3) (2,0) --++ (0,3) (3,0) --++ (0,3) (0,1)
      --++ (4,0) (0,2) --++ (4,0); 
      \node at (.5,.5) {\ft $3$};
      \node at (1.5,.5) {\ft $3$};
      \node at (2.5,.5) {\ft $2$};
      \node at (3.5,.5) {\ft $2$};
      \node at (.5,1.5) {\ft $1$};
      \node at (1.5,1.5) {\ft $3$};
      \node at (2.5,1.5) {\ft $3$};
      \node at (3.5,1.5) {\ft $1$};
      \node at (.5,2.5) {\ft $1$};
      \node at (1.5,2.5) {\ft $2$};
      \node at (2.5,2.5) {\ft $2$};
      \node at (3.5,2.5) {\ft $1$};
    \end{scope}
    \begin{scope}[scale = .4, xshift = 8cm]
      \draw (0,0) --++ (4,0) --++ (0,3) --++ (-4,0) --cycle;
      \draw (1,0) --++ (0,3) (2,0) --++ (0,3) (3,0) --++ (0,3) (0,1)
      --++ (4,0) (0,2) --++ (4,0); 
      \node at (.5,.5) {\ft $4$};
      \node at (1.5,.5) {\ft $4$};
      \node at (2.5,.5) {\ft $2$};
      \node at (3.5,.5) {\ft $2$};
      \node at (.5,1.5) {\ft $1$};
      \node at (1.5,1.5) {\ft $4$};
      \node at (2.5,1.5) {\ft $4$};
      \node at (3.5,1.5) {\ft $1$};
      \node at (.5,2.5) {\ft $1$};
      \node at (1.5,2.5) {\ft $2$};
      \node at (2.5,2.5) {\ft $2$};
      \node at (3.5,2.5) {\ft $1$};
    \end{scope}
    \begin{scope}[scale = .4, xshift = 16cm]
      \draw (0,0) --++ (4,0) --++ (0,3) --++ (-4,0) --cycle;
      \draw (1,0) --++ (0,3) (2,0) --++ (0,3) (3,0) --++ (0,3) (0,1)
      --++ (4,0) (0,2) --++ (4,0); 
      \node at (.5,.5) {\ft $4$};
      \node at (1.5,.5) {\ft $4$};
      \node at (2.5,.5) {\ft $3$};
      \node at (3.5,.5) {\ft $3$};
      \node at (.5,1.5) {\ft $1$};
      \node at (1.5,1.5) {\ft $4$};
      \node at (2.5,1.5) {\ft $4$};
      \node at (3.5,1.5) {\ft $1$};
      \node at (.5,2.5) {\ft $1$};
      \node at (1.5,2.5) {\ft $3$};
      \node at (2.5,2.5) {\ft $3$};
      \node at (3.5,2.5) {\ft $1$};
    \end{scope}
    \begin{scope}[scale = .4, xshift = 24cm]
      \draw (0,0) --++ (4,0) --++ (0,3) --++ (-4,0) --cycle;
      \draw (1,0) --++ (0,3) (2,0) --++ (0,3) (3,0) --++ (0,3) (0,1)
      --++ (4,0) (0,2) --++ (4,0); 
      \node at (.5,.5) {\ft $4$};
      \node at (1.5,.5) {\ft $4$};
      \node at (2.5,.5) {\ft $3$};
      \node at (3.5,.5) {\ft $3$};
      \node at (.5,1.5) {\ft $2$};
      \node at (1.5,1.5) {\ft $4$};
      \node at (2.5,1.5) {\ft $4$};
      \node at (3.5,1.5) {\ft $2$};
      \node at (.5,2.5) {\ft $2$};
      \node at (1.5,2.5) {\ft $3$};
      \node at (2.5,2.5) {\ft $3$};
      \node at (3.5,2.5) {\ft $2$};
    \end{scope}
  \end{tikzpicture}
\end{center}
\end{example}

\bigskip

Let $T^{\omega}_n \defeq \phi_n(\omega)$ denote the regular tableau on $[n]$ corresponding to $\omega \in B\rlap{$^\star$}_{_{(ab,cd)}}$. Note that $T^\omega_n$ is well-defined for all $n\geq 3$, and that $T^\omega_n$ is regular for all $n\geq 4$. Notice that $T_5^\omega(3) \sim \omega$. We point out two further consequences of the proof of Proposition~\ref{3tab-bal-ref}. 

\begin{corollary} \label{reg-line}
If $T = T^\omega_n$ is a regular tableau and $X$ is a $3$-subset of $[n]$, then $T_{_X} \sim T^\omega_3$. 
\end{corollary}

\begin{corollary}\label{tabconc}
If $T^{\omega_1}_n$ and $T^{\omega_2}_n$ are regular tableaux, then
$T^{\omega_1}_n \cdot T^{\omega_2}_n = T^{\omega_1\cdot\omega_2}_n$. 
\end{corollary}

A tableau $T \in A^\tabox$ is called \df{pangrammatic} when every $a_i
\in A$ occurs at least once in $T$. In the following sections all
regular tableau are assumed to be pangrammatic. Let $W_{_{(ab,cd)}}$
be the language of words $\omega \in B\rlap{$^\star$}_{_{(ab,cd)}}$
such that $\omega \notin B\rlap{$^\star$}_{_{(ab)}}$ and $\omega
\notin B\rlap{$^\star$}_{_{(cd)}}$. 

\begin{prop}\label{tab-corr}
Let $n\geq 4$ be fixed. The set of pangrammatic regular tableaux on
$[n]$ is in bijective correspondence with the language $W_{_{(ab,cd)}}$. 
\end{prop}

\begin{proof}
Observe from the definition of $\phi_{_{n}}$ in the proof of
Proposition~\ref{3tab-bal-ref} that for $\omega \in B_{_{(ab)}}$, the
integers occurring in $T^\omega_n$ are $\{1,\dots,n-1\}$, and for
$\omega \in B_{_{(cd)}}$, the integers occurring in $T^\omega_n$ are
$\{2,\dots,n\}$.  
\end{proof}

\section{Geometric tableaux} \label{sec:GeomTab}

\subsection{Tableaux and local sequences}
Let $S$ be a system labeled by $[n]$, and let $\omega_i$ denote the
local sequence of path $i$. We may associate the local sequences of
$S$ with a tableau $T = (\omega_1, \dots, \omega_n)$ on $[n]$. We call
$T$ the tableau \df{associated} to $S$. Clearly $T \in
[n]^\tabox$. For us, the converse direction is important. That is,
when does a tableau on $[n]$ correspond to the local sequences of a
system of paths? 

\begin{define}
  Let $T \in [n]^\tabox$. We say $T$ is \df{geometric} if it is the
  associated tableau of a system. 
\end{define}

We can decide if a tableau is geometric by a simple algorithm.

\begin{algo} \label{bumpalgo}
The input is a tableau $T$ on $[n]$. The output is {\tt True} if $T$
is geometric and {\tt False} otherwise.  

\begin{algorithmic}
\STATE $T' \leftarrow T$
\STATE $\pi\hspace{4pt} \leftarrow (1,\dots,n)$. 
\WHILE{ $T' \neq (\emptyset,\dots,\emptyset)$ }
\IF{
There exists a lexicographically minimal pair of integers $(j,k)$ that
satisfy the following. 
    \begin{itemize}
    \item \hspace{-12pt} $j$ and $k$ are adjacent in $\pi$.
    \item \hspace{-12pt} $T'(j) = (k) \cdot \omega_j $.
    \item \hspace{-12pt} $T'(k) = (j) \cdot \omega_k $.
    \end{itemize}
}
\STATE Transpose elements $j$ and $k$ in $\pi$.
\STATE $T'(j) \leftarrow \omega_j$.
\STATE $T'(k) \leftarrow \omega_k$.
\ELSE 
\STATE Return {\tt False}.
\ENDIF
\ENDWHILE
\STATE Return {\tt True}.
\end{algorithmic}
\end{algo}

\begin{prop} \label{determine}
A tableau $T$ on $[n]$ is geometric if and only if Algorithm
\ref{bumpalgo} returns {\tt True} on $T$. 
\end{prop}

\begin{proof}  Consider a system $S$ ordered by $[n]$ and let $T$ be
  the associated tableau. We show that applying Algorithm
  \ref{bumpalgo} to $T$ corresponds to a ``topological sweep'' of
  $S$.\footnote{This is a method from computational geometry
    introduced by Edelsbrunner and Guibas \cite{edguib}.} Define a
  \df{cut path} of $S$ to be a continuous path contained in the strip
  $[0,1] \times \mb{R}$ that   
  \begin{itemize}
  \item starts on the line $\{0\}\times \mb{R}$ below the left
    endpoints of the paths  of $S$,
  \item ends on the line $\{1\}\times \mb{R}$ above the right
    endpoints of the paths of $S$, and
  \item intersects each path of $S$ at a unique point in $(0,1) \times
    \mb{R}$. 
  \end{itemize} 
  We do not require a cut path to be the graph of a function $f \colon
  [0,1] \to \mb{R}$, but the conditions imply that it intersects every
  path in $S$ transversally. Moreover, a cut path divides
  $[0,1]\times\mb{R}$ into two connected components that respectively
  contain the left and right endpoints of $S$.  For a given cut path
  $\gamma$, we say a point is right of $\gamma$ when it is in the
  component containing the right endpoints of $S$, and we say it is
  left of $\gamma$ when it is in the other component. If we extend the
  local sequences of $S$ to include $\gamma$, then each $T(i)$ will
  include an additional entry, $\gamma$, and the local sequence of
  $\gamma$ will be a permutation of $[n]$. Let $T^{^\gamma}$ be the
  tableau on $[n]$ obtained from $T$ by deleting the prefix that ends
  in $\gamma$ from each $T(i)$, and let $\pi^\gamma$ be the local
  sequence of $\gamma$. 
  
  Suppose $T^{^\gamma} \neq (\emptyset, \dots, \emptyset)$. We claim
  that there exists a pair of integers $j,k$ such that $T^{^\gamma}(j)
  = (k,\dots)$ and $T^{^\gamma}(k) = (j,\dots)$. We call such a pair
  \df{adjacent} with respect to $\gamma$. Note that $j,k$ must be
  consecutive in $\pi^\gamma$. To see why such a pair exists, let $p$
  be the leftmost point among the crossings of $S$ that are right of
  $\gamma$, and let $j$ and $k$ be the paths crossing at the point
  $p$. This means $T^\gamma(j) \neq \emptyset$, and by
  $x$-monotonicity, if $T^\gamma(j) = (i,\dots)$ for some $i \neq k$,
  then paths $j$ and $i$ would cross at a point right of $\gamma$ and
  left of $p$, contradicting our choice of $p$.  Therefore
  $T^\gamma(j) = (k,\dots)$, and similarly $T^\gamma(k) = (j,\dots)$,
  so $j$ and $k$ are adjacent with respect to $\gamma$.

  Suppose $\{j,k\}$ are adjacent with respect to $\gamma$ with
  $T^{^\gamma}(j) = (k)\cdot \omega_j$ and $T^{^\gamma}(k)=(j)\cdot
  \omega_k$. Then, we can perturb $\gamma$ to obtain a new cut path
  $\delta$ such that $T^{^\delta}(j) = \omega_j$ and $T^{^\delta}(k) =
  \omega_k$ while $T^{^\delta}(i) = T^{^\gamma}(i)$ for all $i\notin
  \{j,k\}$. The local sequence of $\delta$ is obtained from $\pi$ by
  transposing the elements $j$ and $k$. Call such a modification of
  the cut path a \df{$(j,k)$-sweep}. 

  
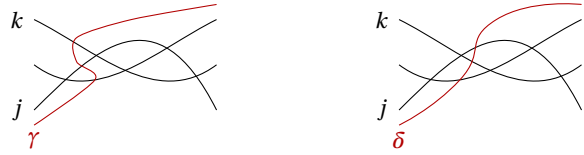
\begin{center}
    \begin{tikzpicture}
      \begin{scope}[scale=.8]
        \draw (0,0) coordinate (a) ..controls (1,1) and (2,2) .. (3,0);
        \draw (0,.75) ..controls (1,0) and (2,1) .. (3,1.5);
        \draw (0,1.5)  coordinate (b) ..controls (1,1) and (2,0) .. (3,.75);
        \draw[red!70!black] plot [smooth] coordinates
        { (0,-.25) (1,.5)  (.7,.8) (.7,1.2)  (1.6,1.5)  (3,1.75)}; 
        \node [red!70!black] at (0,-.5) {\ft $\gamma$};
      \end{scope}
\node[left] at (a) {\ft $j$};
\node[left] at (b) {\ft $k$};
      \begin{scope}[scale=.8, xshift = 6cm]
        \draw (0,0) coordinate (a) ..controls (1,1) and (2,2) .. (3,0);
        \draw (0,.75) ..controls (1,0) and (2,1) .. (3,1.5);
        \draw (0,1.5)  coordinate (b) ..controls (1,1) and (2,0) .. (3,.75);
        \draw[red!70!black] plot [smooth, tension = 1] coordinates
        { (0,-.25) (1,.5) (1.6,1.5)  (3,1.75)}; 
        \node [red!70!black] at (0,-.5) {\ft $\delta$};
      \end{scope}
\node[left] at (a) {\ft $j$};
\node[left] at (b) {\ft $k$};
    \end{tikzpicture}
  \end{center}

  To see why Algorithm \ref{bumpalgo} returns {\tt True} on the geometric tableau $T$ associated to $S$, start
  with a cut path $\gamma_0$ such that all crossings of $S$ are right of $\gamma_0$.  This gives us a tableau
  $T^{^{\gamma_0}} = T$ and a permutation $\pi^{\gamma_0} = (1, \dots, n)$. 
  The algorithm corresponds to a sequence of sweeps starting from $\gamma_0$ in the following way. 
  Let $(j,k)$ be the lexicographically minimal pair that is adjacent with respect
  to cut path $\gamma_t$, and perform a $(j,k)$-sweep. This produces a new
  cut path $\gamma_{t+1}$ with local sequence $\pi^{\gamma_{t+1}}$ and a new tableau
  $T^{^{\gamma_{t+1}}}$. 
  By the observations above, $T^{^{\gamma_{t+1}}}$ and $\pi^{\gamma_{t+1}}$ are obtained from $T^{^{\gamma_{t}}}$ and $\pi^{\gamma_{t}}$ according to a step of Algorithm \ref{bumpalgo}, so by induction on $t$, $\pi = \pi^{\gamma_t}$ and $T' = T^{\gamma_t}$
  at step $t$.  This procedure can be repeated until all crossings of $S$ are left of a cut path $\gamma_m$, which gives $T^{^{\gamma_m}} = (\emptyset, \dots, \emptyset)$. Thus, the algorithm returns {\tt True}. 

  \bigskip

  For the converse direction, consider a tableau $T$ on which Algorithm
  \ref{bumpalgo} returns {\tt True}.  We show that
  $T$ is geometric by constructing a system of paths
having $T$ as its associated tableau.\footnote{This is essentially the
  same as the ``wiring diagram'' construction
  introduced by Goodman \cite{goodburr} as a canonical way of drawing
  pseudoline arrangements, i.e. {1-crossing} systems.} Let
  $(\pi_0,\dots,\pi_m)$ be the sequence 
  of permutations produced by the algorithm. 
To construct the wiring diagram, start with distinct points on the line $x=0$ labeled by $[n]$ from bottom to top.  Point $i$ will be the left endpoint of path $i$.  At each step of Algorithm \ref{bumpalgo}, extend the paths to the right by crossing paths $j,k$ and letting the remaining paths continue horizontally. Notice that at step $t$ a vertical line to the immediate right of this crossing meets the paths in the order of $\pi_t$ from bottom to top.  Furthermore, path $i$ intersects the remaining paths in the order of $T(i)$ from left to right.  Hence, $T$ is the tableau associate with $S$.  To make the construction explicit, define path $i$ as the polygonal path with vertices $p_{i,1},\dots,p_{i,m}$ where $p_{i,t} = m^{-1}(t,h)$ and $i = \pi_t(h)$.
\end{proof}  

\begin{example}
Below is the state at each step of Algorithm \ref{bumpalgo} on a tableau $T$. The shaded
boxes indicate the adjacent pair at the current step.

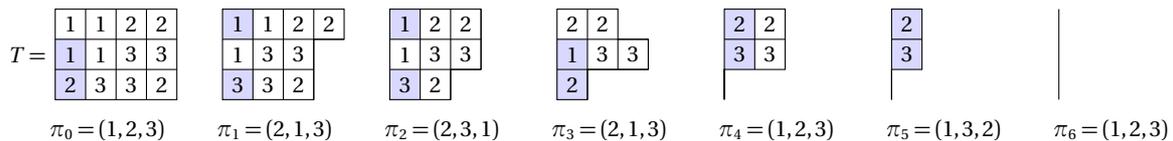
\vspace{.2cm}

\begin{center}
  \begin{tikzpicture}
    \begin{scope}[scale = .4]
      \fill[blue, opacity = .15] (0,0)--++(1,0)--++(0,1)--++(-1,0)--cycle;
      \fill[blue, opacity = .15] (0,1)--++(1,0)--++(0,1)--++(-1,0)--cycle;
      
      \draw (0,0) --++ (4,0) --++ (0,3) --++ (-4,0) --cycle;
      \draw (1,0)--++ (0,3) (2,0)--++ (0,3) (3,0)--++(0,3);
      \draw (0,1)--++(4,0) (0,2)--++(4,0);
      \node at (.5,.5) {\ft $2$};
      \node at (1.5,.5) {\ft $3$};
      \node at (2.5,.5) {\ft $3$};
      \node at (3.5,.5) {\ft $2$};
      \node (a) at (.5,1.5) {\ft $1$};
      \node at (1.5,1.5) {\ft $1$};
      \node at (2.5,1.5) {\ft $3$};
      \node at (3.5,1.5) {\ft $3$};
      \node at (.5,2.5) {\ft $1$};
      \node at (1.5,2.5) {\ft $1$};
      \node at (2.5,2.5) {\ft $2$};
      \node at (3.5,2.5) {\ft $2$};
      \node[right] at (-.5,-1) {\ft $\pi_0 = (1,2,3)$};
    \end{scope}
      \node[left,xshift=-.2cm] at (a) {\ft $T=$};

    \begin{scope}[scale = .4, xshift = 5.5cm]
      \fill[blue, opacity = .15] (0,0)--++(1,0)--++(0,1)--++(-1,0)--cycle;
      \fill[blue, opacity = .15] (0,2)--++(1,0)--++(0,1)--++(-1,0)--cycle;

      \draw (0,0) --++ (3,0) --++ (0,2) --++ (1,0) --++(0,1)
      --++(-4,0) --cycle; 
      \draw (1,0)--++ (0,3) (2,0)--++ (0,3) (3,0)--++(0,3);
      \draw (0,1)--++(3,0) (0,2)--++(4,0);
      \node at (.5,.5) {\ft $3$};
      \node at (1.5,.5) {\ft $3$};
      \node at (2.5,.5) {\ft $2$};
      \node at (.5,1.5) {\ft $1$};
      \node at (1.5,1.5) {\ft $3$};
      \node at (2.5,1.5) {\ft $3$};
      \node at (.5,2.5) {\ft $1$};
      \node at (1.5,2.5) {\ft $1$};
      \node at (2.5,2.5) {\ft $2$};
      \node at (3.5,2.5) {\ft $2$};
      \node[right] at (-.5,-1) {\ft $\pi_1 = (2,1,3)$};
    \end{scope}

    \begin{scope}[scale = .4, xshift = 11cm]
      \fill[blue, opacity = .15] (0,0)--++(1,0)--++(0,1)--++(-1,0)--cycle;
      \fill[blue, opacity = .15] (0,2)--++(1,0)--++(0,1)--++(-1,0)--cycle;

      \draw (0,0) --++ (2,0) --++ (0,1) --++ (1,0) --++(0,2)
      --++(-3,0) --cycle; 
      \draw 
      (1,0)--++ (0,3) 
      (2,0)--++ (0,3) 
      (3,1)--++(0,2);
      \draw 
      (0,1)--++(3,0) 
      (0,2)--++(3,0);
      \node at (.5,.5) {\ft $3$};
      \node at (1.5,.5) {\ft $2$};
      \node at (.5,1.5) {\ft $1$};
      \node at (1.5,1.5) {\ft $3$};
      \node at (2.5,1.5) {\ft $3$};
      \node at (.5,2.5) {\ft $1$};
      \node at (1.5,2.5) {\ft $2$};
      \node at (2.5,2.5) {\ft $2$};
      \node[right] at (-.5,-1) {\ft $\pi_2 = (2,3,1)$};
    \end{scope}

    \begin{scope}[scale = .4, xshift = 16.5cm]
      \fill[blue, opacity = .15] (0,0)--++(1,0)--++(0,1)--++(-1,0)--cycle;
      \fill[blue, opacity = .15] (0,1)--++(1,0)--++(0,1)--++(-1,0)--cycle;

      \draw (0,0) --++ (1,0) --++ (0,1) --++ (2,0) --++(0,1)
      --++(-1,0) --++(0,1) --++(-2,0) --cycle;
      \draw 
      (1,0)--++ (0,3) 
      (2,1)--++ (0,1) 
      (3,1)--++(0,1);
      \draw 
      (0,1)--++(3,0) 
      (0,2)--++(3,0);
      \node at (.5,.5) {\ft $2$};
      \node at (.5,1.5) {\ft $1$};
      \node at (1.5,1.5) {\ft $3$};
      \node at (2.5,1.5) {\ft $3$};
      \node at (.5,2.5) {\ft $2$};
      \node at (1.5,2.5) {\ft $2$};
      \node[right] at (-.5,-1) {\ft $\pi_3 = (2,1,3)$};
    \end{scope}

    \begin{scope}[scale = .4, xshift = 22cm]
      \fill[blue, opacity = .15] (0,1)--++(1,0)--++(0,1)--++(-1,0)--cycle;
      \fill[blue, opacity = .15] (0,2)--++(1,0)--++(0,1)--++(-1,0)--cycle;
 
     \draw (0,0) --++ (0,1) --++ (2,0) --++(0,2)
       --++(-2,0) --cycle;
      \draw 
      (1,1)--++ (0,2);
      \draw 
       (0,2)--++(2,0);
      \node at (.5,1.5) {\ft $3$};
      \node at (1.5,1.5) {\ft $3$};
      \node at (.5,2.5) {\ft $2$};
      \node at (1.5,2.5) {\ft $2$};
      \node[right] at (-.5,-1) {\ft $\pi_4 = (1,2,3)$};
    \end{scope}

    \begin{scope}[scale = .4, xshift = 27.5cm]
      \fill[blue, opacity = .15] (0,1)--++(1,0)--++(0,1)--++(-1,0)--cycle;
      \fill[blue, opacity = .15] (0,2)--++(1,0)--++(0,1)--++(-1,0)--cycle;

      \draw (0,0) --++ (0,1) --++ (1,0) --++(0,2)
       --++(-1,0) --cycle;
      \draw 
       (0,2)--++(1,0);
      \node at (.5,1.5) {\ft $3$};
      \node at (.5,2.5) {\ft $2$};
      \node[right] at (-.5,-1) {\ft $\pi_5 = (1,3,2)$};

    \end{scope}

    \begin{scope}[scale = .4, xshift = 33cm]
      \draw (0,0) --++ (0,3);
      \node[right] at (-.5,-1) {\ft $\pi_6 = (1,2,3)$};
    \end{scope}
  \end{tikzpicture}
\end{center}

Since the final tableau is $(\emptyset, \emptyset, \emptyset)$, it 
follows from Proposition \ref{determine} that $T$ is
geometric. Plotting the sequence of permutations $(\pi_0, \dots, \pi_6)$ produces the following system. 

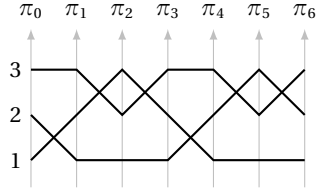
\vspace{.2cm}

\begin{center}
  \begin{tikzpicture}
    \begin{scope}[scale = .6]
      \draw[-latex, gray!50!white](0,-.6) --++ (0,3.5);
      \draw[-latex, gray!50!white](1,-.6) --++ (0,3.5);
      \draw[-latex, gray!50!white](2,-.6) --++ (0,3.5);
      \draw[-latex, gray!50!white](3,-.6) --++ (0,3.5);
      \draw[-latex, gray!50!white](4,-.6) --++ (0,3.5);
      \draw[-latex, gray!50!white](5,-.6) --++ (0,3.5);
      \draw[-latex, gray!50!white](6,-.6) --++ (0,3.5);
      \draw[thick](0,2)\ls\ld\lu\ls\ld\lu;
      \draw[thick](0,1)\ld\ls\ls\lu\lu\ld;
      \draw[thick](0,0)\lu\lu\ld\ld\ls\ls;
      \node[left] at (0,0) {\ft $1$};
      \node[left] at (0,1) {\ft $2$};
      \node[left] at (0,2) {\ft $3$};
      \node at (0,3.3) {\ft $\pi_0$};
      \node at (1,3.3) {\ft $\pi_1$};
      \node at (2,3.3) {\ft $\pi_2$};
      \node at (3,3.3) {\ft $\pi_3$};
      \node at (4,3.3) {\ft $\pi_4$};
      \node at (5,3.3) {\ft $\pi_5$};
      \node at (6,3.3) {\ft $\pi_6$};
    \end{scope}
  \end{tikzpicture}
\end{center}
\end{example}

\subsection{Valid matchings}
We present a similar algorithm which will be
conceptually useful. It basically does the same as Algorithm
\ref{bumpalgo}, except that it disregards the sequence of permutations. 

\begin{algo}\label{algomatch}
The input is a tableau $T$ on $[n]$. The output is {\tt True} or {\tt
  False}. 

\begin{algorithmic}
\STATE $T' \leftarrow T$
\WHILE{ $T' \neq (\emptyset,\dots,\emptyset)$ }
\IF{
There exists a lexicographically minimal pair of integers $(j,k)$ that satisfy the following.
    \begin{itemize}
    \item \hspace{-12pt} $T'(j) = (k) \cdot \omega_j $.
    \item \hspace{-12pt} $T'(k) = (j) \cdot \omega_k $.
    \end{itemize}
}
\STATE $T'(j) \leftarrow \omega_j$.
\STATE $T'(k) \leftarrow \omega_k$.
\ELSE 
\STATE Return {\tt False}.
\ENDIF
\ENDWHILE
\STATE Return {\tt True}.
\end{algorithmic}
\end{algo}

\begin{define}
Given $T \in [n]^\tabox$, we call a pair of integers $\{j,k\}$ \df{weakly adjacent} in $T$
if $T(j) = (k,\dots)$ and $T(k)= (j,\dots)$. 
We say that $T$ has a
\df{valid matching} when the output
of Algorithm \ref{algomatch} on $T$ is {\tt True}.
\end{define}

\begin{prop} \label{val-match}
  For $T \in [n]^\tabox$ with $n\geq 3$, $T$ is geometric if and 
  only if $T$ has a valid matching and the restriction $T_{_X}$ is geometric for every 3-subset
  $X\subset [n]$. 
\end{prop}

\begin{proof} 
One direction is trivial. For the other direction, let $T$ have a
valid matching, let $T_{_X}$ be geometric for all 3-subset $X\subset
[n]$, and suppose $T$ is not geometric. By Proposition
\ref{determine}, Algorithm \ref{bumpalgo} returns {\tt False}, which
means Algorithms \ref{bumpalgo} and \ref{algomatch} diverge. Let
$Y_1,\dots,Y_p$ be the sequence of weakly adjacent pairs at each step
of Algorithm \ref{algomatch}, let $\pi_0,\pi_1,\dots,\pi_q$ be the
sequence of permutations at each step of Algorithm \ref{bumpalgo}
prior to divergence, and let $Y_{q+1}=\{j,k\}$.  For divergence to
occur, $\{j,k\}$ must not be adjacent.  That is, there must be some
integer $i$ appearing between $j$ and $k$ in $\pi_q$.  Consider $U =
T_{\{i,j,k\}}$.    

The subsequence $Y_{t_1},\dots,Y_{t_m},Y_{q+1},\dots$ consisting of all pairs $Y_{t_s} \subset \{i,j,k\}$ defines a valid matching on $U$, and the order of $\{i,j,k\}$ in the permutation $\pi_{t_s}$ is determined by sequentially transposing pairs $Y_{t_1},\dots,Y_{t_s}$. 
For a tableau on three elements, there can be at most one weakly adjacent pair of integers.  Therefore, the sequence of adjacent integers produced by Algorithm \ref{bumpalgo} on $U$ is exactly $Y_{t_1},\dots,Y_{t_m}$ and the sequence of permutations is exactly the order of $\{i,j,k\}$ in $\pi_{t_1},\dots,\pi_{t_m}$, after which there is no adjacent pair and the algorithm returns {\tt False}.  But by Proposition \ref{determine}, Algorithm \ref{bumpalgo} returns {\tt True} on $U$, since $U$ is geometric.  Hence by contradiction, $T$ must be geometric.  
\end{proof}

\subsection{Regular tableaux and regular 
   systems} \label{signs-reg} 

Here we make some observations concerning regular systems and their
associated tableaux. Proposition \ref{factorize} concerns the
concatenation of regular geometric tableaux and is important in the
remaining sections.

\begin{lemma} \label{signa-local}
  Let $S$ be a system labeled by $[3]$ with signature
  $\sigma$ and associated tableau $T$. Then \[T(1) \sim
  \sigma_{_{\{x,y\}}} \; , \; T(2) \sim
  \sigma_{_{\{x,z\}}} \; , T(3) \sim
  \sigma_{_{\{y,z\}}}.\]  
\end{lemma}

\begin{proof} This is just a reformulation of Remark \ref{sign=local}.
\end{proof}

\begin{lemma} \label{system-tableau}
  Let $S$ be a $k$-crossing system and $T$ its associated
  tableau. Then $S$ is regular if and only if $T$ is regular. 
\end{lemma}

\begin{proof} Suppose $S$ is labeled by $[n]$. $T$ is regular if and
  only if for every 3-subset $X\subset [n]$ we have $T_{_X} \sim
  T_{_{[3]}}$. The claim therefore follows from Lemma \ref{signa-local}.
\end{proof}

Given a word $\alpha$ on the alphabet $\{x,y,z\}$, let $X(\alpha)$,
$Y(\alpha)$, and $Z(\alpha)$ denote the number of $x$'s, $y$'s and
$z$'s in $\alpha$, respectively.

\begin{prop} \label{factorize}
  Let $S$ be a regular system of size $n$ with signature $\sigma$ and
  associated tableaux $T^\omega_n$. If $\sigma = \sigma_1\cdot
  \sigma_2$ with with $X(\sigma_j) = Y(\sigma_j) =
  Z(\sigma_j) > 0$, then there exists regular systems $S_1$ and $S_2$
  of size $n$ with 
  signatures $\sigma_1$ and $\sigma_2$.
 Moreover, if $T_n^{\omega_1}$
  and $T_n^{\omega_2}$ are the tableaux associated to $S_1$ and
  $S_2$, then $T_n^\omega = T_n^{\omega_1\cdot\omega_2}$. 
\end{prop}

\begin{proof} Set $T\defeq T_n^\omega$. By assumption there are
  positive integers $k_1$ and $k_2$, 
  such that $k_j \defeq X(\sigma_j) = Y(\sigma_j) = Z(\sigma_j)$.
For every $i\in[n]$, write $T(i) = T_1(i) \cdot T_2(i)$, where
$T_1(i)$ consists of the initial $k_1(n-1)$
entries of $T(i)$, and $T_2(i)$ consists of the final $k_2(n-1)$ entries
of $T(i)$. The sequences $T_1(i)$ and $T_2(i)$ form two tableaux, $T_1
= (T_1(1), \dots, T_1(n))$ and $T_2 = (T_2(1), \dots, T_2(n))$, which
satisfy $T = T_1 \cdot T_2$. We first show that $T_1$ and $T_2$ are
regular. 

Let $A = \{i_1,i_2,i_3\}
  \subset [n]$ be an arbitrary subset with
  $i_1<i_2<i_3$, and
  consider the restriction $T_{_A}$. Since $T$ is regular, Lemma
  \ref{signa-local} implies that $T(i_1)_{_A} 
  \sim \sigma_{_{\{x,y\}}}$, $T(i_2)_{_A} \sim \sigma_{_{\{x,z\}}}$, and
  $T(i_3)_{_A} \sim \sigma_{_{\{y,z\}}}$. It follows that
\[\begin{array}{rcl}
T_1(i_1)_{_A} \cdot T_{2}(i_1)_{_A} & \sim &  (\sigma_1)_{_{\{x,y\}}}
\cdot (\sigma_2)_{_{\{x,y\}}} \\ 
T_1(i_2)_{_A} \cdot T_{2}(i_1)_{_A} & \sim &  (\sigma_1)_{_{\{x,z\}}}
\cdot (\sigma_2)_{_{\{x,z\}}} \\ 
T_1(i_3)_{_A} \cdot T_{2}(i_1)_{_A} & \sim &  (\sigma_1)_{_{\{y,z\}}}
\cdot (\sigma_2)_{_{\{y,z\}}} ,
\end{array}\]
and therefore,
\[\begin{array}{rclcrcl}
T_1(i_1)_{_A} & \sim &  (\sigma_1)_{_{\{x,y\}}} & \hspace{.3cm},\hspace{.3cm} &
T_2(i_1)_{_A} & \sim &  (\sigma_2)_{_{\{x,y\}}} \\
T_1(i_2)_{_A} & \sim &  (\sigma_1)_{_{\{x,z\}}} & \hspace{.3cm},\hspace{.3cm} &
T_2(i_2)_{_A} & \sim &  (\sigma_2)_{_{\{x,z\}}} \\
T_1(i_3)_{_A} & \sim &  (\sigma_1)_{_{\{y,z\}}} & \hspace{.3cm},\hspace{.3cm} &
T_2(i_3)_{_A} & \sim &  (\sigma_2)_{_{\{y,z\}}}.
\end{array}\]

This proves that $T_1$ and $T_2$ are regular, and by Proposition
\ref{tab-corr} there exists words $\omega_1$ and $\omega_2$ in
$W_{_{(ab,cd)}}$ such that $T_1 = T^{\omega_1}_n$ and $T_2 =
T^{\omega_2}_n$. By Corollary \ref{tabconc}, $T = T_n^{\omega_1\cdot\omega_2}$.

We now show that $T_1$ and $T_2$ are geometric. Each path of $S$ is
involved in $(k_1+k_2)(n-1)$ crossings. Cut each 
path of $S$ at a point strictly between the $k_1(n-1)$'th and the
$k_1(n-1)+1$'st crossings. After suitable
homeomorphisms of the plane, we may view this as two distinct systems of
paths, $S_1$ and $S_2$, where the crossings of $S_1$ are precisely the
initial $k_1(n-1)$ crossings of each path of $S$, and the crossings of
$S_2$ are precisely the final $k_2(n-1)$ crossings of each path of
$S$. It is easily seen that $S_1$ and $S_2$ are regular of size
$n$. Clearly $S_1$ 
has signature $\sigma_1$ and associated tableau $T_1$. Furthermore, if
$k_1$ is even, then $S_2$ has signature $\sigma_2$ and associated
tableau $T_2$. On the other hand, if $k_1$ is odd, then the right
endpoints of the paths in $S_1$ appear in reverse order. Thus, if we
consider a vertical reflection of $S_2$ we obtain a regular system of
size $n$ which has signature $\sigma_2$ and associated tableau
$T_2$.
\end{proof}

\begin{example} \label{ex:concat}
Consider the following system $S$ of size $4$.

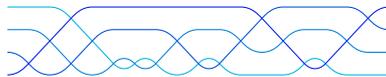
\begin{center}
\begin{tikzpicture}[scale=.3]
\begin{scope}[xscale=.93]
\draw[blue!60!black!30!cyan](0,4)
\ls\ls\ds\ld\ld\du\ud\de\us\ud\de\ls\ls\us\ud\de\ls\ls; 
\draw[blue!70!black!50!cyan](0,3)
\ls\ds\ld\de\us\ud\du\lu\ue\ls\ds\de\us\ue\ls\ds\de\ls;
\draw[blue!80!black!70!cyan](0,2)
\ds\de\us\lu\ue\ls\ls\ds\ld\du\lu\lu\ue\ls\ls\ls\ds\de;
\draw[blue!90!black!90!cyan](0,1)
\us\lu\lu\ue\ls\ls\ls\ls\ls\ls\ls\ds\ld\ld\du\lu\lu\ue;
\end{scope}
\end{tikzpicture}
\end{center}

By inspection we see that $S$ is regular with signature $\sigma =
xyz^3xy^2x$. The associated tableau is $T^\omega_4$ where $\omega =
(ab)(cd)(cd)(dc)(a^2b^2)$.

 \begin{center}
  \begin{tikzpicture}
    \begin{scope}[scale = .4]
      \draw (0,0) --++ (9,0) --++ (0,4) --++ (-9,0) --cycle;
      
      \draw 
      (1,0) --++ (0,4) 
      (2,0) --++ (0,4) 
      (3,0) --++ (0,4) 
      (4,0) --++ (0,4) 
      (5,0) --++ (0,4) 
      (6,0) --++ (0,4) 
      (7,0) --++ (0,4) 
      (8,0) --++ (0,4) ;
      
      \draw 
      (0,1) --++ (9,0) 
      (0,2) --++ (9,0) 
      (0,3) --++ (9,0);

      \node at (.5,.5) {\ft $2$};
      \node at (1.5,.5) {\ft $3$};
      \node at (2.5,.5) {\ft $4$};
      \node at (3.5,.5) {\ft $2$};
      \node at (4.5,.5) {\ft $3$};
      \node at (5.5,.5) {\ft $4$};
      \node at (6.5,.5) {\ft $4$};
      \node at (7.5,.5) {\ft $3$};
      \node at (8.5,.5) {\ft $2$};
     
      \node at (.5,1.5) {\ft $1$};
      \node at (1.5,1.5) {\ft $3$};
      \node at (2.5,1.5) {\ft $4$};
      \node at (3.5,1.5) {\ft $3$};
      \node at (4.5,1.5) {\ft $4$};
      \node at (5.5,1.5) {\ft $4$};
      \node at (6.5,1.5) {\ft $3$};
      \node at (7.5,1.5) {\ft $1$};
      \node at (8.5,1.5) {\ft $1$};
           
      \node at (.5,2.5) {\ft $1$};
      \node at (1.5,2.5) {\ft $2$};
      \node at (2.5,2.5) {\ft $4$};
      \node at (3.5,2.5) {\ft $4$};
      \node at (4.5,2.5) {\ft $4$};
      \node at (5.5,2.5) {\ft $2$};
      \node at (6.5,2.5) {\ft $2$};
      \node at (7.5,2.5) {\ft $1$};
      \node at (8.5,2.5) {\ft $1$};
           
      \node at (.5,3.5) {\ft $1$};
      \node at (1.5,3.5) {\ft $2$};
      \node at (2.5,3.5) {\ft $3$};
      \node at (3.5,3.5) {\ft $3$};
      \node at (4.5,3.5) {\ft $3$};
      \node at (5.5,3.5) {\ft $2$};
      \node at (6.5,3.5) {\ft $2$};
      \node at (7.5,3.5) {\ft $1$};
      \node at (8.5,3.5) {\ft $1$};      
    \end{scope}
  \end{tikzpicture}
\end{center}

Notice that $\sigma$ can be written as $\sigma = \sigma_1 \cdot
\sigma_2$ where $\sigma_1 = xyz$ and $\sigma_2 = z^2xy^2x$. Therefore
we get the following regular systems $S_1$ and $S_2$ of size 4 with signatures
$\sigma_1$ and $\sigma_2$. 
Notice that since $X(\sigma_1)=1$, $S_2$ is a vertical reflection of
its corresponding part in $S$.

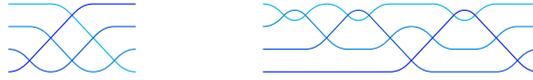
\begin{center}
\begin{tikzpicture}[scale=.3]
\begin{scope}[xscale=.93]
\draw[blue!60!black!30!cyan](0,4)
\ls\ls\ds\ld\ld\de; 
\draw[blue!70!black!50!cyan](0,3)
\ls\ds\ld\de\us\ue;
\draw[blue!80!black!70!cyan](0,2)
\ds\de\us\lu\ue\ls;
\draw[blue!90!black!90!cyan](0,1)
\us\lu\lu\ue\ls\ls;
\end{scope}

\begin{scope}[xscale=.93, xshift = 12cm, yscale = -1, yshift = -5cm]
\draw[blue!60!black!30!cyan](0,1)
\us\ud\de\us\ud\de\ls\ls\us\ud\de\ls\ls; 
\draw[blue!70!black!50!cyan](0,2)
\ds\du\lu\ue\ls\ds\de\us\ue\ls\ds\de\ls;
\draw[blue!80!black!70!cyan](0,3)
\ls\ls\ds\ld\du\lu\lu\ue\ls\ls\ls\ds\de;
\draw[blue!90!black!90!cyan](0,4)
\ls\ls\ls\ls\ls\ls\ds\ld\ld\du\lu\lu\ue;
\end{scope}
\end{tikzpicture}
\end{center}

Their associated tableau are
$T_n^{\omega_1}$ and $T_n^{\omega_2}$ where $\omega_1 = (ab)(cd)$ and
$\omega = (cd)(dc)(b^2a^2)$.

 \begin{center}
  \begin{tikzpicture}
    \begin{scope}[scale = .4]
      \draw (0,0) --++ (3,0) --++ (0,4) --++ (-3,0) --cycle;
      
      \draw 
      (1,0) --++ (0,4) 
      (2,0) --++ (0,4); 
      
      \draw 
      (0,1) --++ (3,0) 
      (0,2) --++ (3,0) 
      (0,3) --++ (3,0);

      \node at (.5,.5) {\ft $2$};
      \node at (1.5,.5) {\ft $3$};
      \node at (2.5,.5) {\ft $4$};
          
      \node at (.5,1.5) {\ft $1$};
      \node at (1.5,1.5) {\ft $3$};
      \node at (2.5,1.5) {\ft $4$};
                
      \node at (.5,2.5) {\ft $1$};
      \node at (1.5,2.5) {\ft $2$};
      \node at (2.5,2.5) {\ft $4$};
           
      \node at (.5,3.5) {\ft $1$};
      \node at (1.5,3.5) {\ft $2$};
      \node at (2.5,3.5) {\ft $3$};
    \end{scope}

    \begin{scope}[scale = .4, xshift = 8cm]
      \draw (0,0) --++ (6,0) --++ (0,4) --++ (-6,0) --cycle;      
      \draw 
      (1,0) --++ (0,4) 
      (2,0) --++ (0,4) 
      (3,0) --++ (0,4) 
      (4,0) --++ (0,4) 
      (5,0) --++ (0,4); 
      
      \draw 
      (0,1) --++ (6,0) 
      (0,2) --++ (6,0) 
      (0,3) --++ (6,0);

      \node at (0.5,.5) {\ft $2$};
      \node at (1.5,.5) {\ft $3$};
      \node at (2.5,.5) {\ft $4$};
      \node at (3.5,.5) {\ft $4$};
      \node at (4.5,.5) {\ft $3$};
      \node at (5.5,.5) {\ft $2$};
     
      \node at (0.5,1.5) {\ft $3$};
      \node at (1.5,1.5) {\ft $4$};
      \node at (2.5,1.5) {\ft $4$};
      \node at (3.5,1.5) {\ft $3$};
      \node at (4.5,1.5) {\ft $1$};
      \node at (5.5,1.5) {\ft $1$};
           
      \node at (0.5,2.5) {\ft $4$};
      \node at (1.5,2.5) {\ft $4$};
      \node at (2.5,2.5) {\ft $2$};
      \node at (3.5,2.5) {\ft $2$};
      \node at (4.5,2.5) {\ft $1$};
      \node at (5.5,2.5) {\ft $1$};
           
      \node at (0.5,3.5) {\ft $3$};
      \node at (1.5,3.5) {\ft $3$};
      \node at (2.5,3.5) {\ft $2$};
      \node at (3.5,3.5) {\ft $2$};
      \node at (4.5,3.5) {\ft $1$};
      \node at (5.5,3.5) {\ft $1$};      
    \end{scope}
  \end{tikzpicture}
\end{center}
\end{example}

\section{Characterization of regular systems} \label{sec:Kara}

We are ready to state our characterization of regular systems
(Proposition \ref{kara}). This is given in terms of certain
combinatorial conditions on the signatures of $k$-crossing systems of
size 3. We also deduce some consequences of the characterization which
will be used in the proofs of Theorems \ref{UL envelope} and
\ref{envelope theorem}. The proof of Proposition \ref{kara} is
given in the end of this section.

\begin{define}
  Let $\sigma$ be the signature of a $k$-crossing system of size 3. We
  say that $\sigma$ is \df{extendable} if there exists a
  regular system of size $4$ with signature $\sigma$.
\end{define}

\begin{remark} \label{nextend}
  By definition, if $S$ is a regular system with signature $\sigma$,
  then $\sigma$ is necessarily extendable, since a regular system has
  size at least 4 and the property of being regular is inherited by
  all subsystems of size 4. A less obvious fact is that for every
  extendable signature $\sigma$ and every integer $n>4$ there exists a
  regular system of size $n$ with signature $\sigma$. Thus the
  characterization of regular system amounts to the characterization
  of extendable signatures. The proof of this fact will be implicit in
  our proof of Proposition \ref{kara}. 
\end{remark}

We start with a few simple reductions. First notice that there are
certain invariant symmetries. For instance, a signature $\sigma$ is
extendable if and only if the {\em reverse} signature, $-\sigma$, is
extendable. This is simply the effect of a horizontal (and possibly
vertical) reflection of the corresponding system of paths. The other
invariant symmetry is that of {\em interchanging} the $x$'s 
and $z$'s in $\sigma$. This corresponds to a vertical reflection of the
corresponding system of paths. These symmetries will be used
to reduce some of the case analysis in our arguments.  

\begin{define} A non-empty word $\omega \in
  \{x,y,z\}^\star$ is called \df{reducible} if there is a non-empty proper
  initial substring $\omega'$ of $\omega$ such that $X(\omega') =
  Y(\omega') = Z(\omega')$. If $\omega$ is not reducible, then
  $\omega$ is called \df{irreducible}. 
\end{define}

In view of Proposition \ref{factorize} it suffices to characterize
extendable irreducible signatures.

\begin{prop}\label{kara}
  If $\sigma$ is an irreducible signature, then $\sigma$ is
  extendable if and only if the following conditions hold.
  \begin{enumerate}
  \item There exists an $\omega \in W_{_{(ab,cd)}}$ such that
    $T^\omega_3$ is the tableau associated to a system with signature $\sigma$.
     
 \item $\sigma = \alpha\cdot \beta \cdot \gamma$ or $\sigma =
    \gamma\cdot \beta \cdot \alpha$ such that 
\[
\left(\begin{array}{ccc}
 X(\alpha) & Y(\alpha) & Z(\alpha) \\
 X(\beta) & Y(\beta) & Z(\beta) \\
 X(\gamma) & Y(\gamma) & Z(\gamma) 
\end{array}\right) = 
\left(\begin{array}{ccc}
 p & p & 0 \\
 q & 0 & p \\
 0 & q & q 
\end{array}\right) 
\] where $p$ and $q$ are non-negative integers not both equal to 0.
  \end{enumerate}
Furthermore, if $S$ is a regular system labeled by $[n]$ with
irreducible signature $\sigma$ and $\omega\in W_{_{(ab,cd)}}$ is the
word satisfying condition (1), then $T^\omega_n$ is the
 tableau associated to $S$. 
\end{prop}

\begin{remark} \label{condition1}
Condition {\em (1)} of Proposition \ref{kara} is a purely combinatorial condition which relates the signature of a regular system to the associated regular tableau. Notice that if there is a word $\omega \in W_{_{(ab,cd)}}$ which satisfies condition {\em (1)} of Proposition \ref{kara}, then it is unique and each distinct letter occurs the same number of times. We call this the word \df{associated} to the irreducible signature $\sigma$. A signature and its associated word are related by Lemma \ref{signa-local}, and Proposition \ref{3tab-bal-ref} implies that a signature $\sigma$ has an associated word if and only if the following conditions hold:
\begin{itemize}
\item $\sigma_{_{\{x,y\}}}$ is balanced and $\langle
  \sigma_{_{\{x,y\}}} \rangle$ is a refinement of $\exp_z(\sigma_{_{\{x,z\}}})$.
\item $\sigma_{_{\{y,z\}}}$ is balanced and $\langle
  \sigma_{_{\{y,z\}}} \rangle$ is a refinement of $\exp_x(\sigma_{_{\{x,z\}}})$.
\end{itemize}
\end{remark}

\begin{remark}
  Condition {\em (2)} of Proposition \ref{kara} is a geometric condition related to the extendability. As soon as an irreducible signature $\sigma$ has an associated word $\omega \in W_{_{(ab,cd)}}$, we automatically obtain the tableaux $T^\omega_n$. For $n=3$ this is the tableau associated to a system with signature $\sigma$, and for all $n\geq 4$ these tableau are regular. It turns out that the tableaux $T^\omega_n$ are geometric for all $n\geq 4$ if and only of condition {\em (2)} holds.
\end{remark}

Proposition \ref{kara} provides a simple way to determine
whether an irreducible signature is extendable. Condition {\em (1)}
tells us that the restrictions $\sigma_{_{\{x,y\}}}$ and
$\sigma_{_{\{y,z\}}}$ factor into balanced 
blocks, and the restriction $\sigma_{_{\{x,z\}}}$ tells us how these
balanced blocks should be arranged to form the associated word
$\omega$. We illustrate this with several examples.

\begin{example}
Let $\sigma = x y^4 x^3 z^4 x^2 z y^2 z$.  It is easily verified that
$\sigma$ is irreducible and satisfies Proposition \ref{parity}, which
implies $\sigma$ is the signature of a system. Now we show that
$\sigma$ satisfies the conditions of Proposition \ref{kara}.  

For condition {\em (1)}, we find the associated word $\omega$ by first
factoring $\sigma_{_{\{x,y\}}}$ and $\sigma_{_{\{x,y\}}}$ into
balanced blocks, 
\[ 
\sigma_{_{\{x,y\}}} = (xy)(y^3x^3)(x^2y^2) \sim \omega_{_{\{c,d\}}}
\quad \mbox{ and } \quad
\sigma_{_{\{y,z\}}} = (y^4z^4)(zy)(yz) \sim \omega_{_{\{a,b\}}}. \]

Then interlace the corresponding blocks of $\omega_{_{\{c,d\}}}$ and
$\omega_{_{\{a,b\}}}$ according to $\sigma_{_{\{x,z\}}}$ as in
Proposition~\ref{3tab-bal-ref}.  

\[\begin{array}{r@{\;}c@{}c@{}c@{}c@{}c@{}c@{\vspace{-4pt}}}
\sigma_{_{\{x,z\}}} = & x^4 & z & z^3 & x & x & z^2 \\ \\
& \tikz \draw[black!40,-latex] (0,0) -- (0,-11pt);  
& \tikz \draw[black!40,-latex] (0,0) -- (0,-11pt);  
& \tikz \draw[black!40,-latex] (0,0) -- (0,-11pt);  
& \tikz \draw[black!40,-latex] (0,0) -- (0,-11pt);  
& \tikz \draw[black!40,-latex] (0,0) -- (0,-11pt);  
& \tikz \draw[black!40,-latex] (0,0) -- (0,-11pt);  
\\ \\
\omega = & (a^4b^4) & (cd) & (d^3c^3) & (ba) & (ab) & (c^2d^2).
\end{array} \]

For condition {\em (2)}, $\sigma$ factors as shown below.

\bigskip

\begin{center}
\begin{tikzpicture}[scale=0.32,every node/.style={anchor=mid}]

\draw[blue!60!black!30!cyan] (0,2)
\g{1} \ds \du \ud \du \ue \g{3}
\ds \du \ud \du \ue \g{2}
\ds \dg{1} \du \ug{1} \ue
;

\draw[blue!75!black!60!cyan] (0,1)
\ds \de \g{3} \us \ud \du \ue
\us \ud \du \ud \dg{1} \du \ue
\us \ue \g{1} \ds \de 
;

\draw[blue!90!black!90!cyan] (0,0)
\us \ug{1} \ud \du \ud \dg{1} \du \ud \de
\g{4} \us \ud \de 
\g{1} \us \ud \de \g{1}
;

\path
(1,-1) node (a0) {\ft $x$}
++(1,0) node {\ft $y$}
++(1,0) node {\ft $y$}
++(1,0) node {\ft $y$}
++(1,0) node {\ft $y$}
++(1,0) node {\ft $x$}
++(1,0) node {\ft $x$}
++(1,0) node (a1) {\ft $x$}
++(2,0) node (b0) {\ft $z$}
++(1,0) node {\ft $z$}
++(1,0) node {\ft $z$}
++(1,0) node {\ft $z$}
++(1,0) node {\ft $x$}
++(1,0) node (b1) {\ft $x$}
++(2,0) node (c0) {\ft $z$}
++(1,0) node {\ft $y$}
++(1,0) node {\ft $y$}
++(1,0) node (c1) {\ft $z$}
;

\draw [decorate,decoration={brace,amplitude=5pt,mirror}]
(a0.south west) -- (a1.south east) node [midway,yshift=-12pt,anchor=mid] {\ft $\alpha$}
;
\draw [decorate,decoration={brace,amplitude=5pt,mirror}]
(b0.south west) -- (b1.south east) node [midway,yshift=-12pt,anchor=mid] {\ft $\beta$}
;
\draw [decorate,decoration={brace,amplitude=5pt,mirror}]
(c0.south west) -- (c1.south east) node [midway,yshift=-12pt,anchor=mid] {\ft $\gamma$}
;

\end{tikzpicture}
\end{center}

Informally, we may interpret paths in a regular system as ``taking
turns'' crossing the other paths, and interpret the factorization in
condition (2) as indicating how each path crosses other paths on its
turn. This can be seen more precisely in the proof of Proposition
\ref{kara}, but for now just consider the factors in this example.  
Here, $\sigma = \alpha \cdot \beta \cdot \gamma$ indicates that the
paths take turns in order $1,\dots,n$. The term ${\alpha =
  (xy)(y^3x^3)}$ indicates that path $i$ crosses paths $[n]_i^+$ on
its turn first in ascending order one at a time and then in descending
order three at a time. The term ${\gamma = (zy)(yz)}$ indicates that
path $i$ crosses paths $[n]_i^-$ on its turn first in descending and
then in ascending order one at a time, and the term $\beta = (z^4x^2)$
indicates that path $i$ first crosses paths 
$[n]_i^+$ then paths $[n]_i^-$.  Below we see a regular system $S'$
with associated tableau $T_5^\omega$.  ``Turns'' are indicated by
thickened paths in the system  and by unshaded boxes in the tableau.  

\bigskip

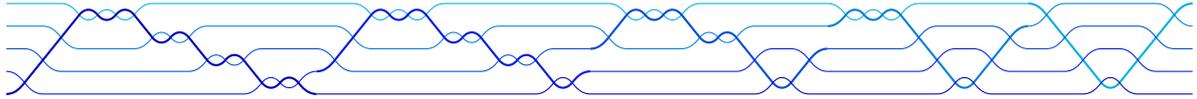
\begin{center}
\begin{tikzpicture}[xscale=0.24,yscale=0.3]

\draw[blue!60!black!30!cyan] (0,4)
\g{3} \ds \du \ud \du \ue
\g{11} \ds \du \ud \du \ue
\g{9} \ds \du \ud \du \ue
\g{7} \ds \du \ud \du \ue \g{6} coordinate (a5)
;

\draw[blue!70!black!50!cyan] (0,3)
\g{2} \ds \de \g{3} \us \ud \du \ue
\g{7} \ds \de \g{3} \us \ud \du \ue
\g{5} \ds \de \g{3} \us \ud \du \ue \g{4} coordinate (a4)
;

\draw[blue!80!black!70!cyan] (0,2)
\g{1} \ds \de \g{7} \us \ud \du \ue
\g{3} \ds \de \g{7} \us \ud \du \ue \g{2} coordinate (a3)
;

\draw[blue!90!black!90!cyan] (0,1)
\ds \de \g{11} \us \ud \du \ue coordinate (a2)
;

;

\draw[thick,blue!60!black!30!cyan] (a5)
\ds \dg{3} \du \ug{3} \ue
;

\draw[thick,blue!70!black!50!cyan] (a4)
\us \ud \du \ud \dg{3} \du \ug{2} \ue coordinate (b4)
;

\draw[thick,blue!80!black!70!cyan] (a3)
\us \ug{1} \ud \du \ud \dg{1} \du \ud \dg{2} \du \ug{1} \ue coordinate (b3)
;

\draw[thick,blue!90!black!90!cyan] (a2)
\us \ug{2} \ud \du \ud \dg{1} \du \ud \dg{1} \du \ud \dg{1} \du \ue coordinate (b2)
;

\draw[thick,blue!70!black] (0,0)
\us \ug{3} \ud \du \ud \dg{1} \du \ud \dg{1} \du \ud \dg{1} \du \ud \de coordinate (b1)
;

\draw[blue!70!black!50!cyan] (b4)
\us \ue \g{5} \ds \de
;

\draw[blue!80!black!70!cyan] (b3)
\g{4} \us \ue \g{3} \ds \de
\g{1} \us \ue \g{3} \ds \de \g{1}
;

\draw[blue!90!black!90!cyan] (b2)
\g{8} \us \ue \g{1} \ds \de
\g{5} \us \ue \g{1} \ds \de
\g{3} \us \ue \g{1} \ds \de \g{2}
;

\draw[blue!70!black] (b1)
\g{12} \us \ud \de
\g{9} \us \ud \de
\g{7} \us \ud \de
\g{5} \us \ud \de \g{3}
;

\end{tikzpicture}

\bigskip

\begin{tikzpicture}[scale=.4]

\fill[blue, opacity = .15]
(0,1) --++ (4,0) --++ (0,1) --++ (4,0) --++ (0,1) --++ (4,0) --++ (0,1) --++ (4,0) --++ (0,1) -- (0,5) -- cycle
(24,4) --++ (-2,0) --++ (0,-1) --++ (-2,0) --++ (0,-1) --++ (-2,0) --++ (0,-1) --++ (-2,0) --++ (0,-1) -- (24,0) -- cycle
;

\foreach \i in {0,...,24}
{ \draw (\i,0) -- (\i,5); }
\foreach \i in {0,...,5}
{ \draw (0,\i) -- (24,\i);}

\path 
(.5,.5) node {\ft $2$}
++(1,0) node {\ft $3$}
++(1,0) node {\ft $4$}
++(1,0) node {\ft $5$}
++(1,0) node {\ft $5$}
++(1,0) node {\ft $5$}
++(1,0) node {\ft $5$}
++(1,0) node {\ft $4$}
++(1,0) node {\ft $4$}
++(1,0) node {\ft $4$}
++(1,0) node {\ft $3$}
++(1,0) node {\ft $3$}
++(1,0) node {\ft $3$}
++(1,0) node {\ft $2$}
++(1,0) node {\ft $2$}
++(1,0) node {\ft $2$}
++(1,0) node {\ft $2$}
++(1,0) node {\ft $2$}
++(1,0) node {\ft $3$}
++(1,0) node {\ft $3$}
++(1,0) node {\ft $4$}
++(1,0) node {\ft $4$}
++(1,0) node {\ft $5$}
++(1,0) node {\ft $5$}
;
\path 
(.5,1.5) node {\ft $1$}
++(1,0) node {\ft $1$}
++(1,0) node {\ft $1$}
++(1,0) node {\ft $1$}
++(1,0) node {\ft $3$}
++(1,0) node {\ft $4$}
++(1,0) node {\ft $5$}
++(1,0) node {\ft $5$}
++(1,0) node {\ft $5$}
++(1,0) node {\ft $5$}
++(1,0) node {\ft $4$}
++(1,0) node {\ft $4$}
++(1,0) node {\ft $4$}
++(1,0) node {\ft $3$}
++(1,0) node {\ft $3$}
++(1,0) node {\ft $3$}
++(1,0) node {\ft $1$}
++(1,0) node {\ft $1$}
++(1,0) node {\ft $3$}
++(1,0) node {\ft $3$}
++(1,0) node {\ft $4$}
++(1,0) node {\ft $4$}
++(1,0) node {\ft $5$}
++(1,0) node {\ft $5$}
;
\path 
(.5,2.5) node {\ft $1$}
++(1,0) node {\ft $1$}
++(1,0) node {\ft $1$}
++(1,0) node {\ft $1$}
++(1,0) node {\ft $2$}
++(1,0) node {\ft $2$}
++(1,0) node {\ft $2$}
++(1,0) node {\ft $2$}
++(1,0) node {\ft $4$}
++(1,0) node {\ft $5$}
++(1,0) node {\ft $5$}
++(1,0) node {\ft $5$}
++(1,0) node {\ft $5$}
++(1,0) node {\ft $4$}
++(1,0) node {\ft $4$}
++(1,0) node {\ft $4$}
++(1,0) node {\ft $2$}
++(1,0) node {\ft $1$}
++(1,0) node {\ft $1$}
++(1,0) node {\ft $2$}
++(1,0) node {\ft $4$}
++(1,0) node {\ft $4$}
++(1,0) node {\ft $5$}
++(1,0) node {\ft $5$}
;
\path 
(.5,3.5) node {\ft $1$}
++(1,0) node {\ft $1$}
++(1,0) node {\ft $1$}
++(1,0) node {\ft $1$}
++(1,0) node {\ft $2$}
++(1,0) node {\ft $2$}
++(1,0) node {\ft $2$}
++(1,0) node {\ft $2$}
++(1,0) node {\ft $3$}
++(1,0) node {\ft $3$}
++(1,0) node {\ft $3$}
++(1,0) node {\ft $3$}
++(1,0) node {\ft $5$}
++(1,0) node {\ft $5$}
++(1,0) node {\ft $5$}
++(1,0) node {\ft $5$}
++(1,0) node {\ft $3$}
++(1,0) node {\ft $2$}
++(1,0) node {\ft $1$}
++(1,0) node {\ft $1$}
++(1,0) node {\ft $2$}
++(1,0) node {\ft $3$}
++(1,0) node {\ft $5$}
++(1,0) node {\ft $5$}
;
\path 
(.5,4.5) node {\ft $1$}
++(1,0) node {\ft $1$}
++(1,0) node {\ft $1$}
++(1,0) node {\ft $1$}
++(1,0) node {\ft $2$}
++(1,0) node {\ft $2$}
++(1,0) node {\ft $2$}
++(1,0) node {\ft $2$}
++(1,0) node {\ft $3$}
++(1,0) node {\ft $3$}
++(1,0) node {\ft $3$}
++(1,0) node {\ft $3$}
++(1,0) node {\ft $4$}
++(1,0) node {\ft $4$}
++(1,0) node {\ft $4$}
++(1,0) node {\ft $4$}
++(1,0) node {\ft $4$}
++(1,0) node {\ft $3$}
++(1,0) node {\ft $2$}
++(1,0) node {\ft $1$}
++(1,0) node {\ft $1$}
++(1,0) node {\ft $2$}
++(1,0) node {\ft $3$}
++(1,0) node {\ft $4$}
;

\end{tikzpicture}

\end{center}

\medskip

Compare path 3 of $S'$ on its turn to the terms $\alpha$ and $\gamma$. This example is particularly simple in that there is no interlacing of the order path $i$ crosses paths $[n]_i^+$ and $[n]_i^-$ on its turn, and path $i$ crosses all other paths an even number of times, returning to its original position in order among the other paths at the end of its turn.

\end{example}

\begin{example}
Let $\sigma = xy^2x^2yzx^4z^3y^4z^3$ and observe that $\sigma$ is the
signature of a system of paths and is also irreducible. 

For condition {\em (1)}, first notice that $\sigma_{_{\{x,y\}}}$ and
$\sigma_{_{\{y,z\}}}$ can be written as
\[\sigma_{_{\{x,y\}}} = (xy)(yx)(xy)(x^4y^4) \quad \mbox{ and } \quad
\sigma_{_{\{y,z\}}} = (y^3z^3)(zy)(y^3z^3),\]

and that $\sigma_{_{\{x,z\}}} = x^3zx^4z^6$. This gives us the associated word
\[\omega = (a^3b^3)(cd)(ba)(a^3b^3)(dc)(cd)(c^4d^4).\]

For condition {\em (2)}, $\sigma$ factors as shown below.

\bigskip

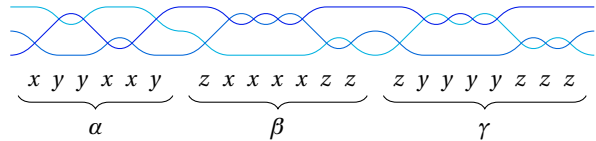
\begin{center}
\begin{tikzpicture}[scale=0.32,every node/.style={anchor=mid}]

\draw[blue!60!black!30!cyan] (0,2)
\g{1} \ds \du \ue \g{1} \ds \de
\ds \de \g{3} \us \ud \de
\us \ug{1} \ud \du \ud \dg{1} \du \ud \de
;

\draw[blue!75!black!60!cyan] (0,1)
\ds \de \g{1} \us \ud \de \g{1}
\us \ug{1} \ud \du \ud \dg{1} \du \ue
\ds \de \g{3} \us \ud \du \ue
;

\draw[blue!90!black!90!cyan] (0,0)
\us \ug{1} \ud \dg{1} \du \ug{1} \ue
\g{1} \ds \du \ud \du \ue \g{3}
\ds \du \ud \du \ue \g{3}
;

\path
(1,-1) node (a0) {\ft $x$}
++(1,0) node {\ft $y$}
++(1,0) node {\ft $y$}
++(1,0) node {\ft $x$}
++(1,0) node {\ft $x$}
++(1,0) node (a1) {\color{white}\ft $x$} node {\ft $y$}
++(2,0) node (b0) {\ft $z$}
++(1,0) node {\ft $x$}
++(1,0) node {\ft $x$}
++(1,0) node {\ft $x$}
++(1,0) node {\ft $x$}
++(1,0) node {\ft $z$}
++(1,0) node (b1) {\ft $z$}
++(2,0) node (c0) {\ft $z$}
++(1,0) node {\ft $y$}
++(1,0) node {\ft $y$}
++(1,0) node {\ft $y$}
++(1,0) node {\ft $y$}
++(1,0) node {\ft $z$}
++(1,0) node {\ft $z$}
++(1,0) node (c1) {\ft $z$}
;

\draw [decorate,decoration={brace,amplitude=5pt,mirror}]
(a0.south west) -- (a1.south east) node [midway,yshift=-12pt,anchor=mid] {\ft $\alpha$}
;
\draw [decorate,decoration={brace,amplitude=5pt,mirror}]
(b0.south west) -- (b1.south east) node [midway,yshift=-12pt,anchor=mid] {\ft $\beta$}
;
\draw [decorate,decoration={brace,amplitude=5pt,mirror}]
(c0.south west) -- (c1.south east) node [midway,yshift=-12pt,anchor=mid] {\ft $\gamma$}
;

\end{tikzpicture}
\end{center}

Below we see a regular system with associated tableau $T_4^\omega$.

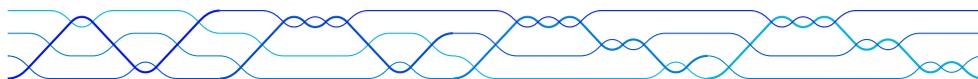
\bigskip

\begin{center}
\begin{tikzpicture}[scale=.3]
\begin{scope}[xscale=.93]
\draw[blue!60!black!30!cyan](0,4)
\ls\ls\ds\du\ue\ls\ls\ls\ds \de
\g{1} \ds\de \ls\ls\ls\us\ue\ls\ds\de
\ds\de\ls\ls\ls\ls\ls\ls\ls\us\ud\de coordinate (a4)
;
\draw[blue!70!black!50!cyan](0,3)
\ls\ds\de\ls\us\ue\ls\ds\de\g{1}
\ds\de\ls
\ls\ls\ls\ls\us\ud\de\g{1} coordinate (a3)
;
\draw[blue!80!black!70!cyan](0,2)
\ds\de\ls\ls\ls\us\ud\de\g{2} coordinate (a2)
;

\draw[thick,blue!60!black!30!cyan](a4)
\us\lu\lu\ud\du\ud\ld\du\ud\ld\du\ud\de
;
\draw[thick,blue!70!black!50!cyan](a3)
\us\lu\lu\ud\du\ud\ld\du\ud\ld\du\ue coordinate (b3)
;
\draw[thick,blue!80!black!70!cyan](a2)
\us\lu\lu\ud\du\ud\ld\ld\du\lu\ue coordinate (b2)
;  
\draw[thick,blue!90!black!90!cyan](0,1)
\us\lu\lu\ud\ld\ld\du\lu\lu\ue coordinate (b1)
;

\draw[blue!70!black!50!cyan](b3)
\ds\de\ls\ls\ls\ls\ls\ls\ls\us\ud\du\ue
;
\draw[blue!80!black!70!cyan](b2)
\ls\ds\de\ls\ls\ls\us\ud\du\ue\ls\ls\ls\ds\de\ls
\ls\ls\us\ud\du\ue\ls\ls\ls
;  
\draw[blue!90!black!90!cyan](b1)
\g{2}\ds\du\ud\du\ue\ls\ls\ls\ls\ls\ls\ds\du\ud\du\ue\ls\ls\ls\ls\ls\ls\ls\ds\du
\ud\du\ue\ls\ls\ls\ls\ls\ls
;
\end{scope}
\end{tikzpicture}
\end{center}

\bigskip

\begin{center}
  \begin{tikzpicture}
    \begin{scope}[scale = .4]

\fill[blue, opacity = .15]
(0,1) --++ (3,0) --++ (0,1) --++ (3,0) --++ (0,1) --++ (3,0) --++ (0,1) -- (0,4) -- cycle
(21,3) --++(-4,0) --++ (0,-1) --++(-4,0) --++ (0,-1) --++(-4,0) --++ (0,-1) -- (21,0) -- cycle
;

      \draw (0,0) --++ (21,0) --++ (0,4) --++ (-21,0) --cycle;
      
      \draw 
      (1,0) --++ (0,4) 
      (2,0) --++ (0,4) 
      (3,0) --++ (0,4) 
      (4,0) --++ (0,4) 
      (5,0) --++ (0,4) 
       (6,0) --++ (0,4) 
      (7,0) --++ (0,4) 
      (8,0) --++ (0,4) 
      (9,0) --++ (0,4) 
      (10,0) --++ (0,4) 
      (11,0) --++ (0,4) 
      (12,0) --++ (0,4) 
      (13,0) --++ (0,4) 
      (14,0) --++ (0,4) 
      (15,0) --++ (0,4) 
      (16,0) --++ (0,4) 
      (17,0) --++ (0,4) 
      (18,0) --++ (0,4) 
      (19,0) --++ (0,4) 
      (20,0) --++ (0,4) ;

      \draw (0,1) --++ (21,0) (0,2) --++ (21,0) (0,3) --++ (21,0);

      \node at (.5,.5) {\ft $2$};
      \node at (1.5,.5) {\ft $3$};
      \node at (2.5,.5) {\ft $4$};
      \node at (3.5,.5) {\ft $4$};
      \node at (4.5,.5) {\ft $3$};
      \node at (5.5,.5) {\ft $2$};
      \node at (6.5,.5) {\ft $2$};
      \node at (7.5,.5) {\ft $3$};
      \node at (8.5,.5) {\ft $4$};
      \node at (9.5,.5) {\ft $2$};
      \node at (10.5,.5) {\ft $2$};
      \node at (11.5,.5) {\ft $2$};
      \node at (12.5,.5) {\ft $2$};
      \node at (13.5,.5) {\ft $3$};
      \node at (14.5,.5) {\ft $3$};
      \node at (15.5,.5) {\ft $3$};
      \node at (16.5,.5) {\ft $3$};
      \node at (17.5,.5) {\ft $4$};
      \node at (18.5,.5) {\ft $4$};
      \node at (19.5,.5) {\ft $4$};
      \node at (20.5,.5) {\ft $4$};

      \node at (.5,1.5) {\ft $1$};
      \node at (1.5,1.5) {\ft $1$};
      \node at (2.5,1.5) {\ft $1$};
      \node at (3.5,1.5) {\ft $3$};
      \node at (4.5,1.5) {\ft $4$};
      \node at (5.5,1.5) {\ft $1$};
      \node at (6.5,1.5) {\ft $1$};
      \node at (7.5,1.5) {\ft $1$};
      \node at (8.5,1.5) {\ft $1$};
      \node at (9.5,1.5) {\ft $4$};
      \node at (10.5,1.5) {\ft $3$};
      \node at (11.5,1.5) {\ft $3$};
      \node at (12.5,1.5) {\ft $4$};
      \node at (13.5,1.5) {\ft $3$};
      \node at (14.5,1.5) {\ft $3$};
      \node at (15.5,1.5) {\ft $3$};
      \node at (16.5,1.5) {\ft $3$};
      \node at (17.5,1.5) {\ft $4$};
      \node at (18.5,1.5) {\ft $4$};
      \node at (19.5,1.5) {\ft $4$};
      \node at (20.5,1.5) {\ft $4$};
     
      \node at (.5,2.5) {\ft $1$};
      \node at (1.5,2.5) {\ft $1$};
      \node at (2.5,2.5) {\ft $1$};
      \node at (3.5,2.5) {\ft $2$};
      \node at (4.5,2.5) {\ft $2$};
      \node at (5.5,2.5) {\ft $2$};
      \node at (6.5,2.5) {\ft $4$};
      \node at (7.5,2.5) {\ft $2$};
      \node at (8.5,2.5) {\ft $1$};
      \node at (9.5,2.5) {\ft $1$};
      \node at (10.5,2.5) {\ft $1$};
      \node at (11.5,2.5) {\ft $1$};
      \node at (12.5,2.5) {\ft $2$};
      \node at (13.5,2.5) {\ft $2$};
      \node at (14.5,2.5) {\ft $2$};
      \node at (15.5,2.5) {\ft $4$};
      \node at (16.5,2.5) {\ft $4$};
      \node at (17.5,2.5) {\ft $4$};
      \node at (18.5,2.5) {\ft $4$};
      \node at (19.5,2.5) {\ft $4$};
      \node at (20.5,2.5) {\ft $4$};
     
      \node at (.5,3.5) {\ft $1$};
      \node at (1.5,3.5) {\ft $1$};
      \node at (2.5,3.5) {\ft $1$};
      \node at (3.5,3.5) {\ft $2$};
      \node at (4.5,3.5) {\ft $2$};
      \node at (5.5,3.5) {\ft $2$};
      \node at (6.5,3.5) {\ft $3$};
      \node at (7.5,3.5) {\ft $3$};
      \node at (8.5,3.5) {\ft $3$};
      \node at (9.5,3.5) {\ft $3$};
      \node at (10.5,3.5) {\ft $2$};
      \node at (11.5,3.5) {\ft $1$};
      \node at (12.5,3.5) {\ft $1$};
      \node at (13.5,3.5) {\ft $1$};
      \node at (14.5,3.5) {\ft $1$};
      \node at (15.5,3.5) {\ft $2$};
      \node at (16.5,3.5) {\ft $2$};
      \node at (17.5,3.5) {\ft $2$};
      \node at (18.5,3.5) {\ft $3$};
      \node at (19.5,3.5) {\ft $3$};
      \node at (20.5,3.5) {\ft $3$};
    \end{scope}
  \end{tikzpicture}
\end{center}
\end{example}

\begin{example}\label{special}
Here we give an example of a signature which satisfies condition {\em
  (1)} of Proposition \ref{kara}, but not condition {\em
  (2)}. Consider the following system.

\medskip

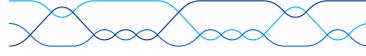
\begin{center}
\begin{tikzpicture}[scale=.3]
\begin{scope}[xscale=.93]
\draw[blue!90!black!30!cyan](0,3)\ls\ds\du\ue\ls\ls\ls\ds\ld\du\ud\du
\lu\ud\ld\du\ue;
\draw[blue!75!black!60!cyan](0,2)\ds\de\ls\us\ud\du\ud\de\us\ud\du\ud
\de\ls\us\ud\de;  
\draw[blue!50!black!80!cyan](0,1)\us\lu\ud\ld\du\ud\du\lu\ue\ls\ls\ls
\ds\du\ue\ls\ls;
\end{scope}
\end{tikzpicture}
\end{center}

Its signature is $\sigma = xy^2x^4yz^4y^2z$ and the associated word is $\omega =
(a^3b^3)(ba)(ab)(cd)(dc)(c^3d^3)$. 

we have 
$\sigma_{_{\{x,y\}}} \sim \omega_{_{\{c,d\}}}$,
$\sigma_{_{\{x,z\}}} \sim \omega_{_{\{a,c\}}}$, and
$\sigma_{_{\{y,z\}}} \sim \omega_{_{\{a,b\}}}$, which shows that
condition {\em (1)} is satisfied.

We now show that $\sigma$ is not
extendable by applying Algorithm \ref{algomatch} to the tableau
$T_4^\omega$. By Proposition \ref{val-match}, $T_4^\omega$ is geometric
if and only if it has a valid matching. In the figure below, the part
of $T_4^\omega$ which is shaded consists of the weakly 
adjacent pairs which get deleted during Algorithm 
\ref{algomatch}. Observe that after deleting these entries, the
remaining tableau has no weakly adjacent pairs. This proves that
$\omega$ is not extendable.

\medskip

\begin{center}
  \begin{tikzpicture}
    \begin{scope}[scale = .4]
      \fill[blue, opacity=.15] 
      (0,0) --++ (2,0) --++ (0,4) --++ (-2,0) --cycle;
      \fill[blue, opacity=.15] 
      (2,0) --++ (2,0) --++ (0,3) --++ (-2,0) --cycle;
      \fill[blue, opacity=.15] 
      (4,0) --++ (2,0) --++ (0,2) --++ (-2,0) --cycle;
      \fill[blue, opacity=.15] 
      (6,0) --++ (4,0) --++ (0,1) --++ (-4,0) --cycle;

      \draw (0,0) --++ (15,0) --++ (0,4) --++ (-15,0) --cycle;      
      \draw 
      (1,0) --++ (0,4) 
      (2,0) --++ (0,4) 
      (3,0) --++ (0,4) 
      (4,0) --++ (0,4) 
      (5,0) --++ (0,4) 
       (6,0) --++ (0,4) 
      (7,0) --++ (0,4) 
      (8,0) --++ (0,4) 
      (9,0) --++ (0,4) 
      (10,0) --++ (0,4) 
      (11,0) --++ (0,4) 
      (12,0) --++ (0,4) 
      (13,0) --++ (0,4) 
      (14,0) --++ (0,4) ;

      \draw (0,1) --++ (15,0) (0,2) --++ (15,0) (0,3) --++ (15,0);

      \node at (.5,.5) {\ft $2$};
      \node at (1.5,.5) {\ft $3$};
      \node at (2.5,.5) {\ft $4$};
      \node at (3.5,.5) {\ft $4$};
      \node at (4.5,.5) {\ft $3$};
      \node at (5.5,.5) {\ft $2$};
      \node at (6.5,.5) {\ft $2$};
      \node at (7.5,.5) {\ft $2$};
      \node at (8.5,.5) {\ft $2$};
      \node at (9.5,.5) {\ft $3$};
      \node at (10.5,.5) {\ft $3$};
      \node at (11.5,.5) {\ft $3$};
      \node at (12.5,.5) {\ft $4$};
      \node at (13.5,.5) {\ft $4$};
      \node at (14.5,.5) {\ft $4$};

      \node at (.5,1.5) {\ft $1$};
      \node at (1.5,1.5) {\ft $1$};
      \node at (2.5,1.5) {\ft $1$};
      \node at (3.5,1.5) {\ft $1$};
      \node at (4.5,1.5) {\ft $1$};
      \node at (5.5,1.5) {\ft $3$};
      \node at (6.5,1.5) {\ft $4$};
      \node at (7.5,1.5) {\ft $4$};
      \node at (8.5,1.5) {\ft $3$};
      \node at (9.5,1.5) {\ft $3$};
      \node at (10.5,1.5) {\ft $3$};
      \node at (11.5,1.5) {\ft $3$};
      \node at (12.5,1.5) {\ft $4$};
      \node at (13.5,1.5) {\ft $4$};
      \node at (14.5,1.5) {\ft $4$};
     
      \node at (.5,2.5) {\ft $1$};
      \node at (1.5,2.5) {\ft $1$};
      \node at (2.5,2.5) {\ft $1$};
      \node at (3.5,2.5) {\ft $2$};
      \node at (4.5,2.5) {\ft $2$};
      \node at (5.5,2.5) {\ft $2$};
      \node at (6.5,2.5) {\ft $2$};
      \node at (7.5,2.5) {\ft $1$};
      \node at (8.5,2.5) {\ft $1$};
      \node at (9.5,2.5) {\ft $2$};
      \node at (10.5,2.5) {\ft $4$};
      \node at (11.5,2.5) {\ft $4$};
      \node at (12.5,2.5) {\ft $4$};
      \node at (13.5,2.5) {\ft $4$};
      \node at (14.5,2.5) {\ft $4$};
     
      \node at (.5,3.5) {\ft $1$};
      \node at (1.5,3.5) {\ft $1$};
      \node at (2.5,3.5) {\ft $1$};
      \node at (3.5,3.5) {\ft $2$};
      \node at (4.5,3.5) {\ft $2$};
      \node at (5.5,3.5) {\ft $2$};
      \node at (6.5,3.5) {\ft $3$};
      \node at (7.5,3.5) {\ft $3$};
      \node at (8.5,3.5) {\ft $3$};
      \node at (9.5,3.5) {\ft $3$};
      \node at (10.5,3.5) {\ft $2$};
      \node at (11.5,3.5) {\ft $1$};
      \node at (12.5,3.5) {\ft $1$};
      \node at (13.5,3.5) {\ft $2$};
      \node at (14.5,3.5) {\ft $3$};
    \end{scope}
  \end{tikzpicture}
\end{center}
\end{example}

\bigskip

In Proposition \ref{kara}, it is assumed that $\sigma$ is the
signature of a system.  An arbitrary word $\sigma
\in \{x,y,z\}^\star$ must satisfy the parity conditions of Proposition
\ref{parity} in order to be a signature, and this imposes additional structure on the words in 
$W_{_{(ab,cd)}}$ associated to extendable signatures. This structure is described below in Lemma \ref{omega-form}. It will be crucial for the proof of Proposition \ref{kara} as well as for the
proofs of Theorems \ref{UL envelope} and \ref{envelope theorem}. First, some notions need
to be introduced. 

Recall the languages $B_{_{(ab)}}$ and
$B_{_{(cd)}}$ consisting of balanced blocks on letters $\{a,b\}$ and
$\{c,d\}$, respectively (defined in the paragraph preceding
Proposition \ref{bal-def}). Let $\overline{B}_{_{(ab)}}$ and
$\overline{B}_{_{(cd)}}$ denote the subsets consisting of {\em odd}
balanced blocks, meaning balanced blocks of the form $(a^kb^k)$,
$(b^ka^k)$ and $(c^kd^k)$, $(d^kc^k)$ where $k$ is an odd positive
integer. Let $U_{_{(ab,cd)}}$ denote the set of words, $\omega
= \omega_1\cdots \omega_k$, where each $\omega_i \in
\overline{B}_{_{(ab)}} \cup \overline{B}_{_{(cd)}}$. 

\begin{define} \label{well-balance}
A word $\omega\in W_{_{(ab,cd)}}$ is called \df{well-balanced} if
there exists words $\omega_1$ and $\omega_2$ $\in B_{_{(ab)}} \cup
B_{_{(cd)}}$ such that 
\begin{itemize}
\item $\omega = \omega_1 \cdot \omega' = \omega_2 \cdot \omega''$,
\item $\omega_1$ contains an {\em even} number of $a$'s and an {\em
    odd} number of $c$'s,
\item $\omega_2$ contains an {\em odd} number of $a$'s and an {\em
    even} number of $c$'s. 
\end{itemize}
\end{define}

\begin{lemma} \label{omega-form}
Let $\sigma$ be an irreducible signature which satisfies
conditions (1) and (2) of Proposition \ref{kara} where $\sigma$ can be
written as $\sigma = \alpha \cdot \beta \cdot \gamma$, and let $\omega\in
W_{_{(ab,cd)}}$ be the associated word.
\begin{itemize}
\item If $p>0$ and $q=0$, then $\omega$ is of the form
\[(a^pb^p) \cdot \delta\]
where $\delta = \delta_1\delta_2 \cdots$, and each $\delta_i \in
\overline{B}_{_{(cd)}}$  
\item If $p$ and $q$ are both non-zero, then $\omega$ is well-balanced
  and of the form \[\omega  = (a^pb^p) \cdot \delta \cdot (c^qd^q)\] 
where $\delta \in U_{_{(ab,cd)}}$. 
\end{itemize}
\end{lemma}

\begin{remark}
  Notice that up to reversal of $\sigma$ and interchanging $x$'s and
  $z$'s, Lemma \ref{omega-form} covers all 
  possible irreducible signatures satisfying the conditions of
  Proposition \ref{kara}. 
\end{remark}

\begin{proof}[Proof of Lemma \ref{omega-form}]
  We first consider the case when $p>0$ and $q=0$. This
  means that $\sigma$ is of the form
  \[\sigma = \alpha \cdot z^p,\]
  where $X(\alpha) = Y(\alpha) = p$ and $Z(\alpha) = 0$. It follows that 
  $\sigma_{_{\{y,z\}}} = y^pz^p$ and $\sigma_{_{\{x,z\}}} = x^pz^p$. By
condition {\em (1)} of Proposition \ref{kara}, $\sigma_{_{\{x,y\}}}$ factors into
balanced blocks. By Proposition \ref{parity}, there are {\em odd} positive
integers $p_i$ with $\sum p_i = p$ such that 
\[\sigma_{_{\{x,y\}}} = (x^{p_1}y^{p_1})(y^{p_2}x^{p_2})\cdots\]
where the blocks of $\sigma_{_{\{x,y\}}}$ appear in an alternating
pattern, meaning that we have $(x^{p_i}y^{p_i})$ for odd $i$, and
$(y^{p_i}x^{p_i})$ for even $i$. Therefore we have
\[\omega = (a^pb^p)(c^{p_1}d^{p_1})(d^{p_2}c^{p_2})\cdots\]
which is what we wanted to show. 

Now suppose both $p$ and $q$ are positive. The assumption that $\sigma =
\alpha \cdot \beta \cdot \gamma$ implies that $\sigma$ starts
with the letter $x$ and ends with the letter $z$. It follows that the
initial balanced block of $\sigma_{_{\{y,z\}}}$ equals $(y^pz^p)$ and
the last balanced block of $\sigma_{_{\{x,y\}}}$ equals
$(x^qy^q)$. This implies that $\omega$ is of the form 
\[\omega = (a^pb^p) \cdot \delta \cdot (c^qd^q)\]
where $\delta \in W_{_{(ab,cd)}}$. We still need to argue that
$\delta$ is comprised of odd blocks, that is, $\delta
\in U_{_{(ab,cd)}}$. By applying
Proposition \ref{parity}, as before, it follows that $\alpha$ can be
written as 
\[\alpha = (x^{p_1}y^{p_1})(y^{p_2}x^{p_2})\cdots\]
where the $p_i$ are odd positive integers with $\sum p_i = p$ and the
blocks appear in an 
alternating pattern. By symmetry, the same argument shows that
$\gamma$ can be written as 
\[\gamma  = \cdots (z^{q_2}y^{q_2})(y^{q_1}z^{q_1})\] where the $q_i$
are odd positive integers with $\sum q_i  = q$ and the blocks appear
in an alternating pattern. This implies that $\sigma_{_{\{x,y\}}}$ and
$\sigma_{_{\{y,z\}}}$ can be written as
\[\sigma_{_{\{x,y\}}} = (x^{p_1}y^{p_1})(y^{p_2}x^{p_2})\cdots
(x^qy^q) \; \mbox{ and } \; \sigma_{_{\{y,z\}}} = (y^pz^p) \cdots
(z^{q_2}y^{q_2})(y^{q_1}z^{q_1}).\] 
Since $\sigma_{_{\{x,y\}}} \sim \omega_{_{\{c,d\}}}$ and
$\sigma_{_{\{y,z\}}} \sim \omega_{_{\{a,b\}}}$, this proves that
$\delta\in U_{_{(ab,cd)}}$. 
It remains to prove that $\omega$ is well-balanced.

Consider the case when $p$ is odd. The block $(a^pb^p)\in
B_{_{(ab)}}$ is an initial substring of $\omega$ which contains an odd
number of $a$'s and an even number of $c$'s. Next, notice that the
last block of $\alpha$ must be $(x^{p_j}y^{p_j})$ for some odd number
$j$. Since $p_j$ is odd, 
Proposition \ref{parity} implies that the first occurrence of the
letter $x$ in $\beta$ must be preceded by an {\em odd} number of $z$'s.
Thus $\beta$ can be written as
\[\beta = z^mx\cdots\]
where $m$ is {\em odd} and strictly less than $p$, or else $\sigma$ is not
irreducible. Therefore the restriction $\sigma_{_{\{x,z\}}}$ can be
written as
\[\sigma_{_{\{x,z\}}} = x^pz^mx\cdots\] 
which implies that there is an initial substring of $\omega$, 
\[(a^pb^p)\cdot \delta_1 \cdot \delta_2 \cdots \delta_{2k-1}
\cdot \delta_{2k}\]
where $\delta_1, \dots, \delta_{2k-1} \in
\overline{B}_{_{(cd)}}$ and $\delta_{2k}\in \overline{B}_{_{(ab)}}$.
This initial substring has an even number of $a$'s and precisely $m$
$c$'s, and therefore $\omega$ is well-balanced.

Now consider the case when $p$ is even. Then the last block of
$\alpha$ must be of the form $(y^{p_i}x^{p_i})$. Since $p_i$ is odd,
Proposition \ref{parity} implies that the first occurrence of the
letter $z$ in $\beta$ must be preceded by an {\em even} number of
$x$'s. Thus $\beta$ can be written as \[\beta = x^mz\cdots\] where $m$ is
an {\em even} non-negative integer. We now distinguish the cases
whether $m=0$ or not. 

If $m$ is positive, then the restriction of
$\sigma_{_{\{x,z\}}}$ can be written as
\[\sigma_{_{\{x,z\}}} = x^{p+m}z\cdots\]
which implies that there is an initial substring of $\omega$, 
\[(a^pb^p)\cdot \delta_1\cdots \delta_{2k}\cdot \delta_{2k+1}\]
where $\delta_1,\dots, \delta_{2k}\in \overline{B}_{_{(ab)}}$ and
$\delta_{2k+1}\in \overline{B}_{_{(cd)}}$. This initial substring
contains an even number of $a$'s and an odd number of
$c$'s. Furthermore, the initial substring \[(a^pb^p)\cdot \delta_1\]
contains an odd number of $a$'s and an even number of $c$'s, and
therefore $\omega$ is well balanced.

Finally, suppose that $m=0$. Then Proposition \ref{parity} implies
that $\beta$ can be written as \[\beta = z^kx\cdots\]
where $k$ is a positive {\em even} integer which is strictly less than
$p$, or else $\sigma$ is not irreducible. Therefore the restriction
$\sigma_{_{\{x,z\}}}$ can be written as
\[\sigma_{_{\{x,z\}}} = x^pz^kx\cdots\]
which implies that there is an initial substring of $\omega$,
\[(a^pb^p)\cdot \delta_1 \cdots \delta_{2j} \cdot \delta_{2j+1}\]
where $\delta_1, \dots, \delta_{2j} \in
\overline{B}_{_{(cd)}}$, and $\delta_{2j+1}\in
\overline{B}_{_{(ab)}}$. This initial substring contains an even
number of $c$'s and an odd number of $a$'s. Furthermore, the initial
substring
\[(a^pb^p)\cdot \delta_1\] contains an even number of $a$'s and an odd
number of $c$'s, and therefore $\omega$ is well-balanced. \end{proof}

\begin{proof}[Proof of Proposition \ref{kara}]
We first show the sufficiency of the conditions of Proposition
\ref{kara}. By Lemma \ref{omega-form} we may assume that the
associated $\omega\in W_{_{(ab,cd)}}$ can be written as $\omega =
(a^pb^p) \cdot \delta \cdot (c^qd^q)$ where $\delta \in
U_{_{(ab,cd)}}$. Notice that there are $q$ $a$'s and $p$ $c$'s in
$\delta$, and that the first case of Lemma
\ref{omega-form} is obtained by setting $q = 0$. We will show that for any $n\geq
4$, the tableau $T_n^\omega$ has a valid matching. By Proposition
\ref{val-match} this will imply that $T_n^\omega$ is geometric and
therefore that $\sigma$ is extendable. 

The tableau $T_n^\omega$ comes with a very particular
structure. Consider its rows which are given by
 \[T_n^\omega(i) = \phi_{_{i,n}}(\omega) = \phi_{_{i,n}}(a^pb^p) \cdot
 \phi_{_{i,n}}(\delta) \cdot \phi_{_{i,n}}(c^qd^q).\] 
Notice that the term $\phi_{_{i,n}}(\delta)$ has length $(n-i)p +
(i-1)q$ with precisely $p$ entries of each number in $[n]_i^+$ and
precisely $q$ entries of each number in $[n]_i^-$. If we view the
tableau by its geometric representation we see that we can decompose
into parts as indicated in the figure below.

\medskip

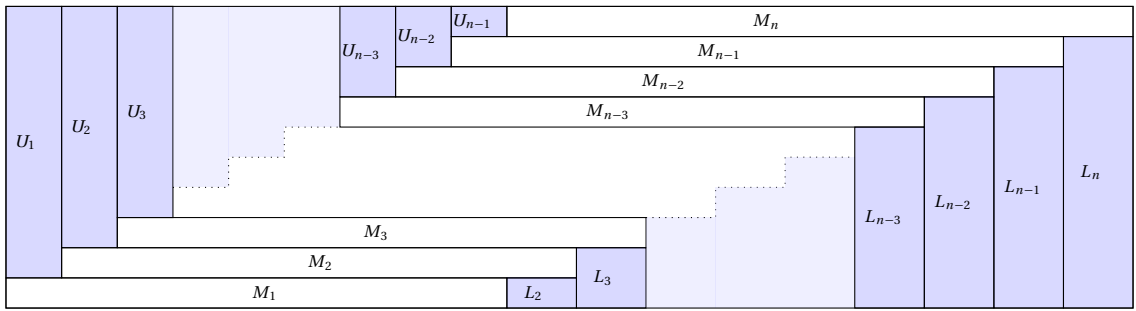
\begin{center}
\begin{tikzpicture}[xscale=.183 , yscale=.4]

\fill[blue, opacity = .15] 
(0,1) --++ (0,9) --++ (4,0) --++ (0,-9) -- cycle;
\draw (0,5.5) node[right]{\tiny $U_1$}  
(0,1) --++ (0,9) --++ (4,0) --++ (0,-9) -- cycle;

\fill[blue, opacity = .15] 
(4,2) --++ (0,8) --++ (4,0) --++ (0,-8) -- cycle;
\draw (4,6) node[right]{\tiny $U_2$}
(4,2) --++ (0,8) --++ (4,0) --++ (0,-8) -- cycle;

\fill[blue, opacity = .15] 
(8,3) --++ (0,7) --++ (4,0) --++ (0,-7) -- cycle;
\draw (8,6.5) node[right]{\tiny $U_3$}
(8,3) --++ (0,7) --++ (4,0) --++ (0,-7) -- cycle;

\fill[blue!20!white, opacity = .3] 
(12,4) --++ (0,6) --++ (4,0) --++ (0,-6) -- cycle;
\draw[dotted] (12,4) --++ (4,0) --++ (0,1) --++ (4,0) --++(0,1) --++ (4,0); 

\fill[blue!20!white, opacity = .3] 
(16,5) --++ (0,5) --++ (4,0) --++ (0,-5) -- cycle;
\draw[dotted] (16,5) ;

\fill[blue!20!white, opacity = .3] 
(20,6) --++ (0,4) --++ (4,0) --++ (0,-4) -- cycle;
\draw[dotted] (20,6) ;

\fill[blue, opacity = .15] 
(24,7) --++ (0,3) --++ (4,0) --++ (0,-3) -- cycle;
\draw (23.43,8.5) node[right]{\tiny $U_{n-3}$}
(24,7) --++ (0,3) --++ (4,0) --++ (0,-3) -- cycle;

\fill[blue, opacity = .15] 
(28,8) --++ (0,2) --++ (4,0) --++ (0,-2) -- cycle;
\draw (27.43,9) node[right]{\tiny $U_{n-2}$}
(28,8) --++ (0,2) --++ (4,0) --++ (0,-2) -- cycle;

\fill[blue, opacity = .15] 
(32,9) --++ (0,1) --++ (4,0) --++ (0,-1) -- cycle;
\draw (31.43,9.5) node[right]{\tiny $U_{n-1}$}
(32,9) --++ (0,1) --++ (4,0) --++ (0,-1) -- cycle;

\draw (17,0.5) node[right]{\tiny $M_{1}$}
(0,0) --++ (0,1) --++ (36,0) --++ (0,-1) -- cycle;

\draw (21,1.5) node[right]{\tiny $M_{2}$}
(4,1) --++ (0,1) --++ (37,0) --++ (0,-1) -- cycle;

\draw (25,2.5) node[right]{\tiny $M_{3}$}
(8,2) --++ (0,1) --++ (38,0) --++ (0,-1) -- cycle;

\draw (41,6.5) node[right]{\tiny $M_{n-3}$}
(24,6) --++ (0,1) --++ (42,0) --++ (0,-1) -- cycle;

\draw (45,7.5) node[right]{\tiny $M_{n-2}$}
(28,7) --++ (0,1) --++ (43,0) --++ (0,-1) -- cycle;

\draw (49,8.5) node[right]{\tiny $M_{n-1}$}
(32,8) --++ (0,1) --++ (44,0) --++ (0,-1) -- cycle;

\draw (53,9.5) node[right]{\tiny $M_{n}$}
(36,9) --++ (0,1) --++ (45,0) --++ (0,-1) -- cycle;

\fill[blue, opacity = .15] 
(36,0) --++ (0,1) --++ (5,0) --++ (0,-1) -- cycle;
\draw (36.5,0.5) node[right]{\tiny $L_{2}$}
(36,0) --++ (0,1) --++ (5,0) --++ (0,-1) -- cycle;

\fill[blue, opacity = .15] 
(41,0) --++ (0,2) --++ (5,0) --++ (0,-2) -- cycle;
\draw (41.5,1) node[right]{\tiny $L_{3}$}
(41,0) --++ (0,2) --++ (5,0) --++ (0,-2) -- cycle;

\fill[blue!20!white, opacity = .3] 
(46,0) --++ (0,3) --++ (5,0) --++ (0,-3) -- cycle;
\draw[dotted] (46,3) --++(5,0) --++ (0,1) --++ (5,0) --++ (0,1) --++ (5,0);

\fill[blue!20!white, opacity = .3] 
(51,0) --++ (0,4) --++ (5,0) --++ (0,-4) -- cycle;
\draw[dotted] (51,0) ;

\fill[blue!20!white, opacity = .3] 
(56,0) --++ (0,5) --++ (5,0) --++ (0,-5) -- cycle;
\draw[dotted](56,0) ;

\fill[blue, opacity = .15] 
(61,0) --++ (0,6) --++ (5,0) --++ (0,-6) -- cycle;
\draw (61,3) node[right]{\tiny $L_{n-3}$}
(61,0) --++ (0,6) --++ (5,0) --++ (0,-6) -- cycle;

\fill[blue, opacity = .15] 
(66,0) --++ (0,7) --++ (5,0) --++ (0,-7) -- cycle;
\draw (66,3.5) node[right]{\tiny $L_{n-2}$}
(66,0) --++ (0,7) --++ (5,0) --++ (0,-7) -- cycle;

\fill[blue, opacity = .15] 
(71,0) --++ (0,8) --++ (5,0) --++ (0,-8) -- cycle;
\draw (71,4) node[right]{\tiny $L_{n-1}$}
(71,0) --++ (0,8) --++ (5,0) --++ (0,-8) -- cycle;

\fill[blue, opacity = .15] 
(76,0) --++ (0,9) --++ (5,0) --++ (0,-9) -- cycle;
\draw (76.5,4.5) node[right]{\tiny $L_{n}$}
(76,0) --++ (0,9) --++ (5,0) --++ (0,-9) -- cycle;

\draw (0,0) --++ (0,10) --++ (81,0) --++ (0,-10) -- cycle ;

\end{tikzpicture}
\end{center}

For $1\leq i \leq n-1$, box $U_i$ has height $n-i$ and width $p$ and all
its entries equal $i$. For $1\leq i \leq n$, box $M_i$ is the part of
row $i$ and the entries are given by $\phi_{_{i,n}}(\delta)$. For $1<
i \leq n$, box $L_i$ has height $i-1$ and width $q$ and all its
entries equal $i$. 

It is now easy to see that $T^\omega_n$ has a valid matching. Each
entry in $M_1$ is weakly adjacent to a corresponding entry in
$U_1$, and therefore Algorithm \ref{algomatch} will delete all entries
in $U_1\cup M_1$. Once these are deleted, we see that
the each entry in $M_2$ is weakly adjacent to an entry in $U_2\cup
L_2$, and therefore Algorithm \ref{algomatch} will delete all entries
of $U_2\cup M_2 \cup L_2$. In general,
for $i<n$, each entry in $M_i$ will be weakly adjacent to an entry in
$U_i\cup L_i$, and Algorithm \ref{algomatch} will delete all entries
of $U_i\cup M_i\cup L_i$. Finally, each entry in $M_n$ will be weakly
adjacent to 
an entry in $L_n$. Notice that in the case when $q=0$, then $M_n$ and
every $L_i$ is empty, and each entry in $M_i$ will be weakly adjacent to a
member in $U_i$. This shows that $T^\omega_n$ has a valid matching.

\bigskip

We now establish the necessity of condition {\em (1)}. Suppose
$\sigma$ is an extendable signature. Let $S$ be a regular system of
size $4$ with signature $\sigma$, and let $T$ be the associated
tableau. By Proposition \ref{system-tableau}, $T$ is regular, and
by Proposition \ref{tab-corr} there exists an 
$\omega \in W_{_{(ab,cd)}}$ such that $T = T^{\omega}_{4}$. Clearly
$T^{\omega}_{3}$ is the tableau associated to a system with signature
$\sigma$. 

We are ready to show the necessity of condition {\em (2)}. Suppose
$\sigma$ is irreducible and extendable. Therefore condition {\em (1)} holds, and $\sigma$ satisfies the refinement conditions stated in Remark \ref{condition1}.
The basic strategy is to factor $\sigma$ into maximal parts
each consisting of only two of the three letters, and use the refinement
conditions to control the distribution of the letters. Throughout we also
use the parity conditions from Proposition \ref{parity}. Let 
\[ \begin{array}{ccc}
\langle \sigma_{_{\{x,y\}}} \rangle = (a_1, a_2, \dots) & , & \langle
\sigma_{_{\{y,z\}}} 
\rangle =  (c_1, c_2, \dots)\\
\exp_z(\sigma_{_{\{x,z\}}}) = (m_1, m_2, \dots) & , &
\exp_x(\sigma_{_{\{x,z\}}}) = (n_1, 
n_2, \dots)\\  
\end{array}\]
and write $\sigma_{_{\{x,y\}}}$ and $\sigma_{_{\{y,z\}}}$ as
\[\sigma_{_{\{x,y\}}} = \alpha_1 \alpha_2 \cdots  \;\; , \;\; \sigma_{_{\{y,z\}}} =
\gamma_1 \gamma_2 \cdots \]
where $\alpha_i \in B_{_{(xy)}}$ and $\gamma_i\in B_{_{(yz)}}$.
Up to symmetry we may assume that $\sigma$ starts with an $x$, and
therefore $\alpha_1 = (x^{a_1}y^{a_1})$. Let $\alpha$ be the set of blocks of
$\sigma_{_{\{x,y\}}}$ which appear in $\sigma$ before the first $z$. We split
into cases depending on whether $\alpha$ is empty or not. 

{\em Case: $\alpha = \emptyset$.} Note that this implies $0< n_1 \leq
a_1$, and we claim that $\sigma =  x^q \cdot \sigma_{_{\{y,z\}}}$. If $n_1 <
a_1$, then $\sigma$ can be written as 
\[\sigma  = x^{n_1}z^{m_1}x\cdots\]
and consequently $n_1 < a_1 \leq m_1 \leq c_1$, which contradicts the
assumption that $\langle \sigma_{_{\{y,z\}}} \rangle$ is a refinement of
$\exp_x(\sigma_{_{\{x,z\}}})$. 
Therefore $n_1 = a_1
\leq m_1$, and since $c_1 + c_2 + \cdots + c_j = n_1$ for some $j \geq
1$, we conclude that $\sigma$ can be written as \[\sigma = x^{n_1}
\cdot (\gamma_1  \gamma_2 \cdots \gamma_j)\] 

{\em Case: $\alpha \neq \emptyset$.} Let $\alpha = \alpha_1 \alpha_2
\cdots \alpha_j$. We
note $a_i$ is odd for all $1\leq i \leq j$, that $\alpha_i =
(x^{a_i}y^{a_i})$ for odd $i \leq j$, and $\alpha_i = 
(y^{a_i}x^{a_1})$ for even $i \leq j$. We now split further into
subcases depending on whether $j$ is even or odd.

{\em Subcase: $j$ odd.} We claim that $\alpha_j$ is directly followed by
$z$. If not, then $\alpha_{j+1} = (y^{a_{j+1}}x^{a_{j+1}})$ and the first
$z$ occurs within this block, but this implies $n_1 < c_1$,
contradicting the assumption that $\langle \sigma_{_{\{y,z\}}}
\rangle$ is a refinement of $\exp_x(\sigma_{_{\{x,z\}}})$. Therefore $c_1 = n_1$
and $\sigma$ 
can be written as \[\sigma  = \alpha_1  \alpha_2 \cdots \alpha_j z^c
\cdots\] If $c$ is 
even, then $z^c$ is followed by $y$, implying that 
$c > n_1$ and therefore $\sigma$ can be written as 
\[\sigma = \alpha_1 \alpha_2 \cdots \alpha_j  z^{n_1}  z \cdots \]
contradicting the assumption that $\sigma$ is
irreducible. Consequently, $c = m_1$ is odd and $m_1 \leq n_1$. If
$m_1 = n_1$ then $\sigma$ can be written as
\[\sigma = \alpha_1 \alpha_2 \cdots \alpha_j z^{n_1}\]
and we are done. Therefore assume that $m_1 < n_1$ which implies
$\alpha_{j+1} = (x^{a_{j+1}}y^{a_{j+1}})$. For some $i>1$ we have $m_1 
+ m_2 + \cdots + m_i = n_1$, where $n_1$ and $m_1$
are odd, $m_2, \dots, m_{i}$ are even. Therefore, $a_{j+1} = n_2 +
n_3 + \cdots + n_{i+1} \leq m_{i+1}$ 
and there is a $k\geq i+1$ for which the sequence $(c_2, c_3,
\dots, c_k)$ is a refinement of 
$(n_2, n_3, \dots, n_{i+1})$. It then follows that $\sigma$ can be written
as
\[\sigma = (\alpha_1 \alpha_2 \cdots \alpha_j) (z^{m_1}  x^{n_2} 
z^{m_2} x^{n_3} \cdots z^{m_i} x^{n_{i+1}}) (\gamma_2 \gamma_3
\cdots \gamma_k)\]

{\em Subcase: $j$ even.} Similar to above, we exclude the case that
$\sigma = \alpha_1 \alpha_2 \cdots \alpha_j z^{n_1}$ and we are left
with the case that $\alpha_{j+1} = (x^{a_{j+1}}y^{a_{}j+1})$ and we
must consider two further subcases depending on whether the initial
$z$ is preceded by $x$ or $y$.

{\em Subsubcase: $z$ preceded by $x$.} This is similar to the case
when $j$ is odd. For some $i\geq 1$, we have $m_1+m_2 +
\cdots + m_i = n_1$, where each $z^{m_s}$ appears before the initial
$y$ in $\alpha_{j+1}$. This induces a partition of the string of $x$'s
in $\alpha_{j+1}$ giving us the exponent sequence $(n_2, n_3, \dots,
n_{i'})$ where $i\leq i' \leq i+1$ and $a_{j+1} = n_2 + n_3 + \cdots +
n_{i'} \leq m_{i+1}$. Moreover there is a $k\geq i'$ such that $(c_2,
c_3, \dots, c_k)$ is a refinement of $(n_2, n_3, \dots, n_{i'})$ which
implies that we can write $\sigma$ as \[\sigma = (\alpha_1 \alpha_2
\cdots \alpha_j) (x^dz^{m_1}x^{n_2}z^{m_2} \cdots x^{n_i}z^{m_i}x^{e})
(\gamma_2 \gamma_3 \cdots \gamma_k)\] where $d\geq 0$ is even, and
$e=0$ with $i'=i$ or $0< e = n_{i'}$ with $i'=i+1$.

{\em Subsubcase: $z$ preceded by $y$.} This case is 
different than the preceding ones, and cannot be dealt by only using
the refinement conditions. For instance, the signature in Example
\ref{special} belongs to this case.
We may assume $\alpha_{j+1} = (x^{a_{j+1}}y^{a_{j+1}})$ with $a_{j+1}$ odd
and $n_1 = a_1 + a_2 + \cdots + a_{j+1}$. We can write 
$\sigma$ as \[\sigma = (\alpha_1 \alpha_2 \cdots \alpha_j) \cdot
(x^{a_{j+1}} y^b z^c y \cdots)\]
where $b$ is odd with $0<b<a_{j+1}$ and $c>0$ is even. Set $a = a_1 +
\cdots +a_j$. 
Now we have $c_1= a + b < c < n_1 = a + a_{j+1}$, 
\[ \begin{array}{rclcrcl}
\langle \sigma_{_{\{x,y\}}} \rangle & = & (a_1, a_2, \dots) & , & 
\langle \sigma_{_{\{y,z\}}}\rangle & = & (c_1, c - c_1, \dots)\\
\exp_z(\sigma_{_{\{x,z\}}}) & = & (c_1+ d, m_2, \dots) & , &
\exp_x(\sigma_y) & = & (a+a_{j+1}, 
n_2, \dots).\\  
\end{array}\]
From this we can get the following partial structure on $\omega$.
\[\omega = (a^{c_1}b^{c_1})(b^{c-c_1}a^{c-c_1}) \cdots
(c^{a_1}d^{a_1})(d^{a_2}c^{a_2})\cdots (d^{a_j}c^{a_j})
(c^{a_{j+1}}d^{a_{j+1}}) \cdots.\] 
This partial information is sufficient to
show that the tableau $T_4^{\omega}$ does not have a valid
matching. Since $T_4^\omega(i) = \phi_{_{i,4}}(\omega)$, we get

\[
\begin{array}{lcl@{\vspace{1.2ex}}}
\text{ $T_4^{\omega}(4)$} & = &
\underbrace{1\dots \dots \dots 1}_{c_1}\ \underbrace{2 \dots \dots \dots 2}_{c_1}\
\underbrace{3 \dots \dots \dots 3}_{c_1}\ \cdots  \\
\text{ $T_4^\omega(3)$ } & = &
\underbrace{1\dots \dots \dots 1}_{c_1}\ \underbrace{2 \dots \dots
  \dots 2}_{c_1}\ \cdots \\
\text{ $T_4^\omega(2)$} & = &
\underbrace{1\dots \dots \dots \dots 1}_{n_1}\ \underbrace{3 \dots 3}_{a_1}\
\underbrace{4 \dots 4}_{a_1}\ \cdots \\
\text{ $T_4^\omega(1)$} & = &
 \underbrace{2 \dots 2}_{a_1}\ \underbrace{3 \dots 3}_{a_1}\
\underbrace{4\dots 4}_{a_1}\
 \underbrace{4 \dots 4}_{a_2}\ \underbrace{3 \dots 3}_{a_2}\
\underbrace{2\dots 2}_{a_2}\ \cdots\cdots\
\underbrace{4 \dots 4}_{a_{j}}\ \underbrace{3 \dots 3}_{a_{j}}\
\underbrace{2\dots 2}_{a_{j}}\ 
\underbrace{2 \dots 2}_{a_{j+1}}\ \underbrace{3 \dots 3}_{a_{j+1}}\
\underbrace{4\dots 4}_{a_{j+1}}\ \cdots .
\end{array}
\]

We now apply Algorithm \ref{algomatch} to $T_4^\omega$ and identify
the positions of the weakly adjacent pairs that get deleted. Recall
that the algorithm always deletes leftmost entries.
\begin{itemize}
\item The first $3a$ steps deletes $3a$ entries from $T_4^\omega(1)$
together with $a$ entries from $T_4^\omega(i)$ for $i = 1,2,3$. 
\item The next $a_{j+1}$ steps
deletes $a_{j+1}$ entries from $T_4^\omega(1)$ together $a_{j+1}$
entries from $T_4^\omega(2)$. 
\item The next 
$b$ steps deletes $b$ entries from $T_4^\omega(1)$
 together with $b$ entries from $T_3^\omega(3)$.
\item The next $a_1$ steps deletes $a_1$ entries from $T_4^\omega(2)$
  together with $a_1$ entries from $T_4^\omega(3)$.
\end{itemize}

At this point there is no weakly adjacent pairs, so Algorithm
\ref{algomatch} will output a non-empty tableau, and consequently
$T_4^\omega$ is not geometric and $\sigma$ is not extendable.
This completes the proof of the necessity of condition {\em (2)}.
\end{proof}

\section{Proof of Theorems \ref{UL envelope} and 
\ref{envelope theorem}} \label{evlps}

Let $\sigma$ be an extendable irreducible signature with $X(\sigma) =
Y(\sigma) = Z(\sigma) = r = p+q$ where $p$ and $q$ are the numbers
from Proposition \ref{kara}. We distinguish
the following types:

\begin{itemize}
\item $\sigma$ is called \df{odd} if $r$ is odd and $p=0$ or $q=0$.
\item $\sigma$ is called \df{even} if $r$ is even and $p=0$ or $q=0$.
\item $\sigma$ is called \df{mixed} if $p>0$ and $q>0$.
\end{itemize}

\begin{example}
  
Consider the signature $\sigma = xyz$. This is an irreducible
extendable signature of the form $\sigma = \alpha \cdot \beta \cdot
\gamma$, where $\alpha = xy$, $\beta = z$, and $\gamma =
\emptyset$. Thus $p = 1$ and $q=0$. Note that this is not unique,
since it also satisfies $p=0$ and 
$q=1$, where $\beta = x$ and $\gamma = yz$. In either case, $\sigma$
is an {\em odd} signature and the associated word is $\omega =
(ab)(cd)$. For $n=4$ we get the following regular system. 

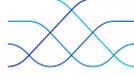
\begin{center}
\begin{tikzpicture}[scale=.3]
\begin{scope}[xscale=.93]
\draw[blue!60!black!30!cyan](0,4)
\ls\ls\ds\ld\ld\de; 
\draw[blue!70!black!50!cyan](0,3)
\ls\ds\ld\de\us\ue;
\draw[blue!80!black!70!cyan](0,2)
\ds\de\us\lu\ue\ls;
\draw[blue!90!black!90!cyan](0,1)
\us\lu\lu\ue\ls\ls;
\end{scope}
\end{tikzpicture}
\end{center}

Notice that the system is lower convex and that the upper envelope
contains only paths $1$ and $4$.

\bigskip

Now consider the signature $\sigma = xy^2xz^2$. This is an irreducible
extendable signature of the form $\sigma = \alpha \cdot \beta \cdot
\gamma$, where $\alpha = (xy)(yx)$, $\beta = z^2$, and $\gamma =
\emptyset$. Thus $p = 2$ and $q=0$, and $\sigma$
is an {\em even} signature with associated word $\omega =
(a^2b^2)(cd)(dc)$. For $n=4$ we get the following regular system. 

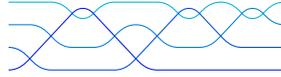
\begin{center}
\begin{tikzpicture}[scale=.3]
\begin{scope}[xscale=-.93, yscale = -1]
\draw[blue!60!black!30!cyan](0,1)
\us\ud\de\us\ud\de\ls\ls\us\ud\de\ls\ls; 
\draw[blue!70!black!50!cyan](0,2)
\ds\du\lu\ue\ls\ds\de\us\ue\ls\ds\de\ls;
\draw[blue!80!black!70!cyan](0,3)
\ls\ls\ds\ld\du\lu\lu\ue\ls\ls\ls\ds\de;
\draw[blue!90!black!90!cyan](0,4)
\ls\ls\ls\ls\ls\ls\ds\ld\ld\du\lu\lu\ue;
\end{scope}
\end{tikzpicture}
\end{center}

Notice that the system is upper convex and that the lower envelope
contains only paths $1$ and $2$.

\bigskip

Finally consider the signature $\sigma = xy^2x^3z^3y^2z$. This is an irreducible
extendable signature of the form $\sigma = \alpha \cdot \beta \cdot
\gamma$, where $\alpha = (xy)(yx)$, $\beta = x^2z^2$, and $\gamma =
(zy)(yz)$. Thus $p = 2$ and $q=2$, and $\sigma$
is a {\em mixed} signature with associated word $\omega =
(a^2b^2)(ba)(ab)(cd)(dc)(c^2d^2)$. For $n=4$ we get the following
regular system. 

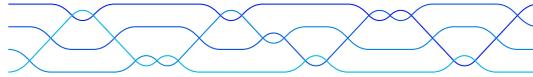
\begin{center}
\begin{tikzpicture}[scale=.3]
\begin{scope}[xscale=.93]
\draw[blue!90!black!90!cyan](0,4)
\ls\ls\ds\du\ue\ls\ls\ls\ls\ds\du\ue\ls\ls\ls\ls\ds\du\ud\ld\ld\du\lu\lu\ue;
\draw[blue!80!black!70!cyan](0,3)
\ls\ds\de\ls\us\ue\ls\ls\ds\de\ls\us\ud\ld\du\lu\lu\ud\du\ue\ls\ls\ls\ds\de;
\draw[blue!70!black!50!cyan](0,2)
\ds\de\ls\ls\ls\us\ud\du\lu\lu\ud\ld\du\ue\ls\ds\de\ls\ls\us\ue\ls\ds\de\ls;
\draw[blue!60!black!30!cyan](0,1)
\us\lu\lu\ud\ld\ld\du\ud\de\ls\ls\ls\ls\us\ud\de\ls\ls\ls\ls\us\ud\de\ls\ls; 
\end{scope}
\end{tikzpicture}
\end{center}

Notice that this system is upper convex and lower convex.
\end{example}

\bigskip

\begin{lemma}\label{parity-cases}
  Let $S$ be a regular system labeled by $[n]$ with irreducible
  signature $\sigma$.
  \begin{enumerate}
  \item If $\sigma$ is odd, then one of the envelopes of $S$ contains
    only paths $1$ and $n$ while the other envelope contains every
    path of $S$.
  \item If $\sigma$ is even, then either $S$ is lower convex and the
    upper envelope of $S$ contains only paths $n-1$ and $n$, or $S$ is upper
    convex and the lower envelope of $S$ contains
    only paths $1$ and $2$. 
  \item If $\sigma$ is mixed, then $S$ is upper convex and lower convex.
  \end{enumerate}
\end{lemma}

\begin{proof}
Let $\omega \in W_{_{(ab,cd)}}$ be the word associated to $\sigma$ and
let $T=T^{\omega}_n$ be the tableau associated to $S$. The local
sequence of path $i$ is given by $T(i) = \phi_{_{i,n}}(\omega)$. By using
Lemma \ref{omega-form} we can obtain the structure of the local
sequences of $S$ and apply Lemma \ref{local-envelope} to determine the
envelopes of $S$. 

Suppose $\sigma$ is even or odd. Up to symmetry we may assume $\sigma$ is of
the form \[\sigma = \alpha \cdot z^p\] where $p$ is a positive
integer. By Lemma 
\ref{omega-form}, $\omega \in W_{_{(ab,cd)}}$ has
the form \[\omega = (a^pb^p)\cdot \delta_1 \cdot \delta_2 \cdots\] where
each $\delta_i \in \overline{B}_{_{(cd)}}$. Obviously the lower
envelope of $S$ contains 
path $1$ and the upper envelope of $S$ contains path $n$.
We claim
that the upper envelope also contains path 1. To see this, consider
the local sequence of path 1 which is given by \[T(1) = 
\phi_{_{1,n}}(\omega) = \phi_{_{1,n}}(\delta_1)
\phi_{_{1,n}}(\delta_2)\cdots\] by Proposition \ref{kara}.   
Since $\phi_{_{1,n}}(\delta_1)$ is an initial string of $T(1)$ which
contains each number in $\{2,\dots,n\}$ an odd number of times, Lemma
\ref{local-envelope} implies that the upper envelope of $S$ contains
path $1$. 

For every $1<i \leq n$, the local sequence of path $i$ is given by
\[T(i) = \phi_{_{i,n}}(\omega) = \phi_{_{i,n}}(a^pb^p)
\phi_{_{i,n}}(\delta_1) \phi_{_{i,n}}(\delta_2)\cdots.\]

If $p$ is odd, then $\phi_{i,n}(a^pb^p)$ is an initial string of
$T(i)$ which contains every number in $[n]^-_i$ an odd
number of times, and Lemma \ref{local-envelope} implies that the lower
envelope of $S$ contains path $i$. Furthermore, for $1<i<n$, we see
from the initial string $\phi_{_{i,n}}(a^pb^p)$ that any
initial string of $T(i)$ which contains each number in $[n]_i^+$ an
odd number of times also contains each number in $[n]_i^-$ an odd
number of times. Therefore the upper envelope of $S$ does not
contain path $i$, by Lemma \ref{local-envelope}. Thus $S$ is lower
convex while the upper envelope of $S$ contains only paths $1$ and
$n$. 

If $p$ is even, then $\phi_{_{i,n}}(a^pb^p)\phi_{_{i,n}}(\delta_1)$ is
an initial string of $T(i)$ which contains each number in $[n]_i^-$ an
even number of times and each number in $[n]_i^+$ an 
odd number of times. By Lemma \ref{local-envelope}, the upper envelope
of $S$ contains path $i$. Furthermore, for $2<i<n$ we see from the
initial string $\phi_{_{i,n}}(a^pb^p)$ that there is no initial string
of $T(i)$ which contains each number in $[n]_i^+$ and odd
number of times. Therefore Lemma \ref{local-envelope} implies that the
lower envelope of $S$ does not contain path $i$. Finally, if $i=2$,
then the initial term of $T(2)$ is $1$, and therefore the lower
envelope of $S$ contains path $2$. Thus $S$ is upper convex while the
lower envelope contains only paths $1$ and $2$. 

Up to symmetry this proves {\em (1)} and {\em (2)}. It remains
to prove {\em (3)}. 

By Lemma \ref{omega-form} we may assume that $\omega$ is of the
form
\[\omega = (a^pb^p)\cdot \delta \cdot (c^qd^q)\]
where $\delta \in U_{_{(ab,cd)}}$. Obviously the lower envelope of $S$
contains path $1$ and the upper envelope of $S$ contains path $n$,
and we want to show that $S$ is upper and lower convex. For this we
use the fact that $\omega$ is well-balanced. For $1\leq i
\leq n$ the local sequence of path $i$ is given by 
$T(i) = \phi_{_{i,n}}(\omega)$. 

Let $\omega_1\in B_{_{(ab)}} \cup B_{_{(cd)}}$ be an initial string
of $\omega$ which consists of an even number of $a$'s and an odd
number of $c$'s. If $1\leq i<n$, then $\phi_{_{i,n}}(\omega_1)$ is an
initial string of $T(i)$ which contains each number in $[n]_i^-$ an
even number of times and each number in $[n]_i^+$ an odd
number of times. By Lemma \ref{local-envelope}, the upper envelope of
$S$ contains path $i$. 
Hence $S$ is upper convex.

Let $\omega_2 \in B_{_{(ab)}} \cup B_{_{(cd)}}$ be an initial string
of $\omega$ which consists of an odd number of $a$'s and an even
number of $c$'s. If $1 < i \leq n$, then $\phi_{_{i,n}}(\omega_2)$ is an
initial string of $T(i)$ which contains each number in $[n]_i^-$ an
odd number of times and each number in $[n]_i^+$ an even 
number of times. By Lemma \ref{local-envelope}, the lower envelope of
$S$ contains path $i$. Hence $S$ is lower convex.
\end{proof}

\begin{proof}[Proof of Theorem \ref{UL envelope}]
  Let $S$ be a regular system of size $n$ with signature $\sigma =
  \sigma_1 \cdot \sigma'$ where $\sigma_1$ is irreducible. By Lemma
  \ref{factorize} there is a regular system $S_1$ of size $n$ with signature
  $\sigma_1$, and the paths which appear on the envelopes of $S_1$
  also appear on the corresponding envelopes of $S$. The claim that $S$
  is upper or lower convex follows from  Lemma \ref{parity-cases}. 
\end{proof}

\begin{proof}[Proof of Theorem \ref{envelope theorem}]
  Suppose $S$ is regular of size $n$ and that
  $S$ is not upper convex. Let
  $\sigma = \sigma_1 \cdots 
  \sigma_m$ be the signature of $S$ where each $\sigma_i$ is
  irreducible. It follows from  Proposition 
  \ref{factorize} that the upper envelope of $S$ can be determined
  from the envelopes corresponding to each individual $\sigma_i$,
  which in turn is determined by Lemma \ref{parity-cases}. This
  immediately implies that none of the $\sigma_i$ are mixed. Therefore
  we may assume that $\sigma$ is the concatenation of odd and even
  signatures. 

  Next, observe that if every $\sigma_i$ is odd, or every
  $\sigma_i$ is even, then the upper envelope of $S$ contains at most two
  distinct paths. If every $\sigma_i$ is odd, then only paths $1$ and $n$ are
  contained in the upper envelope of $S$ (by {\em (1)} of Lemma
  \ref{parity-cases}). On the other hand, if every $\sigma_i$ is even,
  then only paths $n-1$ 
  and $n$ are contained in the upper envelope of $S$ (by {\em (2)}
  of Lemma \ref{parity-cases}). This proves {\em (1)} of the
  theorem, for if $S$ is $2$-crossing, then $\sigma$ cannot contain
  both an even and an odd signature. 
  
  The argument above tells us that the only way we can have more than
  two, but not all paths in the upper envelope of $S$ is when $\sigma$ is the
  concatenation of both odd and even signatures. Suppose that this is
  the case.

  Suppose that $\sigma_i$ is the only even signature. In this case the upper
  envelope of $S$ will contain three paths. To see this note that
  {\em (2)} of Lemma \ref{parity-cases} implies that an even
  signature contributes only the two topmost paths to the upper envelope of
  $S$. These may be either paths $1$ and $2$, or paths $n-1$ and $n$
  depending on the number of odd signatures which precede $\sigma_i$
  in $\sigma$. The odd signatures, on the other hand, will always
  contribute only the two extreme paths, that is, paths $1$ and
  $n$. Therefore, if only one of the $\sigma_i$ is even while all the
  other are odd, then the upper envelope of $S$ contains either paths
  $1$, $2$, and $n$, or paths $1$, $n-1$, and $n$; see for instance
  Example \ref{ex:concat}. This proves
  {\em (2)} of the theorem, for if $S$ is $4$-crossing and $\sigma$ is
  comprised of both odd and even signatures, then $\sigma$ can contain
  at most one even signature. 

  Finally, suppose that there are at least two even signatures,
  $\sigma_i$ and $\sigma_j$. As before, $\sigma_i$ and $\sigma_j$
  contribute only the two topmost paths to the upper envelope of
  $S$. But if $\sigma_i$ and $\sigma_j$ are separated in $\sigma$ by
  an odd number of odd signatures, then the order of the paths will
  switch between them. Therefore one of them will contribute paths $1$
  and $2$, while the other will contribute paths $n-1$ and $n$ to the
  upper envelope of $S$. The odd signatures still only contribute
  paths $1$ and $n$. Therefore the upper envelope of $S$ may contain
  paths $1$, $2$, $n-1$ and $n$, but no more. This proves {\em
    (3)} of the theorem.   
\end{proof}

\section{Remarks and open problems} \label{remark4}

\subsection{} As we mentioned in the introduction there exists arbitrarily
  large families of convex bodies such that any two members have six
  common tangents, any four members are in convex position, but no
  five members are in convex position. To construct such a family we
  apply the same technique as was used in \cite{DHH}. It relies on the
  fact that if $f$ is a $2\pi$-periodic real $C^2$-smooth 
  function such that $f(t)+f''(t)>0$ holds for all $t$, then $f$ is
  the support function of a convex body. (See for instance Lemma 2.2.3
  in \cite{groemer}.) Hence, for any $2\pi$-periodic real $C^2$-smooth
  function $f$, there exists a constant $c_0$ such that $f + c$ is the
  support function of a body for all $c > c_0$. 

  Now consider any system where each pair of paths cross an {\em even}
  number of times. By identifying the endpoints, the paths can be
  represented by $2\pi$-periodic functions which furthermore can be
  approximated by $C^2$-smooth functions. Consequently, there is a
  common constant we can add to each of the (smoothened) functions
  which makes the system the dual of a family of convex bodies. As in
  the proof of Theorem \ref{main Erd Sze}, intersections between
  support curves are in one-to-one correspondence with common
  tangents, and subfamilies in convex position correspond to
  subsystems which are upper convex. Thus the construction we are
  looking for can be obtained from any system where each pair of paths
  cross six times, each subsystem of size four is upper convex, and no
  subsystem of size five is upper convex. A regular system with
  signature $\sigma = xyzxy^2xz^3yxzy^2zx^2$ would be such an example.

\subsection{} We noted in the introduction that Pach and T\'{o}th constructed
  an arbitrarily large family of {\em segments} where every triple is in
  convex position, but no four member are. When we dualize this
  family via the support function it is easily seen that we obtain a
  regular system. As explained in the previous remark, every regular
  system can be realized by a family of convex bodies. A natural
  conjecture is: Every regular $2k$-crossing system be obtained as
  the dual of a family of convex $n$-gons ? (Here $n$ should depend
  only on $k$, but not on the number of paths in the system.)

\subsection{} Determining the quantitative behavior for $h_k(n)$ in Theorem
  \ref{main Erd Sze} remains as one of the main open questions. Our
 bound is certainly not close to the truth, and it should be remarked
 that the condition that any $m_k$ members are in convex position
 changes the nature of the problem significantly from the original
 Erd\H{o}s--Szekeres problem for points when $m_k>3$: If every four points are
 convex position, then the whole set is in convex position. For the
 case of families of {\em pairwise disjoint} convex bodies, Bisztriczky and
 Fejes T\'{o}th \cite{bisz-k} have showed that if every five members
 are in convex position, then the corresponding Erd\H{o}s--Szekeres
 function is polynomial in $n$, and T\'{o}th \cite{toth-k} showed that
 this function is in fact linear. 

\subsection{} It would be interesting to find sharper bounds for the function  
  $S_k(n)$ given in Theorem \ref{general cupscaps}. The case $k=1$ is
  dual to the classical {\em cups-caps} theorem of Erd\H{o}s-Szekeres
  \cite{erd-sze1, erd-sze2} and it is known that $S_1(n) =
  \binom{2n-4}{n-2}+1$. We do not know of any better lower bound than
  this for $k>1$. In fact, it is unclear whether there is actually a
  dependency on the number of crossings. This leads us to conjecture
  the following generalization of the {\em cups-caps} theorem: For
  every $n\geq 3$ there is a minimal positive integer $S(n)$ such that any
  system of paths of size $S(n)$ where each pair of paths cross at
  least once contains a subsystem of size $n$ which is upper or lower convex.

\section{Acknowledgments}
The authors express sincere gratitude to an anonymous referee for the meticulous comments which greatly improved the exposition of our results. 

M.~G.~Dobbins was supported by NRF grant 2011-0030044 funded by the
government of South Korea (SRC-GAIA) and BK21.  

A.~F.~Holmsen was supported  by Basic Science Research Program through the
National Research Foundation of Korea funded by the Ministry of
Education, Science and Technology (NRF-2010-0021048). 

A.~Hubard was supported
by Fondation Sciences Math\'{e}matiques de Paris. A.~Hubard would like to thank
KAIST for their hospitality and support during his visit.

\end{document}